\documentclass[nospthms]{svjour3}
\usepackage[margin=1in]{geometry}
\usepackage{amsmath,amsthm,amssymb}
\usepackage{stmaryrd}
\usepackage{mathrsfs}
\usepackage{enumitem}
\usepackage[all,2cell]{xy}
\usepackage{pb-diagram}
\usepackage{pb-xy}

\usepackage{tikz}
\usepackage{tikz-cd}
\usetikzlibrary{patterns}
\usetikzlibrary{positioning}
\usetikzlibrary{shapes,decorations.markings,arrows.meta}
\usetikzlibrary{matrix,arrows}
\tikzset{hfit/.style={rounded rectangle, inner xsep=0pt},
           vfit/.style={rounded corners}}
           
\usepackage{blkarray}
\usepackage{algpseudocode}
\usepackage[all]{xy}

\usepackage[pdftex]{hyperref}
 \hypersetup{
    colorlinks,%
    citecolor=red,%
    filecolor=black,%
    linkcolor=blue,%
    urlcolor=blue     
 }
\newtheorem{thm}{Theorem}[section]
\newtheorem{lem}[thm]{Lemma}
\newtheorem{defn}[thm]{Definition}
\newtheorem{cor}[thm]{Corollary}
\newtheorem{prop}[thm]{Proposition}
\newtheorem{ex}[thm]{Example}

\newtheorem{alg}[thm]{Algorithm}

\theoremstyle{remark}
\newtheorem{rem}[thm]{Remark}

\algdef{SE}[DOWHILE]{Do}{doWhile}{\algorithmicdo}[1]{\algorithmicwhile\ #1}%

\newcommand{\setof}[1]{\left\{ {#1}\right\}}

\newcommand{\bRMod}{\text{{\bf R-Mod}} }

\newcommand{\cyclic}{strict}
\newcommand{\Cyclic}{Strict}
\newcommand{\cycsub}{s}
\newcommand{\trivial}{minimal}
\newcommand{\Trivial}{Minimal}
\newcommand{\strict}{minimal}
\newcommand{\Strict}{Minimal}

\newcommand{\Hom}{\mathrm{Hom}}

\newcommand{\bVec}{\text{{\bf Vect}}}
\newcommand{\bCh}{\text{{\bf Ch}}}

\newcommand{\bChB}{\text{\bf ChB}(\sP,\bVec)}
\newcommand{\bKB}{\text{\bf KB}(\sP,\bVec)}
\newcommand{\bGMB}{\text{\bf GMB}(\sP,\bVec^\Z)}

\newcommand{\bGVec}{\text{{\bf GrVect}}}
\newcommand{\bChG}{\bCh(\bGVec(\sP))}
\newcommand{\bChGz}{\bCh_{\cycsub}(\bGVec(\sP))}
\newcommand{\bKG}{\bK(\bGVec(\sP))}
\newcommand{\bKGz}{\bK_{\cycsub}(\bGVec(\sP))}

\newcommand{\bFVec}{\text{{\bf FVect}}}
\newcommand{\bChFL}{\bCh(\bFVec(\sL))}
\newcommand{\bChFO}{\bCh(\bFVec(\sO(\sP)))}
\newcommand{\bChFLz}{\bCh_{\cycsub}(\bFVec(\sL))}

\newcommand{\bKF}{\bK(\bFVec(\sL))}
\newcommand{\bKFz}{\bK_{\cycsub}(\bFVec(\sL))}

\newcommand{\bFPoset}{\text{{\bf FPoset}}}

\newcommand{\bFDLat}{\text{{\bf FDLat}}}
\newcommand{\Cell}{\textrm{Cell}}
\newcommand{\fL}{\mathfrak{L}}
\newcommand{\fB}{\mathfrak{B}}
\newcommand{\frH}{\mathfrak{H}}
\newcommand{\frU}{\mathfrak{u}}
\newcommand{\frF}{F}
\newcommand{\frG}{G}
\newcommand{\quo}{q}

\newcommand{\sABlock}{{\mathsf{ ABlock}}}

\newcommand{\bA}{{\bf A}}
\newcommand{\bB}{{\bf B}}

\newcommand{\bK}{{\bf K}}

\newcommand{\K}{{\mathbb{K}}}
\newcommand{\N}{{\mathbb{N}}}
\newcommand{\R}{{\mathbb{R}}}

\newcommand{\Z}{{\mathbb{Z}}}

\newcommand{\scC}{ \mathscr{C} }
\newcommand{\scB}{ \mathscr{B} }

\newcommand{\scG}{ \mathscr{G} }

\newcommand{\cA}{{\mathcal A}}
\newcommand{\cB}{{\mathcal B}}
\newcommand{\cC}{{\mathcal C}}

\newcommand{\cF}{{\mathcal F}}

\newcommand{\cK}{{\mathcal K}}
\newcommand{\cL}{{\mathcal L}}

\newcommand{\cX}{{\mathcal X}}

\newcommand{\sA}{{\mathsf A}}

\newcommand{\sJ}{{\mathsf J}}
\newcommand{\sK}{{\mathsf K}}
\newcommand{\sL}{{\mathsf L}}
\newcommand{\sM}{{\mathsf M}}
\newcommand{\sN}{{\mathsf N}}
\newcommand{\sO}{{\mathsf O}}
\newcommand{\sP}{{\mathsf P}}
\newcommand{\sQ}{{\mathsf Q}}

\newcommand{\sU}{{\mathsf U}}

\def\setof#1{\left\{{#1}\right\}}

\newcommand{\id}{\text{id}}

\newcommand{\cl}{\mathop{\mathrm{cl}}\nolimits}

\DeclareMathOperator{\img}{im}
\DeclareMathOperator{\spans}{span}
\DeclareMathOperator{\st}{star}

\newcommand{\Sub}{\mathsf{Sub}}

\newcommand{\pred}{\overleftarrow}

{\begin{snugshade}\begin{quote}}
{\hfill\end{quote}\end{snugshade}}

\definecolor{shadecolor}{rgb}{0.8,0.8,0.8}

\usepackage{bm}

\newcommand{\sAtt}{{\mathsf{ Att}}}
\newcommand{\sABlockR}{{\mathsf{ABlock}}_{\mathscr{R}}}

\title{A Computational Framework for Connection Matrix Theory}
\author{Shaun Harker, Konstantin Mischaikow, Kelly Spendlove}

\institute{S. Harker \at
              Department of Mathematics, Rutgers University, Piscataway, NJ 08854, USA \\
              \email{sharker@math.rutgers.edu}           
           \and
           K. Mischaikow \at
              Department of Mathematics, Rutgers University, Piscataway, NJ 08854, USA\\
              \email{mischaik@math.rutgers.edu}
            \and 
            K. Spendlove \at
            Mathematical Institute, University of Oxford, Oxford, Oxfordshire OX2 6GG, UK\\
            \email{spendlove@maths.ox.ac.uk}
}

\date{Received: date / Accepted: date}

\begin{document}

\maketitle

\begin{abstract}
The connection matrix is a powerful algebraic topological tool from Conley index theory, a subfield of topological dynamics.  Conley index theory is a purely topological generalization of Morse theory in which the connection matrix subsumes the role of the Morse boundary operator.  Over the last few decades, Conley's approach to dynamics has been cast into a purely computational form.  In this paper we introduce a computational and categorical framework for connection matrix theory.  Broadly speaking, this contribution promotes the computational Conley theory to a computational, homological theory for dynamical systems. More specifically, within this paper we have three specific aims:
\begin{enumerate}
\item We cast connection matrix theory in an appropriate categorical, homotopy-theoretic language. We demonstrate the relationship to the previous definitions of connection matrix.  Lastly, the homotopy-theoretic language allows us to formulate connection matrix theory categorically.  
\item We describe an algorithm for the computation of connection matrices based on algebraic-discrete Morse theory and formalized with the notion of reductions.  We advertise an open-source implementation of our algorithm. 
\item We show that the connection matrix can be used to compute persistent homology.  Ultimately, we believe that connection matrix theory has the potential to be an important tool within topological data analysis.
\end{enumerate}

\end{abstract}

\section{Introduction}\label{sec:intro}

The last few decades have seen the development of algebraic topological techniques for the analysis of data derived from experiment or computation.
An essential step is to make use of the data to construct a finite complex from which the algebraic topological information is computed.
For most applications this results in a high dimensional complex that, because of its lack of structure, provides limited insight into the problems of interest.
The purpose of this paper is to present an efficient algorithm for transforming the complex so that it possesses a particularly simple boundary operator, called the \emph{connection matrix}.
This process is not universally applicable; it requires the existence of a distributive lattice that is coherent with the information that is to be extracted from the complex.
However, there are at least two settings in which we believe that it offers significant potential.

We begin by considering persistent homology, which is a primary tool within the rapidly developing field of topological data analysis~\cite{edelsbrunner:harer,oudot}. In the simplest setting, the input is a cell complex $\cX$ along with a filtration $\varnothing = \cX^0 \subset \cX^1 \subset \cdots \subset \cX^n = \cX$.
Heuristically, persistent homology keeps track of how the homology generators from one level of the filtration are mapped to generators in another level of the filtration, i.e.,  $\iota_\bullet \colon H_\bullet(\cX^i) \to H_\bullet(\cX^j)$ where $\iota_\bullet$ is induced by the inclusion $\cX^i\subset \cX^j$.  In the case of a filtration, this information can be tabulated as a barcode or persistence diagram~\cite{edelsbrunner:harer,oudot}.
In the context of this paper, the filtration is regarded as a  distributive lattice\footnote{See Section~\ref{sec:prelims} for formal definitions associated with order theory and algebraic topology.}  with a total ordering given by the indexing $0 < 1 < \cdots < n$.
The fact that one is constrained to using total orders is a serious limitation and has spurred the development of multi-parameter persistent homology, which remains a topic of current research~\cite{allili2019acyclic,csz,host,mil,scaramuccia2018,sclro}.

A simple generalization of the case of a filtration is to assume that $\cX$ is filtered via a distributive lattice.
To be more precise, assume that $\sL$ is a finite distributive lattice with  partial order denoted by $\leq$.
Let $\setof{\cX^a \subset \cX \mid a\in \sL}$ be an isomorphic lattice (the indexing provides the isomorphism) with operations $\cap$ and $\cup$ and minimal and maximal elements $\varnothing$ and $\cX$, respectively.
An aim of this paper is to provide an efficient algorithm for computing a boundary operator $\Delta$ -- called the \emph{connection matrix} -- of the form
\begin{equation}
\label{eq:connectionMatrix}
\Delta \colon \bigoplus_{a\in \sJ(\sL)} H_\bullet(\cX^a,\cX^{ \pred{a}}) \to \bigoplus_{a\in \sJ(\sL)} H_\bullet(\cX^a,\cX^{\pred a}),
\end{equation}
and, moreover, which has the property that it is strictly upper triangular with respect to  $\leq$ where $\sJ(\sL)$ denotes the set of join-irreducible elements of $\sL$ and $\pred a$ denotes the unique predecessor of $a$, again with respect to $\leq$.

To put this into context, consider the classical handle body decomposition of a manifold.
In this case we have a filtration, i.e., $\sL$ is totally ordered and every element of the filtration $\cX^a$ is join-irreducible. In this setting $\Delta$ is the classical Morse boundary operator~\cite{robbin:salamon2}.  As a consequence, it should be clear that the connection matrix encodes considerable information concerning the relationships between the homology generators of the elements of the lattice. 
Furthermore, as is shown in Section~\ref{sec:PH}, with regard to persistent homology, no information is lost when computing persistent homology from the connection matrix instead of directly from the lattice $\setof{\cX^a}$.
More precisely, we prove the following theorem. 

\begin{thm}
\label{thm:PH}
Let $\cX$ be a finite cell complex with associated chain complex $(C(\cX),\partial)$.
Let $\sL$ be a finite distributive lattice. 
Let $\leq$ be the partial order on $\sL$ defined by $a\leq b$ if and only if $a=a\wedge b$.
Let $\setof{\cX^a \subset \cX \mid a\in \sL}$ be an isomorphic lattice of subcomplexes with operations $\cap$ and $\cup$ and minimal and maximal elements $\varnothing$ and $\cX$, respectively.
Let 
\[
\Delta \colon \bigoplus_{a\in \sJ(\sL)} H_\bullet(\cX^a,\cX^{ \pred a  }) \to \bigoplus_{a\in \sJ(\sL)} H_\bullet(\cX^a,\cX^{ \pred a})
\]
be an associated connection matrix.  Define 
\[
M^a:= \bigoplus_{\setof{b\in \sJ(\sL)\mid b\leq a}} H_\bullet(\cX^b,\cX^{\pred b}).
\]
Then, the persistent homology groups of $\setof{M^a}_{a\in \sL}$ and $\setof{\cX^a}_{a\in \sL}$ are isomorphic.
\end{thm}

The significance of Theorem~\ref{thm:PH} is that the connection matrix can be interpreted as a preprocessing step to computing persistent homology, akin to~\cite{mn}.  A second use of Theorem~\ref{thm:PH} is in settings in which the partial order $\leq$ of $\sL$ has multiple linear extensions that are of interest and $|M|\ll |\cX|$.
Since $|M|\ll |\cX|$, the computation of the persistent homology groups using the chain complex $(M,\Delta)$ is significantly cheaper.
As a consequence of Theorem~\ref{thm:PH}, the relatively expensive computation of the reduced complex $(M,\Delta)$ only needs to be performed once.

The primary motivation for this paper is the goal of developing efficient techniques for the analysis of time series data and computer-assisted proofs associated with deterministic nonlinear dynamics.
As background we recall that C. Conley developed a framework for the global analysis of nonlinear dynamics that makes use of two fundamental ideas \cite{conley:cbms}.
The first is the Conley index of an isolated invariant set, which is an algebraic topological generalization of the Morse index~\cite{mrozek,robbin:salamon:1,salamon}. 
The second is the use of attractors to organize the gradient-like structure of the dynamics.

To explain the relevance of these concepts in greater detail requires a digression.
Let $\varphi$ denote a dynamical system, e.g., a continuous semiflow or a continuous map, defined on a locally compact metric space.
Let $X$ denote a compact invariant set under $\varphi$.
The set of attractors  in $X$ is a bounded distributive lattice~\cite{lsa}.
A \emph{Morse decomposition} of $X$ consists of a finite collection of mutually disjoint compact invariant sets $M(p)\subset X$, called \emph{Morse sets} and indexed by a partial order $(\sP,\leq)$, such that if $x\in X\setminus \bigcup_{p\in \sP}M(p)$, then in forward time (with respect to the dynamics) $x$ limits to a Morse set $M(p)$, in backward time an orbit through $x$ limits to a Morse set $M(q)$, and $p < q$.
To each Morse decomposition there is associated a finite lattice of attractors $\sA$ such that each Morse set $M(p)$ is the maximal invariant set of $A\setminus \pred A$ where $A\in\sJ(\sA)$ and hence has a unique predecessor under the ordering of $\sA$ \cite{kmv3}. 
In fact, Birkhoff's theorem (see Theorem~\ref{thm:birkhoff}) provides an isomorphism between the poset $\sP$ and $\sJ(\sA)$.

In general, invariant sets such as attractors and Morse sets are not computable.
Instead one needs to focus on attracting blocks, which are compact subsets of $X$ that are mapped immediately in forward time into their interior.  
The set of all attracting blocks forms a bounded distributive lattice  under $\cap$ and $\cup$.
An essential fact is that if $\cX$ is a decomposition of $X$ into a cell complex, then starting with a directed graph defined on $\cX$ that acts as an appropriate approximation of $\varphi$ it is possible to rigorously compute attracting blocks. 
Typically, the lattice of all attracting blocks has uncountably many elements.
However, as is shown in \cite{lsa} given a finite lattice of attractors $\sA$ and a fine enough cellular decomposition $\cX$ of $X$, then there exists a lattice of attracting blocks $\sABlock$ constructed using elements of  $\cX$ that is isomorphic (via taking omega limit sets) to $\sA$.

Consider a Morse decomposition of $X$ with its associated lattice of attractors $\sA$.
Let  $\sABlock$ denote an isomorphic lattice of attracting blocks as described above.
In the context of semiflows the homology Conley index of any Morse set $M(p)$ is given by 
\[
CH_\bullet M(p) = H_\bullet (N, \pred N)
\]
for the appropriate choice of $N\in \sJ(\sABlock)$.
Thus, the connection matrix of \eqref{eq:connectionMatrix} can be rewritten as
\begin{equation}
\label{eq:connectionMatrix2}
\Delta \colon \bigoplus_{p\in \sP} CH_\bullet M(p) \to \bigoplus_{p\in \sP} CH_\bullet M(p).
\end{equation}

The existence of $\Delta$ expressed in the form of \eqref{eq:connectionMatrix2} is originally due to R. Franzosa~\cite{fran}; the name connection matrix arose since $\Delta$ can be used to identify and give lower bounds on the structure of connecting orbits between Morse sets \cite{mcmodels,scalar,mischaikow}.   Although Franzosa's proof of the existence of connection matrices is constructive, it is not straightforwardly amenable to computation.    Instead, the classical method of applying connection matrix theory is to compute some of the entries in a connection matrix from knowledge of the flow (such as the flow-defined entries) and exploit the existence theorem of~\cite{fran} by leveraging algebraic constraints (e.g., upper-triangularity, $\ker \Delta/ \img\Delta = H_\bullet(X)$, etc) to reason about the unknown entries.  A nice example of this form of analysis is provided in~\cite[Section 5]{mcr}. The point of view expressed in this paper differs from classical connection matrix literature; our viewpoint is oriented toward computation and data analysis.    In our setup all the chain data is provided as input, from which one fashions a connection matrix.  From a purely data analysis perspective, computation of a connection matrix can be viewed as (chain-level) data reduction without loss of homological information (this is formalized in Section~\ref{sec:reductions}).  Our viewpoint aligns more closely with~\cite{robbin:salamon2}, wherein J. Robbin and D. Salamon provided an alternative proof for the existence of connection matrices,  explicitly relying on the language of posets and lattices which we have also employed. In particular, they introduced the idea of a chain complex being either graded by a poset $\sP$ or filtered by a lattice $\sL$.   Respectively, theses objects and their morphisms constitute the categories $\bChG$ and $\bChFL$, which are described in Sections~\ref{sec:grad} and \ref{sec:lfc}.  The approach we take in this paper is in part an algorithmic analogue to~\cite{robbin:salamon2}.  



In contrast, Franzosa's treatment uses the notion of a {\em chain complex braid} indexed over a poset $\sP$.  The chain complex braid can be understood as a data structure that stores the singular chain data associated to a lattice $\sA$ of attracting blocks.  In this case, the poset $\sP$ arises as the poset of join-irreducibles $\sJ(\sA)$. These objects constitute the category $\bChB$ and are reviewed in Section~\ref{sec:CMT}.   Graded module braids are data structures for storing the homological information contained in a chain complex braid.  Graded module braids form a category $\bGMB$ and there is a functor $\mathfrak{H}\colon \bChB\to \bGMB$ which is analogous to a homology functor. Connection matrix theory for continuous self-maps, as developed by D. Richeson, also employs the structures of chain complex braids and graded module braids~\cite{richeson}.

One goal of this paper is to address how our approach, Franzosa's approach and Robbin and Salamon's approach fit together.  First, we wish to emphasize that in applications data come in the form of a $\sP$-graded cell complex, the collection of which we call $\Cell(\sP)$. Let $\sL$ be the lattice $\sO(\sP)$ of downsets of $\sP$.  A $\sP$-graded cell complex determines three distinct objects: a $\sP$-graded chain complex, an $\sL$-filtered chain complex, and a chain complex braid over $\sP$.   This information can be organized into the following diagram.  
\begin{equation}\label{dia:concept}
\begin{tikzcd}
 & \Cell(\sP) \arrow[d,"\cC",dashed]\arrow[ddl,swap,"\cL",dashed]\arrow[ddr,"\cB",dashed] & \\
 & \bChG\arrow[dl,"\fL"] \arrow[dr,swap,"\fB"]  \ar[d] & \\
 \bChFL \ar[d]& \bKG \ar[dl,"\fL_\bK"]\ar[dr,"\fB_\bK",swap] & \bChB \ar[d]\\
 \bKF && \bKB
\end{tikzcd}
\end{equation}

The dashed arrows are assignments while the solid arrows are functors.  These are described in Sections~\ref{sec:grad}--\ref{sec:CMT}.  Franzosa's theory comprises  the right-hand side of \eqref{dia:concept}, while Robbin and Salamon's theory comprises the left-hand side.  One of our contributions to connection matrix theory is to phrase it in a homotopy-theoretic language.  In particular, we introduce the appropriate homotopy categories  $\bKF$, $\bKG$ and $\bKB$ on the bottom of \eqref{dia:concept}.  This has the following payoffs:
\begin{itemize}
    \item We  resolve the non-uniqueness problems of the connection matrix.  Classically, this is still an open problem.  We show that in our formulation, a connection matrix is unique up to isomorphism. See Proposition~\ref{prop:grad:cmiso} and Remark~\ref{rem:grad:unique}.
    \item We distill the construction of the connection matrix to a particular functor, which we call a {\em Conley functor}.   See Section~\ref{sec:catform}. 
    \item We readily relate persistent homology and connection matrix theory, as in Theorem~\ref{thm:PH}.  See Section~\ref{sec:PH}.  This result implies that the connection matrix has the potential for application in topological data analysis.
\end{itemize}

Moreover, in the context of applications we show that the notions of connection matrices for both $\bChFL$ and $\bChB$ (that is, the formulations of both Robbin-Salamon and Franzosa) may be computed by utilizing graded algebraic-discrete Morse theory (see Algorithm~\ref{alg:cm}) within the category $\bChG$; see Theorems~\ref{thm:filt:cm} and~\ref{thm:braid:cm}.  

This paper lays the theoretical foundation for a new (computational) approach to connection matrix theory.  The result is that this paper is (necessarily) quite formal. However, the work here is complemented by two concrete papers:\ \cite{braids} and\ \cite{hms2}. In~\cite{braids} there are two applications of the computational connection matrix theory developed here to dynamics: one application is to classical examples from connection matrix theory; the other to a Morse theory on spaces of braid diagrams, which has implications for scalar parabolic PDEs.  In~\cite{hms2}, more discussion is given to the specifics of Algorithm~\ref{alg:cm} (\textsc{ConnectionMatrix}), demonstrating that it may be used to compute connection matrices for large, high-dimensional examples, e.g., a graded 10-dimensional cubical complex containing over $25\times 10^9$ cells.


This paper is organized into three parts.
The first part consists of Sections~\ref{sec:examples} and \ref{sec:prelims}.
As indicated above much of the exposition of this paper is rather formal.
To give perspective to this formalism we begin Section~\ref{sec:examples} with a small set of simple examples that we refer back to throughout the paper.
We conclude with a few comments highlighting how the topics in this paper are related to the earlier work of Franzosa~\cite{fran} and Robbin-Salamon~\cite{robbin:salamon2}.  
Section~\ref{sec:prelims} provides the necessary mathematical prerequisites. These are for the most part elementary, but the subsections  cover a wide range of mathematical topics.
We expect that the typical reader is familiar with most, but perhaps not all of these subjects, and thus most of this section can be skipped on first reading.

The second part, consisting of Sections~\ref{sec:grad}--\ref{sec:alg}, presents our perspective on connection matrices and their computability.
In Section~\ref{sec:grad} we introduce poset-graded complexes and our definition of the connection matrix.  
In Section~\ref{sec:reductions} we generalize the concept of reductions for the category of poset-graded complexes.  
In Section~\ref{sec:alg} we give a algebraic-discrete Morse theoretic algorithm to compute a connection matrix.

The final part, made up of Sections~\ref{sec:lfc}--\ref{sec:PH}, is concerned with the applicability of our approach. 
In Section~\ref{sec:lfc} we introduce lattice-filtered complexes and relate our notion of connection matrix to that of Robbin and Salamon.  
In Section~\ref{sec:catform} we show that our the use of homotopy categories gives a functorial formulation of the connection matrix.  
In Section~\ref{sec:CMT} we discuss Franzosa's approach to connection matrix theory and the relationship with ours. 
A consequence of these three sections is that the classical results concerning connection matrix theory are immediately applicable in the context of the connection matrix presented in this paper.
Finally, in Section~\ref{sec:PH} we examine the relationship of the connection matrix to persistent homology and conclude the section with the proof of Theorem~\ref{thm:PH}.



\section{Examples}\label{sec:examples}

As indicated in the introduction we present a few simple examples that will be used later in the paper to illuminate various concepts or definitions.
Unfortunately, many settings in which one is traditionally interested in applying connection matrix theory (viz., dynamics and the search for connecting orbits) require setting up quite a bit of mathematical machinery from dynamical systems. Such examples can be found in~\cite{braids}.  
Thus, for this paper we content ourselves with a selection of examples drawn from applied topology. 
For all of the examples we work with cell complexes (see Definition~\ref{defn:cellComplex}) over the field $\Z_2$.  An overview of the examples is as follows:
\begin{itemize}
    \item Example~\ref{ex:complex} gives an example of a connection matrix.
    \item Example~\ref{ex:bigcomplex} provides a computational perspective on connection matrix theory.  Succinctly, the connection matrix is a form of `homologically-lossless data compression'.
    \item Example~\ref{ex:PH} examines the relationship between the connection matrix and persistent homology.
\end{itemize}
Definitions of the technical terms used these examples are provided in Section~\ref{sec:prelims}.

\begin{ex}[Connection Matrix]\label{ex:complex}
{\em
Let $\sP=\{p,q,r\}$ with order $\leq$ generated by $p\leq q$ and $r \leq q$.   
Consider the cell complexes $\cX, \cX'$ and the maps $\nu,\nu'$ shown below.  

\begin{minipage}[t]{0.45\textwidth}
\begin{center} $\cX$
\begin{tikzpicture}[dot/.style={draw,circle,fill,inner sep=1.5pt},line width=.7pt]
\node (v0) at (0cm:0cm)[dot,label=above:{$v_0$}] {};
\node (v1) at (0cm:2cm)[dot,label=above:{$v_1$}] {};
\node (v2) at (0cm:4cm)[dot,label=above:{$v_2$}] {};
\node (e0) at (0cm:1cm)[label=above:{$e_0$}] {};
\node (e1) at (0cm:3cm)[label=above:{$e_1$}] {};
\draw (v0) to[bend left=10] (v1);
\draw (v1) to[bend left=10] (v2);
\end{tikzpicture}
\[
  \nu(x) =
  \begin{cases}
          p & \text{$x=v_0$ } \\
          r & \text{$x=v_2$ } \\
          q & \text{$x \in \{e_0,v_1,e_1\}$ }\\
  \end{cases}
\]
\vspace{2mm}
\end{center}
\end{minipage}
\begin{minipage}[t]{0.5\textwidth}
\begin{center} $\cX'$
\begin{tikzpicture}[dot/.style={draw,circle,fill,inner sep=1.5pt},line width=.7pt]
\node (v0) at (0cm:0cm)[dot,label=above:{$v_0'$}] {};
\node (v1) at (0cm:2cm)[dot,label=above:{$v_1'$}] {};
\node (e0) at (0cm:1cm)[label=above:{$e_0'$}] {};
\draw (v0) to[bend left=10] (v1);
\end{tikzpicture}
\[
  \nu'(x) =
  \begin{cases}
          p & \text{$x=v_0'$ } \\
          r & \text{$x=v_1'$ } \\
          q & \text{$x=e_0'$ }\\
  \end{cases}
\]
\end{center}
\end{minipage}

 The maps $\nu$ and $\nu'$ are, respectively, poset morphisms from the face posets of $\cX$ and $\cX'$ to $\sP$. 
 The pairs $(\cX,\nu)$ and $(\cX',\nu')$ are called {\em $\sP$-graded cell complexes} (see Section~\ref{sec:grad:cell}).
 The associated chain complexes $C_\bullet(\cX)$ and $C_\bullet(\cX')$ are ($\sP$-filtered) chain equivalent (see Example~\ref{ex:homotopy}). 
 A $\sP$-graded cell complex may be visualized as in Fig.~\ref{fig:ex1}: the Hasse diagram of $\sP$ is given and each $s\in \sP$ is annotated with itself and its fiber $\cX^s:=\nu^{-1}(s)$.  
 In our visualization, if $t,s\in \sP$ and $t$ covers $s$ then there is a directed edge $t\to s$.  
 This orientation coincides with the action of the boundary operator and agrees with the Conley-theoretic literature.  
 
 \begin{figure}[h!]
 \begin{minipage}[t]{0.5\textwidth}
\begin{center}
\begin{tikzpicture}[node distance=.5cm]
\node[draw , ellipse]  (q) {$q: e_0,v_1,e_1\vphantom{v_0'}$};
\node[draw , ellipse]  (p) [below left=.5cm of q] {$p: v_0\vphantom{v_0'}$};
\node[draw , ellipse]  (r) [below right =.5cm of q] {$r: v_2\vphantom{v_0'}$};
\draw[->,>=stealth,thick] (q) -- (p);
\draw[->,>=stealth,thick] (q) -- (r);
\end{tikzpicture}
\caption{Visualization of $(\cX,\nu)$.}\label{fig:ex1}
\end{center}
\end{minipage}
\begin{minipage}[t]{0.5\textwidth}
\begin{center}
\begin{tikzpicture}[node distance=.5cm]
\node[draw , ellipse]  (q) {$q: e_0'$};
\node[draw , ellipse]  (p) [below left=.5cm of q] {$p: v_0'$};
\node[draw , ellipse]  (r) [below right =.5cm of q] {$r: v_1'$};
\draw[->,>=stealth,thick] (q) -- (p);
\draw[->,>=stealth,thick] (q) -- (r);
\end{tikzpicture}
\caption{Visualization of $(\cX',\nu')$.}\label{fig:ex2}
\end{center}
\end{minipage}
\end{figure}

Let $\sL$ be the lattice of down-sets of $\sP$, i.e.,  the lattice with operations $\cap$ and $\cup$, consisting of unions of the sets
\[
\alpha := \downarrow\! p = \{s\in \sP ~|~ s\leq p\} \quad\quad \beta := \downarrow\! q  = \{s\in \sP ~|~ s\leq q\}\quad\quad \gamma := \downarrow\! r = \{s\in \sP ~|~ s\leq r\}.
\]

For $a\in \sL$ define $\cX^a:=\nu^{-1}(a)$.  The collection of fibers  $\{\cX^a\subset \cX\mid a\in \sL\}$ is a lattice of subcomplexes of $\cX$.  Similarly, setting $\cX'^a = \nu'^{-1}(a)$,  the collection $\{\cX'^a\subset \cX'\mid a\in \sL\}$ is a lattice of subcomplexes of $\cX'$.  These lattices are given in Figs.~\ref{fig:lattice:X}--\ref{fig:lattice:X'}, where the down-sets are annotated.
\begin{figure}[h!]
\begin{minipage}[t]{0.45\textwidth}
\begin{center}
\begin{tikzpicture}[node distance=.5cm]
\node (0) {$\varnothing$};
\node (v0) [above left=.5cm of 0] {$\alpha : \setof{v_0}$};
\node (v2) [above right =.5cm of 0] {$\gamma : \setof{v_2}$};
\node (v) [above =1.15 cm of 0] {$\setof{v_0,v_2}$};
\node (e) [above =.775cm of v] {$\beta : \setof{v_0,e_0,v_1,e_1,v_2}$};
\draw (0) -- (v0);
\draw (0) -- (v2);
\draw (v0) -- (v);
\draw (v2) -- (v);
\draw (v) -- (e);
\end{tikzpicture}
\caption{$\setof{\cX^a\mid a\in \sL}$.}\label{fig:lattice:X}
\end{center}
\end{minipage}
\begin{minipage}[t]{0.5\textwidth}
\begin{center}
\begin{tikzpicture}[node distance=.5cm]
\node (0) {$\varnothing$};
\node (v0) [above left=.5cm of 0] {$\alpha : \setof{v_0'}$};
\node (v2) [above right =.5cm of 0] {$\gamma : \setof{v_1'}$};
\node (v) [above =1.1 cm of 0] {$\setof{v_0',v_1'}$};
\node (e) [above =.7cm of v] {$\beta : \setof{v_0',e_0',v_1'}$};
\draw (0) -- (v0);
\draw (0) -- (v2);
\draw (v0) -- (v);
\draw (v2) -- (v);
\draw (v) -- (e);
\end{tikzpicture}
\caption{$\setof{\cX'^a\mid a\in \sL}$.}\label{fig:lattice:X'}
\end{center}
\end{minipage}
\end{figure}

The chain complex $C_\bullet(\cX')$ is called a \emph{Conley complex} for $C_\bullet(\cX)$ and the boundary operator $\Delta'$ of $C_\bullet(\cX')$ is a connection matrix for $C_\bullet(\cX)$.  
To unpack this a bit more, for $\Delta_1'\colon C_1(\cX')\to C_0(\cX')$ we have
\[
\Delta_1' = 
\bordermatrix{    
                  & e_0'  \cr
              v_0' & 1   \cr
              v_1' & 1   }.
\]
A quick computation shows that
\[
H_1(\cX^\beta,\cX^{\pred \beta}) \cong H_1(\cX^\gamma) \cong \Z_2 = C_1(\cX'),
\]
and, using $\cX^{\pred \alpha}=\cX^{\pred \gamma} = \varnothing$, 
\[
H_0(\cX^\alpha,\cX^{\pred \alpha})\oplus H_0(\cX^\gamma,\cX^{\pred \gamma}) \cong H_0(\cX^p)\oplus H_0(\cX^r) \cong \Z_2\oplus \Z_2 = C_0(\cX').
\]
Thus we may write $\Delta_1'$ as a map
\[
\Delta_1'\colon 
H_1(\cX^\beta,\cX^{\pred \beta}) \to
H_0(\cX^\alpha,\cX^{\pred \alpha})\oplus H_0(\cX^\gamma,\cX^{\pred \gamma}).
\]

In context of nonlinear dynamics the relative homology groups $\{H_\bullet(\cX^a,\cX^{\pred a})\}_{a\in \sJ(\sL)}$ are the Conley indices indexed by the join-irreducibles of $\sL$. The Conley index is an algebraic invariant of an isolated invariant set that generalizes the notion of the Morse index~\cite{conley:cbms}.  The classical form of the connection matrix is a boundary operator on Conley indices:
\[
\Delta'\colon \bigoplus_{a\in \sJ(\sL)} H_\bullet(\cX^a,\cX^{\pred a})
\to 
\bigoplus_{a\in \sJ(\sL)} H_\bullet(\cX^a,\cX^{\pred a}).
\]
This form makes it more apparent that non-zero entries in the connection matrix may force the existence of connecting orbits.  See~\cite{fran}.  Note that the classical theory almost exclusively emphasizes the connection matrix.  While we emphasize the connection matrix as well, we also emphasize the notion of the Conley complex (Definition~\ref{def:grad:cm}), which is the chain groups (which may be interpreted as Conley indices) together with the connection matrix.
}
\end{ex}

\begin{ex}[Data Compression]\label{ex:bigcomplex}
{\em 
In applications, the data are typically many orders of magnitude larger than Example~\ref{ex:complex}.  Let $N=9\times 10^9$ and $\cK$ be the cell complex on $[0,1]\subset \R$ where the vertices $\cK_0$ and edges $\cK_1$ are given by
\[
\cK_0 = \left\{
\Big[\frac{k}{N},\frac{k}{N} \Big]
\right\}_{0\leq k\leq N}
\quad\quad \quad\quad 
\cK_1 = \left\{  \Big[
\frac{k}{N},\frac{k+1}{N} \Big]
\right\}_{0\leq k < N}.
\]

Let $\sP$ be as in Example~\ref{ex:complex}.  Let $\mu\colon \cK\to \sP$ be the poset morphism
\[
  \mu(x) = \mu([l,r])=
  \begin{cases}
          p & \text{if $r\leq 3\times 10^9$ }\\
          r & \text{if $l\geq 6\times 10^9$ } \\
          q & \text{otherwise. } \\
  \end{cases}
\]

$\cK$ contains a large number of cells; see~\cite{hms2,braids} for examples where connection matrices are computed with our algorithm for high-dimensional graded cell complexes of similar orders of magnitude.  Due to the cardinality of $\cK$, in Fig.~\ref{fig:ex:big} we give a different visualization for the $\sP$-graded cell complex $(\cK,\mu)$. Here $M:=N/3$ and each $p\in\sP$ is annotated with itself $p$ and the polynomial $F_{\cK^p}(t):=\sum_i \alpha_i t^i$ where $\alpha_i$ is the number of cells in the fiber $\cK^p=\mu^{-1}(p)$ of dimension $i$, viz., the $f$-polynomial of $\cK^p$ (see Definition~\ref{defn:fpoly}).  We call this the {\em fiber graph} of $(\cK,\mu)$, as the Hasse diagram is a directed acyclic graph and each vertex is labeled with information about the fiber.  

 \begin{figure}[h!]
 \centering
\begin{tikzpicture}[node distance=.5cm]
\node[draw , ellipse]  (q)
{$q: (M-1)t^0+Mt^1$};
\node[draw , ellipse]  (p) [below left=.5cm of q] 
{$p: (M+1) t^0 + M t^1$};
\node[draw , ellipse]  (r) [below right =.5cm of q] 
{$r: (M+1) t^0 + Mt^1$};
\draw[->,>=stealth,thick] (q) -- (p);
\draw[->,>=stealth,thick] (q) -- (r);
\end{tikzpicture}
\caption{Fiber graph of $(\cK,\mu)$.}\label{fig:ex:big}
\end{figure}

\begin{figure}[h!]
\begin{center}
\begin{tikzpicture}[node distance=.5cm]
\node[draw , ellipse]  (q) {$q: t^1$};
\node[draw , ellipse]  (p) [below left=.5cm of q] {$p: t^0$};
\node[draw , ellipse]  (r) [below right =.5cm of q] {$r: t^0$};
\draw[->,>=stealth,thick] (q) -- (p);
\draw[->,>=stealth,thick] (q) -- (r);
\end{tikzpicture}
\caption{Conley-Morse graph of $(\cX,\nu)$, i.e., the fiber graph of $(\cX',\nu')$.}\label{fig:ex:cm}
\end{center}
\end{figure}
The graded cell complex $(\cX',\nu')$ of Example~\ref{ex:complex} may also be visualized using the fiber graph, as in Fig.~\ref{fig:ex:cm}. 
An argument similar to the one in Example~\ref{ex:complex} shows that $\Delta'$ is a connection matrix for $(\cK,\mu)$.  Another way to see this is as follows.  Recall that the Poincare polynomial of $\cX$ is defined as the polynomial $P_\cX(t):=\sum_i \dim H_i(\cX) t^i$.  A quick computation shows that
\begin{align*}
    P_{\cK^q}(t) &= t^1 = F_{\cX'^{q}}(t),\\
    P_{\cK^p(t)} &= P_{\cK^r(t)} = t^0 =  F_{\cX'^{p}}(t) = F_{\cX'^{r}}(t) .
\end{align*}
Therefore for each $p\in \sP$ the $f$-polynomial $F_{\cX'^p}(t)$ is precisely the Poincare polynomial $P_{\cK^p}(t)$. This implies that the boundary operator $\Delta'$ can be interpreted as a map on the relative homology groups of the pairs $\{(\cX^a,\cX^{\pred a})\}$.  We call the fiber graph of $(\cX',\nu')$ the \emph{Conley-Morse graph} of $(\cX,\nu)$.  Note that while this is a visualization of a graded complex, the data of the connection matrix itself is not shown.

This example highlights two aspect of the connection matrix.

\begin{enumerate}
    \item The 
    cell complex $(\cX',\nu')$ and connection matrix $\Delta'$ can be thought of as a compression of $(\cK,\mu)$.  Moreover, as part of our definition of connection matrix (see Definition~\ref{def:grad:cm}) there is a ($\sP$-filtered) chain equivalence $\phi$ between the chain complexes $C(\cK)$ and $C(\cX')$. The chain equivalence $\phi$ induces an isomorphism on any homological invariant  (e.g., homology, persistent homology, graded module braid).  Thus from the computational perspective, a connection matrix is a form of (chain-level) compression that does not lose information with respect to homological invariants; that is, a connection matrix is a form of \emph{homologically lossless} compression.  See Example~\ref{ex:PH}, Section~\ref{sec:PH} and Theorems~\ref{thm:PH:iso}.  
    
    \item Notice that $(\cX',\nu')$ cannot be compressed further as $P_{\cX'p}(t)=F_{\cX'^p}(t)$ for each $p$, i.e.,  the $f$-polynomial of $\cX'^p$ is precisely the Poincare polynomial of $\cX'^p$.  In this sense, the connection matrix is maximally compressed and it is the smallest object (of the graded chain equivalence class) capable of recovering the homological invariants.
\end{enumerate}

}
\end{ex}

\begin{ex}[Persistent Homology]\label{ex:PH}
{\em 

In this example we address Theorem~\ref{thm:PH} and the interplay of persistent homology and connection matrix theory.\footnote{In this example we restrict to filtrations.  However, we wish to emphasize that our results hold for multi-parameter persistence; see Section~\ref{sec:PH}, in particular Theorem~\ref{thm:PH:iso}.}  Let $(\cX,\nu)$ and $(\cX',\nu')$ be as in Example~\ref{ex:complex}.  Let $\sQ$ be the poset $\sQ = \{0,1,2\}$ with order $0\leq 1 \leq 2$.  Consider the poset morphism $\rho\colon \sP\to \sQ$ given as follows.
\[
  \rho(x) =
  \begin{cases}
          0 & x=p\\
          1 & x=r \\
          2 & x=q \\
  \end{cases}
\]

Let $\mu:=\rho\circ \nu\colon \cX\to \sQ$ and $\mu':=\rho\circ \nu'\colon \cX'\to \sQ$.  Then $(\cX,\mu)$ and $(\cX',\mu')$ are $\sQ$-graded cell complexes, which are visualized in Fig.~\ref{fig:Q}.

\begin{figure}[h!]
 \begin{minipage}[t]{0.55\textwidth}
     \begin{minipage}[t]{0.45\textwidth}
    \centering
    \begin{tikzpicture}[node distance=.5cm]
    \node[draw , ellipse]  (0) {$0: v_0\vphantom{v_0'}$};
    \node[draw , ellipse]  (1) [above=.5cm of 0] {$1: v_2\vphantom{v_0'}$};
    \node[draw , ellipse]  (2) [above =.5cm of 1] {$2: e_0,v_1,e_1\vphantom{v_0'}$};
    \draw[->,>=stealth,thick] (1) -- (0);
    \draw[->,>=stealth,thick] (2) -- (1);
    \end{tikzpicture}
    \end{minipage}
\begin{minipage}[t]{0.35\textwidth}
\centering
\begin{tikzpicture}[node distance=.5cm]
\node[draw , ellipse]  (0) {$0: v_0'$};
\node[draw , ellipse]  (1) [above=.5cm of 0] {$1: v_1'$};
\node[draw , ellipse]  (2) [above =.5cm of 1] {$2: e_0'$};
\draw[->,>=stealth,thick] (1) -- (0);
\draw[->,>=stealth,thick] (2) -- (1);
\end{tikzpicture}
\end{minipage}
\caption{Visualization of $(\cX,\mu),(\cX',\mu')$.}\label{fig:Q}
\end{minipage}
\begin{minipage}[t]{0.45\textwidth}
\begin{minipage}[t]{0.25\textwidth}
\centering
\begin{tikzpicture}[node distance=.5cm]
\node (e) {$\varnothing$};
\node (0) [above =.5cm of e] {$\downarrow\! 0$};
\node (1) [above =.5cm of 0] {$\downarrow\! 1$};
\node (2) [above =.5cm of 1] {$\downarrow\! 2$};
\draw (0) -- (e);
\draw (1) -- (0);
\draw (2) -- (1);
\end{tikzpicture}
\end{minipage}
\begin{minipage}[t]{0.5\textwidth}
\centering
\begin{tikzpicture}[node distance=.5cm]
\node (e) {$\varnothing$};
\node (0) [above =.5cm of e] {$\setof{0}$};
\node (1) [above =.5cm of 0] {$\setof{0,1}$};
\node (2) [above =.5cm of 1] {$\sQ=\setof{0,1,2}$};
\draw (0) -- (e);
\draw (1) -- (0);
\draw (2) -- (1);
\end{tikzpicture}
\end{minipage}
\caption{Lattice $\sK$ of downsets of $\sQ$.}\label{fig:K}
\end{minipage}
\end{figure}

Let $\sK$ be the lattice of down-sets of $\sQ$; then $\sK$ is the totally ordered lattice of Fig.~\ref{fig:K}.  Setting $\cX^{\downarrow n} := \mu^{-1}(\downarrow \! n)$  for $0\leq n \leq 2$ gives the filtration
\[
\varnothing \subset \cX^{\downarrow 0} \subset \cX^{\downarrow 1} \subset \cX^{\downarrow 2}=\cX'
\quad\quad\quad\quad \varnothing \subset v_0 \subset v_0,v_2  \subset v_0,v_2,e_0,v_1,e_1,
\]
and setting $\cX'^{\downarrow n} :=\mu'^{-1}(\downarrow \! n)$ the filtration
\[
\varnothing \subset \cX'^{\downarrow 0} \subset \cX'^{\downarrow 1} \subset \cX'^{\downarrow 2}=\cX'
\quad\quad\quad\quad
\varnothing \subset v_0'\subset v_0',v_1' \subset v_0',v_1',e_0'.
\]

For each downset $\downarrow i$, the ($\sQ$-filtered) chain equivalence $\phi\colon C(\cX)\to C(\cX')$ induces a chain equivalence  $\phi^{\downarrow i}\colon C(\cX^{\downarrow i})\to C(\cX'^{\downarrow i})$ (see Section~\ref{sec:lfc}), which fit into the following commutative diagram:
\begin{equation}\label{dia:chainfilt}
\begin{tikzcd}[column sep=scriptsize]
0 \ar[r,hookrightarrow]& \ar[r,hookrightarrow]C(\cX^{\downarrow 0}) \ar[r,hookrightarrow]\ar[d,"\phi^{\downarrow 0}"] & C(\cX^{\downarrow 1})\ar[r,hookrightarrow] \ar[d,"\phi^{\downarrow 1}"]& C(\cX) \ar[d,"\phi"]\\
0 \ar[r,hookrightarrow]& C(\cX'^{\downarrow 0}) \ar[r,hookrightarrow]& C(\cX'^{\downarrow 1}) \ar[r,hookrightarrow]& C(\cX').
\end{tikzcd}
\end{equation}
Applying the $i$-th homology functor $H_i(-)$ we have the diagram
\begin{equation}\label{dia:pmm}
\begin{tikzcd}[column sep=scriptsize]
0 \ar[r]& \ar[r]H_iC(\cX^{\downarrow 0}) \ar[r]\ar[d,"H_i\phi^{\downarrow 0}"] & H_iC(\cX^{\downarrow 1})\ar[r] \ar[d,"H_i\phi^{\downarrow 1}"]& H_iC(\cX) \ar[d,"H_i\phi"]\\
0 \ar[r]& H_iC(\cX'^{\downarrow 0}) \ar[r]& H_iC(\cX'^{\downarrow 1}) \ar[r]& H_iC(\cX').
\end{tikzcd}
\end{equation}
Each horizontal sequence in\ \eqref{dia:pmm} is called a \emph{persistence module}~\cite{bubenik2014categorification,oudot}; the collection of vertical maps form a \emph{morphism of persistence modules}. As each map $\phi^{\downarrow i}$ is a chain equivalence, the induced map $H_i\phi^{\downarrow i}$ is an isomorphism.  Thus the collection $\{H_i\phi^{\downarrow i}\}$ is an isomorphism of persistence modules.  It follows that  $\phi$ induces an isomorphism on the persistent homology groups; see Section~\ref{sec:PH}.  As a corollary, all computational tools that tabulate the persistent homology groups, e.g., the persistence diagram and the barcode~\cite{edelsbrunner:harer,oudot}, can be computed via $\Delta'$.  In this way, the connection matrix may be regarded as a preprocessing step for computing the persistence diagram, cf.~\cite{allili2019acyclic,mn,scaramuccia2018}.   


}
\end{ex}

\section{Preliminary material}\label{sec:prelims}

In this section we recall the necessary mathematical prerequisites.   For a more complete discussion the reader is referred to \cite{davey:priestley,roman} for order theory; \cite{mac2013categories,gelfand,weibel} for category theory and homological algebra; \cite{lefschetz} for algebraic topology and cell complexes;  \cite{focm,mn,sko} for discrete Morse theory; \cite{kmv,lsa,lsa2} for (computational) dynamics.  Readers with familiarity of these subjects are recommended to skip this section and only refer to it when necessary.

\subsection{Notation}

Boldface font is used to denote specific categories and Fraktur font to denote particular functors introduced in this paper. Sans-serif font is used for order-theoretic structures, such as posets and lattices, as well as functors between their respective categories.  Lower case Greek letters are used for morphisms of (graded, filtered) chain complexes. Script-like letters are used for chain complex braids and graded module braids and upper case Greek letters are used for morphisms of these objects.    Calligraphic font is used for notation referring to combinatorial objects and cell complexes.

\subsection{Category Theory}
The exposition of category theory primarily follows~\cite{weibel}, see also~\cite{gelfand,mac2013categories}.

\begin{defn}
{\em
An \emph{additive category} $\bA$ is a category such that
\begin{enumerate}
    \item $\bA$ is enriched over the category of abelian groups; that is,  every hom-set $\Hom_\bA(A,B)$ in $\bA$ has the structure of an abelian group such that composition distributes over addition (the group operation),
    \item $\bA$ has a zero object and a product $A\times B$ for every pair $A,B$ of objects in $\bA$.
\end{enumerate}
}
\end{defn}

\begin{defn}
{\em 
In an additive category $\bA$ a \emph{kernel} of a morphism $f\colon B\to C$ is a map $i\colon A\to B$ such that $fi = 0$ and that is universal with respect to this property.  Dually, a \emph{cokernel} of $f$ is a map $e\colon C\to D$ which is universal with respect to having $ef=0$.  in $\bA$, a map $i\colon A\to B$ is \emph{monic} if $ig=0$ implies $g=0$ for every map $g\colon A'\to A$. A monic map is called a \emph{monomorphism}. A map $e\colon C\to D$ is an \emph{epi}, or \emph{epimorphism}, if $he = 0$ implies $h=0$ for every map $h\colon D\to D'$.      
}
\end{defn}

\begin{defn}
{\em
An \emph{abelian category} is an additive category $\bA$ such that 
\begin{enumerate}
    \item Every map in $\bA$ has a kernel and a cokernel,
    \item Every monic in $\bA$ is the kernel of its cokernel,
    \item Every epi in $\bA$ is the cokernel of its kernel.
\end{enumerate}
}
\end{defn}

Let $\frF\colon \bA\to \bB$ be a functor.  For a pair of objects $A,B$ in $\bA$, $\frF$ induces a map on hom-sets
\[
\frF_{A,B}\colon \Hom_\bA(A,B)\to \Hom_\bB(\frF(A),\frF(B)).
\]

\begin{defn}
{\em
A functor $\frF\colon \bA\to \bB$ between additive categories is \emph{additive} if each map $\frF_{A,B}$ is a group homomorphism. 
}
\end{defn}

\begin{defn}
{\em 
A functor $\frF\colon \bA\to \bB$ is \emph{full} if $\frF_{A,B}$ is surjective for all pairs $A,B$.  $\frF$ is \emph{faithful} if $\frF_{A,B}$ is injective for all pairs $A,B$. 
A subcategory $\bA$ of $\bB$ is \emph{full} if the inclusion functor $\bA\hookrightarrow \bB$ is full. A functor is \emph{fully faithful} if it is both full and faithful.
A functor $\frF\colon \bA\to \bB$ is \emph{essentially surjective} if for any object $B$ of $\bB$ there is an object $A$ in $\bA$ such that $B$ is isomorphic to $\frF(A)$.  A functor $\frF\colon \bA\to \bB$ is an \emph{equivalence of categories} if there is a functor $\frG\colon \bB\to \bA$ and natural isomorphisms $\epsilon\colon \frF\frG\to \id_\bB$ and $\eta\colon \id_\bA\to \frG\frF$.  Categories $\bA$ and $\bB$ are \emph{equivalent} if there is an equivalence $\frF\colon \bA\to \bB$.
}
\end{defn}

\begin{prop}[\cite{mac2013categories}, Theorem  IV.4.1]\label{prop:cats:equiv}
$\frF\colon \bA\to \bB$ is an equivalence of categories if and only if $\frF$ is full, faithful and essentially surjective.
\end{prop}
\begin{defn}
{\em
A functor $\frF\colon \bA\to \bB$ is \emph{conservative} if given a morphism $f\colon A\to B$ in $\bA$, $\frF(f)\colon \frF(A)\to \frF(B)$ is an isomorphism only if $f$ is an isomorphism.
}
\end{defn}

\begin{defn}
{\em
Following~\cite[Section II.8]{mac2013categories}, we say that a \emph{congruence relation} $\sim$ on a category $\bA$  is a collection of equivalence relations $\sim_{A,B}$ on $\Hom(A,B)$ for each pair of objects $A,B$ such that the equivalence relations respect composition of morphisms.  That is, if $f,f'\colon A\to B$ and $f\sim f'$ then for any $g\colon A'\to A$ and $h\colon B\to B'$ we have $hfg\sim hf'g$.  If $\bA$ is an additive category, we say a congruence relation $\sim$ is \emph{additive} if $f_0,f_1,g_0,g_1\colon A\to B$ with $f_i\sim g_i$ then $f_0+f_1\sim g_0+g_1$.
}
\end{defn}

\begin{defn}
{\em
Given a congruence relation $\sim$ on a category $\bA$ the \emph{quotient category} $\bA/\!\sim$ is defined as the category whose objects are those of $\bA$ and whose morphisms are equivalence classes of morphisms in $\bA$.  That is,
\[
\Hom_{\bA/\!\sim}(A,B)=\Hom_\bA(A,B)/\!\sim_{A,B}.
\]
There is a quotient functor from $\quo\colon \bA\to \bA/\!\sim $ which is the identity on objects and sends each morphism to its equivalence class.  If $\sim$ is additive then the quotient category $\bA/\!\sim$ is additive, and the quotient functor $\quo$ is an additive functor.
} 
\end{defn}

\begin{prop}[\cite{mac2013categories}, Proposition II.8.1]\label{prelims:cats:quotient}
Let $\sim$ be a congruence relation on the category $\bA$.  Let $\frF\colon \bA\to \bB$ be a functor such that $f\sim f'$ implies $\frF(f)=\frF(f')$ for all $f$ and $f'$, then there is a unique functor $\frF'$ from $\bA/\!\sim$ to $\bB$ such that $\frF'\circ q=\frF$.
\end{prop}

The next result is elementary.

\begin{prop}\label{prop:cat:subquo}
Let $\bA$ be a category, $\sim$ a congruence relation on $\bA$ and $\bB=\bA/\!\sim$ be the quotient category.  Let $\bA'$ be a full subcategory of $\bA$.  If $\bB'$ is the full subcategory of $\bB$ whose objects are the objects in $\bA'$, then $\bB'$ is the quotient category $\bA'/\sim'$ where $\sim'$ is the restriction of $\sim$ to $\bA'$.
\end{prop}

\subsection{Homological Algebra}\label{sec:prelims:HA}

\subsubsection{Additive Categories}

Let $\bA$ be an additive category.  Unless explicitly stated, we assume that functors of additive categories are additive.  In this section we outline the construction of the homotopy category.

\begin{defn}
{\em 
A \emph{chain complex} $(C_\bullet,\partial_\bullet)$ in $\bA$ is a family $C_\bullet =\{C_n\}_{n\in \Z}$ of objects of $\bA$ together with morphsims $\partial_\bullet = \{\partial_n\colon C_n\to C_{n-1}\}_{n\in \Z}$, called \emph{boundary operators}, or \emph{differentials}, such that $\partial_{n-1}\circ\partial_n =0$ for all $n$.   
}
\end{defn}

When the context is clear we abbreviate $(C_\bullet,\partial_\bullet)$ as $(C,\partial)$ or more simply as just $C$.
A morphism $\phi\colon C\to C'$ of chain complexes $C$ and $C'$ is a {\em chain map}, that is, a family $\phi=\{\phi_n\colon C_n\to C_n'\}_{n\in \Z}$ of morphisms in $\bA$ such that $\phi_{n-1}\circ \partial_n = \partial'_n \circ \phi_n$ for all $n$. Chain complexes and chain maps constitute the category
of chain complexes
denoted by $\bCh=\bCh(\bA)$.

\begin{defn}
 {\em
 A chain map $\phi\colon C\to C'$ is {\em null homotopic} if there exists a family $\gamma=\{\gamma_n\colon C_n\to C'_{n+1}\}_{n\in \Z}$ of morphisms in $\bA$ such that 
 \[
 \phi_n=\gamma_{n-1} \circ \partial_n + \partial'_{n+1}\circ\gamma_n .
 \]
 The morphisms $\{\gamma_n\}$ are called a \emph{chain contraction} of $\phi$.
 }
\end{defn}

\begin{defn}
{\em
Two chain maps $\phi,\psi \colon C\to C'$ are {\em chain homotopic} if their difference $\phi-\psi$ is null homotopic, that is, if there exists a family of morphisms $\{\gamma_n\}$ such that
\[
\phi_n-\psi_n = \gamma_{n-1}\circ \partial_n+\partial'_{n+1}\circ \gamma_n .
\]  
The morphisms $\{\gamma_n\}$ are called a \emph{chain homotopy} from $\phi$ to $\psi$.  We write $\phi\sim \psi$ to indicate that $\phi,\psi$ are chain homotopic.  
}
\end{defn}

\begin{defn}\label{defn:degree1}
{\em
Given chain complexes $C$ and $C'$ in $\bA$, a family $\gamma = \{\gamma_n\colon C_n\to C'_{n+1}\}$ of morphisms in $\bA$ is called a {\em degree 1 map} from $C$ to $C'$.
}
\end{defn}

\begin{defn}
 {\em
We say that $\phi \colon C\to D$ is a {\em chain equivalence} if there is a chain map $\psi \colon D\to C$ such that $f\circ g$ and $g\circ f$ are chain homotopic to the respective identity maps of $C$ and $D$.  We say that $C$ and $D$ are {\em chain equivalent} if there exists a chain equivalence $\phi\colon C\to D$.
 }
\end{defn}

The proofs of the next two propositions are straightforward.

\begin{prop}
Chain equivalence is an equivalence relation on the objects of $\bCh(\bA)$.
\end{prop}

\begin{prop}
  The relation $\sim$ is an additive congruence relation on $\bCh(\bA)$.
\end{prop}

\begin{defn}\label{def:prelim:HA:hocat}
{\em
We define the \emph{homotopy category} of $\bCh(\bA)$, which we denote by $\bK = \bK(\bA)$, to be the category whose objects are chain complexes and whose morphisms are chain homotopy equivalence classes of chain maps between chain complexes.  In other words, $\bK(\bA)$ is the quotient category $\bCh(\bA)/\!\sim$ formed by defining hom-sets
\[
\Hom_{\bK(\bA)}(A,B) = \Hom_{\bCh(\bA)}(A,B)/\!\sim.
\]
We denote by $q\colon \bCh(\bA)\to \bK(\bA)$ the quotient functor which sends each chain complex to itself and each chain map to its chain homotopy equivalence class.
}
\end{defn}

It follows from the construction of $\bK(\bA)$ that two chain complexes are isomorphic in $\bK(\bA)$ if and only if they are chain equivalent. The  proofs of the next two results are elementary.

The proofs of the following results are straightforward.

\begin{prop}\label{prop:prelims:induce}
If $F\colon \bA\to \bB$ is a functor of additive categories then there is an associated functor $F_\bCh\colon \bCh(\bA)\to \bCh(\bB)$ given by
\[
F_\bCh(C,\partial) = \Big(\{F(C_n)\}_{n\in \Z},\{F(\partial_n\colon C_n\to C_{n-1})\}_{n\in Z}\Big).
\]
Moreover, $F$ induces a functor $F_\bK\colon \bK(\bA)\to \bK(\bB)$ between the homotopy categories $\bK(\bA)$ and $\bK(\bB)$.
\end{prop}

\begin{prop}\label{prop:HA:equiv}
 If $F\colon \bA\to \bB$ is a equivalence of categories then the induced functor $F_\bCh$ is an equivalence of categories.
\end{prop}

\subsubsection{Abelian Categories}
Let $\bA$ be an abelian category.  Let $C$ be a chain complex in $\bA$.  The elements of $C_n$ are called the $n$-chains.  The elements of $\ker \partial_n\subset C_n$ are called the {\em $n$-cycles}; the elements of $\img\partial_{n+1}\subset C_n$ are called the \emph{$n$-boundaries}.  Since $\bA$ is an abelian category we may define the notion of homology of a chain complex.

\begin{defn}\label{def:homology}
{\em 
We say that a chain complex $(C,\partial)$ is {\em \trivial{}} if $\partial_n=0$ for all $n$. The \trivial{} chain complexes and their morphisms form the full subcategory  $\bCh_0(\bA)\subset \bCh(\bA)$.   A chain complex is called {\em acyclic} if it is exact, i.e.,  $\ker \partial = \img \partial$. A chain complex $B$ is a {\em subcomplex} of $C$ if each $B_n$ is a subspace of $C_n$ and $\partial^B= \partial^C |_B$.  That is, the inclusions $\{i_n\colon B_n\to C_n\}$ form a chain map $i\colon B\to C$.
If $\phi\colon A\to B$ is a chain map then $\ker(\phi)$ and $\img(\phi)$ are subcomplexes of $A$ and $B$ respectively.
Suppose $B$ is a subcomplex of $C$.  The {\em quotient complex} $C/B$ is the chain complex consisting of the family $\{C_n/B_n\}_{n\in \Z}$ together with differentials $\{x+B_n\mapsto \partial_n(x)+B_{n-1}\}_{n\in \Z}$.
The $n$-th \emph{homology} of $C$ is the quotient $H_n(C):= \ker \partial_n/\img \partial_{n+1}$.  We define the \emph{homology} of $C$ as $H_\bullet(C):=\{H_n(C)\}_{n\in \Z}$ equipped with boundary operators $\{0\colon H_n(C)\to H_{n-1}(C)\}_{n\in \Z}$ and regard it as a \trivial{} chain complex.
}
\end{defn}

 A chain complex $C$ is acyclic if and only if $H_\bullet(C)=0$.  Chain maps induce morphisms on homology: let $C,C'$ be chain complexes and $\phi\colon C\to C'$ a chain map.  There there exists a well-defined map $H_\bullet(\phi)$ called the {\em induced map on homology} given via 
\[
H_n(\phi): z+\img \partial_{n+1}\mapsto \phi(z)+\img\partial_{n+1}'.
\]

The proofs of the next few results are straightforward.

\begin{prop}
Homology is a functor $H_\bullet\colon \bCh\to \bCh_0$.  For each $n\in \Z$ the $n$-th homology is a functor $H_n\colon \bCh\to \bA$.
\end{prop}

We often write $H_\bullet$ more simply as $H$.

\begin{prop}
Chain homotopic maps induce the same map on homology.  
\end{prop}

\begin{prop}\label{prop:prelim:chiso}
A chain equivalence $\phi\colon C\to D$ induces an isomorphism on the homology $H(\phi)\colon H(C)\to H(D)$.
\end{prop}

The category $\bK$ enjoys a universal property with respect to chain equivalences.
\begin{prop}[\cite{weibel}, Proposition 10.1.2]
  Let $F\colon \bCh\to D$ be any functor that sends chain equivalences to isomorphisms.  Then $F$ factors uniquely through $\bK$.  In particular, there exists a unqiue functor $H_\bK\colon \bK\to \bCh_0$ such that $H_\bK\circ q = H$.
\end{prop}

Let $\bVec$ be the category of finite-dimensional vector spaces over a field $\K$.  We have the following result for $\bCh(\bVec)$.

\begin{prop}[\cite{gelfand}, Proposition III.2.4]\label{prop:prelims:HA:equiv}
 The pair of functors $q\circ i\colon \bCh_0(\bVec)\to \bK(\bVec)$ and $H_\bK\colon \bK(\bVec)\to \bCh_0(\bVec)$ form an equivalence of categories.
  \end{prop}

We can give an alternative to Proposition~\ref{prop:prelims:HA:equiv} using the results and perspectives of this paper. This will use Algorithm~\ref{alg:hom} and give the flavor of Theorem~\ref{thm:inc:equiv}.  Let $\bK_0(\bVec)$ denote the full subcategory of $\bK(\bVec)$ whose objects are the objects of $\bCh_0(\bVec)$ and whose morphisms are given by
\[
\Hom_{\bK_0}(C,D) = \Hom_{\bCh_0}(C,D)/\!\sim .
\]
There is a quotient functor $q\colon \bCh_0(\bVec)\to \bK_0(\bVec)$.  The next result shows that $\bCh_0(\bVec)$ may be identified with $\bK_0(\bVec)$.
 
\begin{prop}
The quotient functor $q\colon \bCh_0(\bVec)\to \bK_0(\bVec)$ is an isomorphism on hom-sets.
\end{prop}
\begin{proof}
Given $\psi,\phi\colon C\to D$ we have that $\psi\sim \phi$ if there exists $\gamma$ such that 
\[
\psi - \phi = \gamma\circ \partial+\partial \circ\gamma = 0,
\]
where the last equality follows since $\partial=0$ within the subcategory of \trivial{} objects.  Thus $\psi\sim \phi$ if and only if $\psi=\phi$.  
\end{proof}

\begin{prop}
The inclusion functor $i\colon \bK_0(\bVec)\to \bK(\bVec)$ is an equivalence of categories.
\end{prop}
\begin{proof}
The functor $i$ is full and faithful.  Moreover, $i$ is essentially surjective from the Theorem~\ref{thm:alg:hom}, which is the proof of correctness of Algorithm~\ref{alg:hom} (\textsc{Homology}). It follows from Proposition~\ref{prop:cats:equiv} that $i$ is an equivalence of categories.
\end{proof}

This result implies there is an inverse functor to $i$, call it $F$, such that $i$ and $F$ are an equivalence of categories.  In particular, $i\circ F(C)$ is \trivial{} and $i\circ F(C)$ and $C$ are chain equivalent.

\subsection{Order Theory}

\subsubsection{Posets}
\begin{defn}
{\em 
A \emph{partial order} $\leq$ is a reflexive, antisymmetric, transitive binary relation.
A set $\sP$ together with a partial order $\leq$ on $\sP$ is called a \emph{partially ordered set}, or \emph{poset}, and is denoted by $(\sP,\leq)$, or more simply as $\sP$.
We let $<$ be the relation on $\sP$ such that $x<y$ if and only if $x\leq y$ and $x\neq y$.
A function $\nu \colon\sP\to \sQ$ is \emph{order-preserving} if $p\leq q$ implies that $\nu(p)\leq \nu(q)$.  
The \emph{category of finite posets}, denoted $\bFPoset$, is the category whose objects are finite posets and whose morphisms are order-preserving maps.
}
\end{defn}



\begin{defn}\label{def:order:chain}
{\em
Let $\sP$ be a finite poset and $p,q\in \sP$.  We say that $q$ and $p$ are \emph{comparable} if $p\leq q$ or $q\leq p$ and that $p$ and $q$ are {\em incomparable} if they are not comparable.  We say that $q$ {\em covers} $p$ if $p< q$ and there does not exist an $r$ with $p< r < q$.  If $q$ covers $p$ then $p$ is a {\em predecessor} of $q$. Let $I\subset \sP$.  We say that $I$ is a \emph{chain} in $\sP$ if any two elements in $I$ are comparable.  We say that $I$ is an \emph{antichain} in $\sP$ if any two elements in $I$ are incomparable.
}
\end{defn}

\begin{defn}\label{defn:prelims:downset}
{\em 
Let $\sP$ be a finite poset. An {\em up-set} of $\sP$ is a subset $U\subset \sP$ such that if $p\in \sU$ and $p\leq q$ then $q\in U$.  For $p\in \sP$ the {\em up-set at $p$} is $\uparrow\!p:=\{q\in \sP:p \leq q\}$.  Following~\cite{lsa}, we denote the collection of up-sets by $\sU(\sP)$.  A {\em down-set} of $\sP$ is a set $D\subset \sP$ such that if $q\in D$ and $p\leq q$ then $p\in D$.  The {\em down-set at $q$} is $\downarrow\! q:=\{p\in \sP: p \leq q\}$.    Following~\cite{lsa}, we denote the collection of down-sets by $\sO(\sP)$.

}
\end{defn}

\begin{rem}
Any down-set can be obtained by a union of down-sets of the form $\downarrow\!q$.  In fact, $\sO(\sP)$ are the closed sets of the Alexandroff topology of the poset $\sP$.  Under a poset morphism, the preimage of a down-set is a down-set.  Similarly, the preimage of an up-set is an up-set.
\end{rem}

\begin{defn}\label{defn:interval}
{\em
Let $\sP$ be a finite poset. For $p,q\in \sP$ the {\em interval} from $p$ to $q$, denoted $[p,q], $ is the set $\{x\in \sP: p\leq x \leq q\}$. A subset $I\subset \sP$ is {\em convex} if whenever $p,q\in I$ then $[p,q]\subset I$.  Following~\cite{fran}, we denote the collection of convex sets by $I(\sP)$. 
}
\end{defn}
 
 \begin{rem}
 Let $\sP$ be a finite poset.  Any convex set of $\sP$ can be obtained by an intersection of a down-set and an up-set.  Under a poset morphism the preimage of a convex set is a convex set.   See~\cite{roman}.
 \end{rem}
 \begin{rem}
 In~\cite{fran2,fran3,fran} convex sets are instead called intervals.  We adopt the terminology convex as this is standard in order theory literature.
 \end{rem}

\subsubsection{Lattices}\label{sec:prelims:lat}

\begin{defn}
{\em
A {\em lattice} is a set $\sL$ with binary operations $\vee,\wedge \colon \sL\times \sL\to \sL$ satisfying the following four axioms:

\begin{enumerate}[wide, labelwidth=!, labelindent=0pt]
\item (idempotent) $a\wedge a = a \vee a = a$ for all $a\in \sL$,
\item (commutative) $a\wedge b = b\wedge a$ and $a\vee b = b \vee a$ for all $a,b\in \sL$,
\item (associative) $a\wedge (b\wedge c) = (a\wedge b)\wedge c$ and $a\vee(b\vee c) = (a\vee b)\vee c$ for all $a,b,c\in \sL$,
\item (absorption) $a\wedge (a\vee b) = a\vee (a\wedge b)=a$ for all $a,b\in \sL$.\\
A lattice $\sL$ is {\em distributive} if it satisfies the additional axiom:
\item (distributive) $a\wedge (b\vee c) = (a\wedge b)\vee (a\wedge c)$ and $a\vee (b\wedge c) = (a\vee b) \wedge (a\vee c)$ for all $a,b,c\in \sL$.\\
A lattice $\sL$ is {\em bounded} if there exist {\em neutral} elements $0$ and $1$ that satisfy the following property:
\item $0\wedge a = 0, 0\vee a = a, 1\wedge a = a, 1\vee a = 1$ for all $a\in \sL$.
\end{enumerate}
}
\end{defn}

A lattice homomorphism $f \colon \sL\to \sM$ is a map such that if $a,b\in \sL$ then $f(a\wedge b) = f(a)\wedge f(b)$ and $f(a\vee b) = f(a)\vee f(b)$.  
If $\sL$ and $\sM$ are bounded lattices then we also require that $f(0)=0$ and $f(1)=1$.  We are particularly interested in finite lattices.  Note that every finite lattice is bounded.  A subset $\sK\subset \sL$ is  a sublattice of $\sL$ if $a,b\in \sK$ implies that $a\vee b\in \sK$ and $a\wedge b\in \sK$. For sublattices of bounded lattices we impose the extra condition that $0,1\in \sK$.

\begin{defn}
{\em
The \emph{category of finite distributive lattices}, denoted $\bFDLat$, is the category whose objects are finite distributive lattices and whose morphisms are lattice homomorphisms.
}
\end{defn}

A lattice $\sL$ has an associated poset structure given by $a\leq b$ if $a=a\wedge b$ or, equivalently, if $b=a\vee b$.

\begin{defn}\label{def:prelims:joinirr}
{\em
An element $a\in \sL$ is {\em join-irreducible} if it has a unique predecessor; given a join-irreducible $a$ we denote its unique predecessor by $\pred a$.
The set of join-irreducible elements of $\sL$ is denoted by $\sJ(\sL)$.  The poset of join-irreducible elements, also denoted $\sJ(\sL)=(\sJ(\sL),\leq)$, is the set $\sJ(\sL)$ together with $\leq$, where $\leq$ is the restriction of the partial order of $\sL$ to $\sJ(\sL)$.
}
\end{defn}

\begin{defn}
{\em 
For $a\in \sL$ the expression
\[
a = b_1\vee \cdots \vee b_n
\]
where the $b_i$'s are distinct join-irreducibles is called \emph{irredundant} if it is not the join of any proper subset of $U=\{b_1,\ldots,b_n\}$. 
}
\end{defn}
\begin{prop}[\cite{roman},Theorem 4.29]\label{prop:jr:rep}
If $\sL$ is a finite distributive lattice then every $a\in \sL$ has an irredundant join-irreducible representation
\[
a = b_1\vee \cdots \vee b_n
\]
and all such representations have the same number of terms.
\end{prop}


\begin{defn}\label{defn:boolean}
{\em
A {\em complemented lattice}, is a bounded lattice (with least element $0$ and greatest element $1$), in which every element $a$ has a complement, i.e., an element $b$ such that $a\vee b = 1$ and $a\wedge b= 0$. 
}
\end{defn}

 
 \begin{defn}\label{def:prelims:lat:rc}
 {\em
 A {\em relatively complemented} lattice is a lattice such that every interval $[a,b]$, viewed as a bounded lattice, is complemented.
 }
 \end{defn}

\begin{ex}\label{def:prelims:lat:subv}
{\em
Let $V$ be a vector space.  The associated \emph{lattice of subspaces}, denoted by $\Sub(V)$, consists of all subspaces of $V$ with the operations $\wedge := \cap$ and $\vee := +$ (span).  $\Sub(V)$ is a relatively complemented lattice.  It is not distributive in general.
}
\end{ex}

\begin{defn}
{\em
Let $C$ be a chain complex.  The associated \emph{lattice of subcomplexes}, denoted by $\Sub(C)$, consists of all  subcomplexes of $C$ with the operations $\wedge := \cap$ and $\vee := +$ (span), i.e.,
\begin{align*}
(A_\bullet,\partial^A)\wedge (B_\bullet,\partial^B) &:= (A_\bullet\cap B_\bullet,\partial^C|_{A\cap B}),\\
(A_\bullet,\partial^A)\vee (B_\bullet,\partial^B) &:= (A_\bullet+ B_\bullet,\partial^C|_{A+B}).
\end{align*}
$\Sub(C)$ is a bounded lattice, but  is not distributive in general. 
}
\end{defn}

\subsubsection{Birkhoff's Theorem and Transforms}\label{sec:birkhoff}
As indicated above, given a finite distributive lattice $\sL$,  $\sJ(\sL)$ has a poset structure.
In the opposite direction, given a finite poset $(\sP,\leq)$ the collection of downsets $\sO(\sP)$ is a finite distributive lattice with operations $\wedge = \cap$ and $\vee = \cup$.  The following theorem often goes under the moniker `Birkhoff's Representation Theorem'.

\begin{thm}[\cite{roman}, Theorem 10.4]\label{thm:birkhoff}
$\sJ$ and $\sO$ are contravariant functors from $\bFDLat$ to $\bFPoset$ and $\bFPoset$ to $\bFDLat$, respectively.  Moreover,  $\sO$ and $\sJ$ form an equivalence of categories.
\end{thm}

The pair of contravariant functors $\sO$ and $\sJ$ are called the {\em Birkhoff transforms}.   Given $\nu\colon \sP\to \sQ$ we say that $\sO(\nu)$ is the {\em Birkhoff dual} to $\nu$.  Similarly, for $h\colon \sK\to \sL$ we say that $\sJ(h)$ is the {\em Birkhoff dual} to $h$. 

\begin{ex}
{\em
Consider the poset $\sP$ of Example~\ref{ex:complex}, recalled in Fig.~\ref{fig:birkhoffex}(a).  The lattice of down-sets $\sO(\sP)$ is given in Fig.~\ref{fig:birkhoffex}(b) and the join-irreducibles $\sJ(\sO(\sP))$ in Fig.~\ref{fig:birkhoffex}(c).
\begin{figure}[ht!]
\begin{minipage}[t]{0.3\textwidth}
\begin{center}
\begin{tikzpicture}[node distance=.5cm]
\node (0) {};
\node (p) [above left=.5cm of 0] {$p$};
\node (r) [above right =.5cm of 0] {$r$};
\node (q) [above =1.15 cm of 0] {$q$};
\draw (q) -- (p);
\draw (q) -- (r);
\end{tikzpicture}\\
(a)
\end{center}
\end{minipage}
\begin{minipage}[t]{0.3\textwidth}
\begin{center}
\begin{tikzpicture}[node distance=.5cm]
\node (0) {$\varnothing$};
\node (v0) [above left=.5cm of 0] {$\setof{p}$};
\node (v2) [above right =.5cm of 0] {$\setof{r}$};
\node (v) [above =1.15 cm of 0] {$\setof{p,r}$};
\node (e) [above =.775cm of v] {$\setof{p,r,q}$};
\draw (0) -- (v0);
\draw (0) -- (v2);
\draw (v0) -- (v);
\draw (v2) -- (v);
\draw (v) -- (e);
\end{tikzpicture}\\
(b)
\end{center}
\end{minipage}
\begin{minipage}[t]{0.3\textwidth}
\begin{center}
\begin{tikzpicture}[node distance=.5cm]
\node (0) {};
\node (p) [above left=.5cm of 0] {$\setof{p}$};
\node (r) [above right =.5cm of 0] {$\setof{r}$};
\node (v) [above =1.15 cm of 0] {};
\node (q) [above =.775cm of v] {$\setof{p,r,q}$};
\draw (q) -- (p);
\draw (q) -- (r);
\end{tikzpicture}\\
(c)
\end{center}
\end{minipage}
\caption{(a) Poset  $\sP$ (b) Lattice of down-sets $\sO(\sP)$ (c) Join-irreducibles $\sJ(\sO(\sP))$.
}\label{fig:birkhoffex}
\end{figure}
}
\end{ex}

\subsection{Cell Complexes}\label{sec:prelims:cell}

Since our ultimate focus is on data analysis, we are interested in combinatorial topology.  We make use of the following complex, whose definition is inspired by~\cite[Chapter III (Definition 1.1)]{lefschetz}.  Recall that $\K$ is a field.

\begin{defn}
\label{defn:cellComplex}
{\em
A {\em cell complex} $(\cX,\leq,\kappa,\dim)$ is an object $(\cX,\leq)$ of $\bFPoset$ together with two associated functions $\dim\colon \cX\to \N$ and $\kappa\colon \cX\times \cX\to \K$ subject to the following conditions:
\begin{enumerate}
\item $\dim\colon(\cX,\leq)\to (\N,\leq)$ is order-preserving;
\item  For each $\xi$ and $\xi'$ in $\cX$:
\[
\kappa(\xi,\xi')\neq 0\quad\text{implies } \xi'\leq \xi \quad\text{and}\quad \dim(\xi) = \dim(\xi')+1;
\]
\item\label{cond:three} For each $\xi$ and $\xi''$ in $\cX$,
\[
\sum_{\xi'\in X} \kappa(\xi,\xi')\cdot \kappa(\xi',\xi'')=0.
\]
\end{enumerate}

For simplicity we typically write $\cX$ for $(\cX,\leq,\kappa,\dim)$.  
The partial order $\leq$ is the {\em face partial order}.
The set of cells $\cX$ is a graded set with respect to $\dim$, i.e., $\cX = \bigsqcup_{n\in \N} \cX_n$ with $\cX_n = \dim^{-1}(n)$.  
An element $\xi\in \cX$ is called a {\em cell} and $\dim \xi$ is the {\em dimension} of $\xi$. 
The function $\kappa$ is the {\em incidence function} of the complex. The values of $\kappa$ are referred to as the {\em incidence numbers}.  
}
\end{defn}

\begin{defn}
{\em
Given a cell complex $\cX$ the {\em associated chain complex $C(\cX)$} is the chain complex $C(\cX) = \{C_n(\cX)\}_{n\in\Z}$ where $C_n(\cX)$ is the vector space over $\K$ with basis elements given by the cells $\xi\in \cX_n$ and the boundary operator $\partial_n\colon C_n(\cX) \to C_{n-1}(\cX)$ is defined by
\[
\partial_n( \xi) := \sum_{\xi' \in \cX} \kappa(\xi, \xi')\xi'.
\]
Condition~(\ref{cond:three}) of Definition~\ref{defn:cellComplex} ensures $\partial_{n-1}\circ \partial_n = 0$. 
}
\end{defn}

\begin{defn}
{\em
Given a cell complex $\cX$ the {\em homology} of $\cX$, denoted $H_\bullet(\cX)$, is defined as the homology of the associated chain complex $H_\bullet(C_\bullet(\cX))$.
}
\end{defn}


\begin{defn}\label{defn:cell:cyclic}
{\em
A cell complex $\cX$ is {\em \trivial} if the associated chain complex $(C(\cX),\partial)$ is \trivial~(see Definition~\ref{def:homology}).
}
\end{defn}
Given a cell complex $(\cX,\leq,\kappa,\dim)$ and a subset $\cK\subset \cX$, we denote by $(\cK,\leq,\kappa,\dim)$ the set $\cK$ together with the restriction of $(\leq,\kappa,\dim)$ to $\cK\subset \cX$. 

\begin{defn}
{\em
Given $\cK\subset \cX$, we say that $(\cK,\leq,\kappa,\dim)$ is a \emph{subcomplex of $\cX$} if $(\cK,\leq,\kappa,\dim)$ is a cell complex.  
}
\end{defn}

\begin{rem}

 In an abuse of language, we will often refer to the subset $\cK\subset\cX$ itself as a subcomplex of $\cX$, although what is meant is the 4-tuple $(\cK,\leq,\kappa,\dim)$, i.e., the subset $\cK$ together with the restriction of $(\leq,\kappa,\dim)$ to $\cK$.
\end{rem}

\begin{rem}
Given any subcomplex $(\cK,\leq,\kappa,\dim)$ there is an associated chain complex $C(\cK)$.  However the inclusion $\cK\subset \cX$ need not induce a chain map $C(\cK)\to C(\cX)$.  In other words, the associated chain complex $C(\cK)$ need not be a subcomplex of $C(\cX)$.  For example, let  $\cX$ be as in Example~\ref{ex:complex} and set $\cK=\{e_0,e_1\}$.  $(\cK,\leq,\kappa,\dim)$ is itself a cell complex, however $C(\cK)$ is not a subcomplex of $C(\cX)$.
\end{rem}

\begin{prop}\label{prop:cell:convex}
Let $\cX$ be a cell complex. If $I\subset \cX$ is a convex set in $(\cX,\leq)$  then $(I,\leq,\kappa,\dim)$ is a subcomplex of $\cX$.
\end{prop}


\begin{defn}
{\em
Given a subcomplex $(\cK,\leq,\kappa,\dim)$ of $(\cX,\leq,\kappa,\dim)$, we say that $(\cK,\leq,\kappa,\dim)$ is \emph{closed} if $\cK\in \sO(\cX)$ and \emph{open} if $\cK\in \sU(\cX)$.
}
\end{defn}
\begin{prop}
  If $(\cK,\leq,\kappa,\dim)$ is a closed subcomplex of $(\cX,\leq,\kappa,\dim)$, then $C(\cK)$ is a subcomplex of $C(\cX)$. 
\end{prop}


\begin{defn}\label{defn:latclsub}
{\em
Given a cell complex $(\cX,\leq,\kappa,\dim)$, the \emph{lattice of closed (cell) subcomplexes of $\cX$}, denoted by $\Sub_{Cl}(\cX)$, consists of all closed subcomplexes of $\cX$, together with operations
\begin{align*}
(\cK,\leq,\kappa,\dim)\wedge (\cK',\leq,\kappa,\dim) := (\cK\cap \cK',\leq,\kappa,\dim),\\
(\cK,\leq,\kappa,\dim)\vee (\cK',\leq,\kappa,\dim) := (\cK\cup \cK',\leq,\kappa,\dim).\\
\end{align*}
There is a lattice monomorphism $\spans \colon \Sub_{Cl}(\cX)\to \Sub(C(\cX))$ given by
\[
a\mapsto \spans(a) = \setof{\sum_{i=0}^n \lambda_i \xi_i: n\in \N, \lambda_i\in \K, \xi_i\in a}.
\]
The \emph{lattice of closed (chain) subcomplexes of $C(\cX)$} is defined as $\Sub_{Cl}(C(\cX)):=\img \spans$.  A subcomplex of $C(\cX)$ which belongs to $\Sub_{Cl}(C(\cX))$ is a \emph{closed (chain) subcomplex} of $C(\cX)$.
}
\end{defn}
\begin{rem}
The lattices $\Sub_{Cl}(C(\cX))$ and  $\Sub_{Cl}(\cX)$ are isomorphic. This implies that $\Sub_{Cl}(C(\cX))$ is a distributive lattice, whereas in general $\Sub(C(\cX))$ is not distributive.  The lattice morphism $\spans$ factors as
\[
\Sub_{Cl}(\cX)\xrightarrow{\spans} \Sub_{Cl}(C(\cX))\hookrightarrow \Sub(C(\cX)).
\]
\end{rem}



We define the star and closures: 
\[
\st(\xi):= \, \uparrow\! \xi= \{\xi': \xi \leq \xi'\}\quad\quad \text{ and } \quad\quad \cl(\xi) := \ \downarrow\!\xi = \{\xi':\xi'\leq \xi\}.
\]

The star defines an open subcomplex while the closure defines a closed subcomplex.  In order-theoretic terms these are an up-set and down-set of $(\cX,\leq)$ at $\xi$.  We use the duplicate notation $\st,\cl$ to agree with the literature of cell complexes.



\begin{defn} \label{defn:topcell}
{\em 
Given a complex $\cX$, a cell $\xi\in \cX$ is a {\em top-cell} if it is maximal with respect to $\leq$, i.e., $\st(\xi) = \{\xi\}$.  Following~\cite{focm}, we denote the set of top-cells is denoted $\cX^+\subset \cX$.
}
\end{defn}

\begin{defn} \label{defn:cored}
{\em
Given a subcomplex $\cX'\subseteq \cX$, a pair of cells $(\xi,\xi')\in \cX'\times \cX'$ is a {\em coreduction pair} in $\cX'$ if $\partial(\xi) = \kappa(\xi,\xi') \xi'$ with $\kappa(\xi,\xi')\neq 0$.
A cell $\xi\in \cX$ is {\em free in $\cX'$} if $\kappa(\xi,\xi')=0$ for $\xi'\in \cX'$.
}
\end{defn}

\begin{defn}\label{defn:fpoly}
{\em
Let $\cX$ be a cell complex.  The {\em $f$-vector} of $\cX$ is the integral sequence 
\[
(f_0,f_1,f_2,\ldots)
\]
where $f_i$ is the number of $i$-dimensional cells.  The \emph{$f$-polynomial} of $\cX$ is the polynomial
\[
\cF_\cX(t) = \sum_i f_it^i .
\]
The {\em Poincare polynomial} of $\cX$ is the polynomial
\[
P_\cX(t) = \sum_i \dim H_i(\cX)t^i.
\]
}
\end{defn}

\subsection{Discrete Morse Theory}\label{prelims:dmt}
We review the use of discrete Morse theory to compute homology of complexes. Our exposition is brief and follows~\cite{focm}.  See also~\cite{gbmr,mn,sko}.

\begin{defn}\label{defn:acyclicmatching}
{\em
A {\em partial matching} of cell complex $\cX$ consists of a partition of the cells in $\cX$ into three sets $A$, $K$, and $Q$ along with a bijection $w\colon Q\to K$ such that for any $\xi\in Q$ we have that $\kappa(w(\xi),\xi)\neq 0$.  A partial matching is called {\em acyclic} if  the transitive closure of the binary relation  $\ll$ on $Q$ defined by $$\xi' \ll \xi \text{ if and only if } \kappa (w(\xi),\xi')\neq 0$$
 generates a partial order $\leq$ on $Q$.
 }
 \end{defn} 


We lift the partial matching to a degree 1 map (see Definition~\ref{defn:degree1}) $V\colon C_\bullet (\cX)\to C_{\bullet+1}(\cX)$ by defining it using the distinguished basis via:
 \begin{align}\label{eqn:defn:V}
 V(x) = 
\begin{cases}
\kappa (\xi, \xi') w(x) & x\in Q, \\
0 & \text{otherwise.}
\end{cases}
 \end{align}
 
We denote acyclic partial matchings by the tuple $(A,w\colon Q\to K)$.  
An acyclic partial matching $(A,w\colon Q\to K)$ of $\cX$ can be used to construct a new chain complex. 
This is done through the observation that acyclic partial matchings produce degree 1 maps $C_\bullet(\cX)\to C_{\bullet+1}(\cX)$ called {\em splitting homotopies}.  
Splitting homotopies are reviewed in depth in Section~\ref{sec:reductions}. Further references to the use of splitting homotopies within discrete Morse theory can be found in~\cite{sko}.  
The following proposition is from~\cite{focm}, however we make a sign change to agree with the exposition in Section~\ref{sec:reductions}.
  
 \begin{prop}[\cite{focm}, Proposition 3.9]\label{prop:matchinghomotopy}
An acyclic partial matching $(A,w)$ induces a unique linear map $\gamma\colon C_\bullet(\cX)\to C_{\bullet+1}(\cX)$ so that $\img(\mathrm{id}_\cX-\partial \gamma)\subseteq C_\bullet(A)\oplus C_\bullet(K)$, $\img\gamma = C_\bullet(K)$ and $\ker\gamma = C_\bullet(A)\oplus C_\bullet(K)$.  It is given by the formula
\begin{align}\label{eqn:gamma}
\gamma = \sum_{i\geq 0} V(\mathrm{id}-\partial V)^i .
\end{align}
\end{prop}

  Let $\iota_A\colon C_\bullet(A)\to C_\bullet(\cX)$ and $\pi_A\colon C_\bullet(\cX)\to C_\bullet(A)$ be the canonical inclusion and projection.  Define $\psi\colon C_\bullet(\cX)\to C_\bullet(A), \phi\colon C_\bullet(A)\to C_\bullet(\cX)$ and $\partial^A\colon C_\bullet(A)\to C_{\bullet-1}(A)$ by 
  \begin{equation}\label{eqn:prelim:dmt}
    \psi:=\pi_A\circ (\id-\partial \gamma) \quad\quad  \phi:= (\id-\gamma \partial)\circ \iota_A \quad\quad \partial^A:= \psi\circ \partial\circ \phi.
  \end{equation}

 \begin{thm}[\cite{focm}, Theorem 3.10]\label{thm:focm:red}
 $(C_\bullet(A),\partial^A)$ is a chain complex and $\psi,\phi$ are chain equivalences.  In particular,
 \[
 \psi\circ \phi = \mathrm{id}\quad\quad \phi\circ\psi - \mathrm{id} = \partial\circ \gamma+\gamma\circ \partial.
 \]
 \end{thm}
 
 As a corollary $H_\bullet(C_\bullet(A))\cong H_\bullet(C_\bullet(\cX))$.  Regarding computations, acyclic partial matchings are relatively easy to produce, see~[Algorithm 3.6 (Coreduction-based Matching)]\cite{focm}, which is recalled in Section~\ref{sec:alg}.   Moreover, given an acyclic partial matching there is an efficient algorithm to produce the associated splitting homotopy~\cite[Algorithm 3.12 (Gamma Algorithm)]{focm}, also recalled in Section~\ref{sec:alg}.




\section{Graded Complexes}\label{sec:grad}

In this section (and later in Section~\ref{sec:lfc}) we introduce objects which result from a marriage of homological algebra and order theory.  We also introduce our notion of connection matrix, which is part of what we call a Conley complex.  In particular, we employ categorical language and explicitly develop an appropriate homotopy category for connection matrix theory over fields.  Along the way we provide  motivation through a selection of examples, primarily building upon Example~\ref{ex:complex} in the introduction. 

A primary result of this section, with implications for  connection matrix theory, is Proposition~\ref{prop:grad:cmiso} that shows that the Conley complex is an invariant of the chain equivalence class.  This result implies that the non-uniqueness of the connection matrix is captured in terms of a change of basis, cf.~\cite{atm}.

For the remainder of this section let $\K$ be a field and let $\sP$ be a finite poset.  Recall that $\bVec$ is the category of vector spaces over $\K$.

\subsection{Graded Vector Spaces}\label{sec:grad:vs}

\begin{defn}\label{def:graded:vs}
{\em 
A {\em $\sP$-graded vector space} $V=(V,\pi)$ is a vector space $V$ equipped with a $\sP$-indexed family of idempotents (projections) $\pi=\{\pi^p\colon V\to V\}_{p\in \sP}$ such that $\sum_{p\in \sP} \pi^p=\id_V$ and if $p\neq q$ then $\pi^p\circ \pi^q=0$.  We call $\pi$ a \emph{$\sP$-grading} of $V$. Suppose $(V,\pi_V)$ and $(W,\pi_W)$ are $\sP$-graded vector spaces.  A map $\phi\colon V\to W$ is {\em $\sP$-filtered} if for all $p,q\in \sP$ 
\begin{equation}\label{eqn:map:filt}
\phi^{pq}:=\pi^p_{W}\circ \phi\circ \pi_V^q\neq 0\text{ implies that } p\leq q.
\end{equation}
Equivalently, $\phi$ is $\sP$-filtered if $$\phi=\sum_{p\leq q}\pi^p_{W}\circ \phi\circ \pi_V^q =  \sum_{p\leq q} \phi^{pq} .$$
}
\end{defn}

In the sequel, we write $\pi$ for $\pi_V$ and drop the dependence on $V$; the domain of $\pi$ can always be inferred from context. We also write $V^p$ for $\img \pi^p$.  The next few results establish that working with $\sP$-graded vector spaces and $\sP$-filtered linear maps follows the rules of working with upper triangular matrices.  The proofs involve  elementary linear algebra. 


\begin{prop}
 A $\sP$-graded vector space $(V,\pi)$ admits a decomposition $V=\bigoplus_{p\in \sP} V^p$.
Conversely, if $V$ is a vector space and $V=\bigoplus_{p\in \sP} V^p$ then the collection $\pi = \setof{\pi^p}$ with $\pi^p(\sum_{q\in \sP} v^q) := v^p$ where $v^q \in V^q$ is a $\sP$-grading of $V$.
\end{prop}

\begin{prop}\label{prop:grad:comp}
If $\phi\colon (U,\pi)\to(V,\pi)$ and $\psi\colon (V,\pi)\to (W,\pi)$ are $\sP$-filtered linear maps, then the composition $\psi\circ \phi$ is $\sP$-filtered and
\begin{align*}
(\psi\circ\phi)^{pq} = \sum_{p \leq r \leq q} \psi^{pr}\circ\phi^{rq}\quad\quad\quad\quad \psi\circ\phi= \sum_{p\leq q} (\psi\circ\phi)^{pq}.
\end{align*}
\end{prop}

\begin{rem}
In~\cite{fran}, a linear map $\phi$ which obeys Eqn.~\eqref{eqn:map:filt} is called \emph{upper triangular with respect to $\sP$}.  The terminology $\sP$-filtered is apt as it is readily verified that $\phi$ obeys Eqn.~\eqref{eqn:map:filt} if and only if for all $q\in \sP$
\[
\phi(V^q)\subseteq \bigoplus_{p\leq q} W^p.
\]
This is in turn equivalent to
\[
\phi(\bigoplus_{p\leq q} V^q)\subseteq \bigoplus_{p\leq q} W^q.
\]
See Proposition~\ref{prop:map:filtgrad}; cf. Definition~\ref{defn:filt:vs}.
\end{rem}



\begin{defn}
{\em 
The \emph{category of $\sP$-graded vector spaces}, denoted $\bGVec(\sP)$, is the category whose objects are $\sP$-graded vector spaces and whose morphisms are $\sP$-filtered linear maps.  
}
\end{defn}

For a $\sP$-graded vector space $(V,\pi)$ any projection $\pi^p\colon V\to V$ factors as 
\[
V\xrightarrow{e^p} V^p \xrightarrow{\iota^p} V,
\]
where  $e^p\colon V\to V^p$ is the epimorphism to $\img\pi^p=V^p$ and $\iota^p\colon V^p\hookrightarrow V$ is the natural inclusion.  We have the identities
\[
\pi^p = \iota^p\circ e^p,\quad\quad\quad\quad \pi^p\circ \iota^p = \iota^p,\quad\quad\quad\quad
e^p\circ \pi^p = e^p.
\]

Given a linear map $\phi\colon (V,\pi)\to (W,\pi)$ we define 
$$\Phi^{pq}:= e^p\circ \phi\circ \iota^q \colon V^q\to W^p.$$
Using the upper-case $\Phi^{pq}$ as above is our convention to refer to the restriction of $\phi$ to the $(p,q)$-matrix entry.  The linear map $\phi$ is equivalent to the matrix of linear maps $\Phi = \{\Phi^{pq}\}_{p,q\in \sP}$, via
\[
\sum_{p, q} \Phi^{pq}e^q x = \sum_{p, q} \phi^{pq} x,
\]
and it is straightforward that $\phi$ is $\sP$-filtered if and only if  $\Phi$ obeys the condition
\begin{equation*}
\Phi^{pq}\neq 0 \text{ implies that } p\leq q.
\end{equation*}

Given a $\sP$-graded vector space $(V,\pi)$ and a subset $I\subset \sP$ we define
\[
\pi^I := \sum_{p\in I} \pi^p\quad\quad V^I:= \img \pi^I = \bigoplus_{p\in I} V^p.
\]

The space $V^I$ is a subspace of the underlying vector space $V$.  For a $\sP$-filtered linear map $\phi \colon (V,\pi)\to (W,\pi)$, we define $\phi^I\colon V\to V$ and $\Phi^I\colon V^I\to V^I$ via
\begin{equation}
\phi^I:=\pi^I\circ \phi\circ \pi^I 
\quad\quad \text{and} \quad\quad
\Phi^I:=e^I\circ \phi\circ \iota^I\colon V^I\to W^I .
\end{equation}

The next result is immediate.
\begin{prop}
If $\phi\colon (U,\pi)\to (V,\pi)$ and $\psi\colon (V,\pi)\to (W,\pi)$ are $\sP$-filtered maps and $I\subset \sP$ then
\[
\phi^I = \sum_{\substack{p\leq q\\ p,q\in I}} \phi^{pq}.
\]
Moreover, if $I\subset \sP$ is convex then $
\psi^I\circ \phi^I = (\psi\circ \phi)^I.$
\end{prop}

The above result enables the definition of the following family of forgetful functors, parameterized by the convex sets of $\sP$.
\begin{defn}\label{def:grad:forget}
{\em
We denote by $\frU\colon \bGVec(\sP)\to \bVec$ the (forgetful) functor which forgets the grading, i.e., $\frU(V,\pi)=V$ and $\frU(\phi)=\phi$.
Let $I\subset P$ be a convex set.  The forgetul functor $\frU^I\colon \bGVec(\sP)\to \bGVec(I)$ is defined via
\[
\frU^I\big((V,\pi)\big) := (V^I,\{\pi^p\}_{p\in I}).
\]
For $\phi\colon (V,\pi)\to (W,\pi)$, we define
\[
\frU^I(\phi) := \Phi^I=e^I\circ \phi\circ \iota^I\colon V^I\to W^I .
\]
}
\end{defn}

\begin{prop}\label{prop:map:filtgrad}
Let $(V,\pi)$ and $(W,\pi)$ be $\sP$-graded vector spaces.  A linear map $\phi\colon (V,\pi)\to (W,\pi)$ is $\sP$-filtered if and only if $\phi(V^a)\subset W^a$ for all $a\in \sJ(\sO(\sP))$.
\end{prop}
\begin{proof}
We start with showing that $\phi(V^a)\subset W^a$ if $\phi$ is $\sP$-filtered.  Let $a\in \sJ(\sO(\sP))$.  From Birkhoff's Theorem there exists $s\in \sP$ such that  $a=\downarrow \! s$. If $x\in V^a$ then $x=\pi^a(x) = \sum_{r\leq s} \pi^r (x)$.   Since $a$ is a down-set, if $p\leq s$ then $p\in a$ and $W^p\subset W^a$. Therefore
\[
\phi(x) = \sum_{p \leq q}\sum_{r\leq s}\phi^{pq}\pi^r(x) = \sum_{p\leq q\leq s} \phi^{pq}(x) \in W^a.
\]
Conversely, assume that $\phi(V^a)\subset W^a$ for all $a\in \sJ(\sO(\sP))$.  Suppose $\phi^{pq}(x)\neq 0$ for $p,q\in \sP$ and $x\in V$.  Let $b$ denote $\downarrow\! q$. Then $\pi^q(x)\in V^b$ and $\phi(\pi^q(x))\subset W^b$. We have $\phi^{pq}(x) = \pi^p(\phi(\pi^q(x)) \neq 0$, which implies $p\in b$.  Therefore $p \leq q$. 
\end{proof}

We write a $\Z$-indexed family of $\sP$-graded vector spaces as  $(V_\bullet,\pi_\bullet)=\{(V_n,\pi_n\})_{n\in \Z}$. For a fixed $p\in \sP$ there is a family of vector spaces $V_\bullet^p= \{V_n^p\}_{n\in \Z}$.

\subsection{Graded Chain Complexes}\label{sec:grad:ch}

The category $\bGVec(\sP)$ is additive but not abelian.  Following Section~\ref{sec:prelims:HA} we may form the category $\bCh(\bGVec(\sP))$ of chain complexes in $\bGVec(\sP)$.  An object $C$ of $\bCh(\bGVec(\sP))$ is a chain complex of $\sP$-graded vector spaces.  For short, we say that this is a {\em $\sP$-graded chain complex}.  The data of $C$ can be unpacked as the triple  $C=(C_\bullet,\partial_\bullet,\pi_\bullet)$ where:
\begin{enumerate}
    \item $(C_\bullet,\partial_\bullet)$ is a chain complex,
    \item $(C_n,\pi_n)$ is a $\sP$-graded vector space for all $n$, and
    \item $\partial_n\colon (C_n,\pi_n)\to (C_{n-1},\pi_{n-1})$ is a $\sP$-filtered linear map for each $n$.
\end{enumerate}

Typically we denote $C$ by $(C,\pi)$ to distinguish it as carrying a grading.  A morphism $\phi\colon (C,\pi)\to (C',\pi)$ is a chain map $\phi\colon (C,\partial)\to (C',\partial')$, such that $\phi_n\colon (C_n,\pi_n)\to (C'_n,\pi_n)$ is a $\sP$-filtered linear map for each $n$.  We call the morphisms of $\bCh(\bGVec(\sP))$ the {\em $\sP$-filtered chain maps}. 



Proceeding with our convention, we  define $\Delta_j^{pq}:=e^p_j\circ \partial_j\circ \iota_j^q\colon C_j^q\to C_{j-1}^p$.  Since $(C,\pi)$ is $\sP$-graded we have 
$$C_j=\bigoplus_{q\in \sP} C_j^q \quad\quad \text{and}\quad\quad C_{j-1}=\bigoplus_{p\in \sP} C_{j-1}^p.$$

The boundary operator $\partial_j\colon C_j\to C_{j-1}$ is equivalent to the matrix of maps $\{\Delta_j^{pq}\}_{p,q\in \sP}$.  The $\sP$-filtered condition of Eqn.~\eqref{eqn:map:filt} is equivalent to the condition that 
\begin{equation}
\Delta_j^{pq}\neq 0\implies p\leq q .
\end{equation}  

\begin{rem}
Viewing the boundary operator $\partial$ as a matrix of maps $\{\Delta^{pq}\}_{p,q\in \sP}$ is the origin of the term `connection matrix'.
\end{rem}

It follows from Proposition~\ref{prop:prelims:induce} that the forgetful functor $\frU\colon \bGVec(\sP)\to \bVec$ induces a (forgetful) functor $\frU_\bCh\colon \bChG\to \bCh(\bVec)$, where
\[
\frU_\bCh(C,\pi) = \big(\{\frU(C_n)\}_{n\in \Z},\{\frU(\partial_n:C_n\to C_{n-1})\}_{n\in \Z}\big) = 
\big(
\{ C_n\}_{n\in \Z},\{\Delta_n\}_{n\in\Z}\big).
\]
Similarly, for a convex set $I\subset \sP$ the functor $\frU^I\colon \bGVec(\sP)\to \bGVec(I)$ induces a functor $\frU_\bCh^I\colon \bChG\to \bCh(\bGVec(I))$.  

Given a convex set $I\subset \sP$ and the forgetful functor $\frU\colon \bCh(\bGVec(I))\to \bCh(\bVec)$ it is often useful to consider the composition 
\[
\frU_\bCh\circ \frU^I_\bCh\colon \bChG\to \bCh(\bVec).
\]
Unpacking Definition~\ref{def:grad:forget} shows that $\frU_\bCh\circ \frU^I_\bCh$ may be written simply as
\[
\frU_\bCh\circ \frU^I_\bCh(C,\pi)=(C^I,\Delta^I),
\quad\quad 
\frU_\bCh\circ \frU^I_\bCh(\phi)=\Phi^I\colon C^I\to C^I.
\]

\begin{prop}
Let $(C,\pi)$ be a $\sP$-graded chain complex.  If $a\in \sO(\sP)$, i.e., $a$ is a down-set of $\sP$, then $(C^a,\Delta^a)=\frU\circ\frU^a(C,\pi)$ is a subcomplex of $C$.
\end{prop}
\begin{proof}
If $a\in \sO(\sP)$ the fact that $\partial$ is $\sP$-filtered implies that $\partial( \bigoplus_{p\in a} C^p) \subseteq \bigoplus_{p\in a} C^p$.  Moreover $\Delta^a = e^p \circ \partial\circ  \iota^p =  \partial|_{C^a}$.  Therefore $(C^a,\Delta^a)$ is a subcomplex of $C$. 
\end{proof}





\subsection{The Subcategory of \Cyclic{} Objects}\label{sec:grad:cyclic}
\begin{defn}
\label{defn:grad:cyclic}
{\em 
A $\sP$-graded chain complex $(C,\pi)$ is said to be \emph{\cyclic} if for each $j\in \Z$ and $p\in \sP$
\begin{equation}\label{eqn:cyclic}
    \partial_j^{pp} = 0.
\end{equation}
Equivalently, if for each $j\in \Z$ 
    \begin{equation}
    \partial_j = \sum_{p<q} \partial_j^{pq}.
    \end{equation}  
 The \cyclic{} objects form the full subcategory $\bChGz\subset \bChG$, called the {\em subcategory of \cyclic{} objects}.
}
\end{defn}


\begin{rem}
In~\cite{fran}, a boundary operator $\partial_j$ which obeys condition \eqref{eqn:cyclic}  is called \emph{strictly upper triangular} with respect to $\sP$.
\end{rem}


\begin{prop}\label{prop:cyclic}
If $(C,\pi)$ is \cyclic{}, then $u^p(C,\pi)=(C_\bullet^p,\Delta_\bullet^{pp})$ is a \trivial{} chain complex (see Definition~\ref{def:homology}) for any $p\in \sP$. Moreover, for any $j\in \Z$ 
\[
C_j = \bigoplus_{p\in \sP} H_j(C_\bullet^p,\Delta_\bullet^{pp}).
\]
\end{prop}
\begin{proof}
If $(C,\pi)$ is \cyclic{} then the boundary operators $\Delta_j^{pp}=0$ for all $j\in \Z$ by definition.  Therefore $H_j(C^p_\bullet,\Delta^{pp})=C_j^p$.  Finally, since $(C,\pi)$ is $\sP$-graded we have that
\[
C_j = \bigoplus_{p\in \sP} C_j^p = \bigoplus_{p\in \sP} H_j(C_\bullet^p,\Delta_\bullet^{pp}).
\tag*{\qedhere}
\]
\end{proof}

Proposition~\ref{prop:cyclic} implies that $\partial_j\colon C_j\to C_{j-1}$ may be regarded as a $\sP$-filtered map on homology:
\begin{equation}\label{eqn:cm}
\partial_j\colon \bigoplus_{p\in \sP} H_j(C_\bullet^p,\Delta_\bullet^{pp}) \to \bigoplus_{p\in \sP} H_{j-1}(C_\bullet^p,\Delta_\bullet^{pp}).
\end{equation}

In the context of Conley theory,  Eqn.~\eqref{eqn:cm} implies that $\partial_j$ is a boundary operator on Conley indices.

\begin{rem}
The significance of Eqn.~\eqref{eqn:cm} is that in this form the nonzero entries in the boundary operator relate to connecting orbits.
\end{rem}

\begin{ex}
{\em
Let $X$ be a closed manifold and $\varphi\colon \R\times X\to X$ a Morse-Smale gradient flow.  The set $\sP$ of fixed points is partially ordered by the flow and there is a poset morphism $\mu\colon \sP\to \N$ which assigns each $p$ its Morse index, i.e., the dimensionality of its unstable manifold.   The associated Morse-Witten complex may be written $$C_\bullet({X,\varphi}) = \bigoplus_{p\in \sP} C^p_\bullet,$$ where $C^p_\bullet$ is the \trivial{} chain complex in which the only nonzero chain group is in dimension $\mu(p)$, and $C^p_{\mu(p)}=\K$.  The boundary map $\Delta$ is defined using trajectories~\cite{floer,robbin:salamon2}.  It is thus $\sP$-filtered. In particular, when $\K=\Z_2$ the entry $\Delta_{qp}$ counts the number of flow lines from $q$ to $p$ modulo two.  It is a classical result that the homology $H_\bullet(C({X,\varphi}))$ is isomorphic to the singular homology of $X$.
}
\end{ex}


\subsection{Graded Cell Complexes}\label{sec:grad:cell}

In applications, data often come in the form of a cell complex $\cX=(\cX,\leq,\kappa,\dim)$ graded by a partial order $\sP$.  This is codified in terms of an order preserving map $\nu\colon (\cX,\leq)\to \sP$. See~\cite{braids} for an example of how these structures arise in the context of computational dynamics.

\begin{defn}\label{defn:grad:assoc}
{\em
A {\em $\sP$-graded cell complex} is a cell complex $\cX=(\cX,\leq,\kappa,\dim)$ together with $\sP$  and a poset morphism $\nu \colon (\cX,\leq)\to \sP$.   The map $\nu$ is called the {\em grading}.   We denote by $\Cell(\sP)$ the collection of $\sP$-graded cell complexes.   For a $\sP$-graded cell complex $(\cX,\nu)$, the underlying set $\cX$ can be decomposed as  \[
\cX=\bigsqcup_{p\in \sP} \cX^p,\quad\quad \text{where } \cX^p:=\nu^{-1}(p) \, .
\]
For each $p$, the fiber $\cX^p$ together with the restriction of $(\leq,\kappa,\dim)$ to $\cX^p$ is a subcomplex of $\cX$. A $\sP$-graded cell complex $(\cX,\nu)$ determines an {\em associated $\sP$-graded chain complex} $(C_\bullet(\cX),\pi^\nu)$ (see Section~\ref{sec:grad:ch}) where for any $j\in \Z$
\[
C_j(\cX) = \bigoplus_{p\in \sP} C_j(\cX^p) \, .
\]
 The projection maps $\pi^\nu = \{\pi_j^p\}$ project to the fibers of $\nu$, i.e.,
\[
\pi_j^p\colon C_j(\cX)\to C_j(\cX^p).
\]
The boundary operator 
\[
\partial_j\colon C_j(\cX)\to C_{j-1}(\cX)
\]
is $\sP$-filtered since $\nu$ is order-preserving; $
\kappa ( \xi,\xi')\neq 0$ implies that $\xi' \leq \xi$ which in turn implies  $\nu(\xi')\leq \nu(\xi)$.
The boundary operator $\partial_j$ can be written as an upper triangular matrix of maps $\{\Delta_j^{pq}\}$ where $\Delta_j^{pq}\colon C_j(\cX^q)\to C_{j-1}(\cX^p)$.  We denote by $\cC\colon \Cell(\sP)\to \bChG$ the assignment $(\cX,\nu)\mapsto (C(\cX),\pi^\nu)$.
}
\end{defn}

Akin to graded chain complexes, there is a notion of being \cyclic{}.
\begin{defn}
{\em
A $\sP$-graded cell complex $(\cX,\nu)$ is {\em \cyclic{}} if, for each $p\in \sP$, the fiber $\nu^{-1}(p)$ is a \trivial{} cell complex (see Definition~\ref{defn:cell:cyclic}).
}
\end{defn}

A routine verification shows that \cyclic{} $\sP$-graded cell complexes engender \cyclic{} $\sP$-graded chain complexes.  
\begin{prop}
If $(\cX,\nu)$ is a \cyclic{} $\sP$-graded cell complex, then the associated $\sP$-graded chain complex $(C(\cX),\pi^\nu)$ is \cyclic{}.
\end{prop}


\begin{ex}\label{ex:gradcomplex}
{\em
Consider $(\cX,\nu)$ and $(\cX',\nu')$ and $\sP=\{p,q,r\}$ of Example~\ref{ex:complex}.  It is a routine verification that $(\cX,\nu)$ and $(\cX',\nu')$ are both $\sP$-graded complexes.  In particular, $(\cX',\nu')$ is a \cyclic{} $\sP$-graded complex.   The underlying set $\cX$ decomposes as $\cX=\cX^p\sqcup \cX^r\sqcup \cX^q$.  More explicitly, 
\[
\cX = \{v_0,e_0,v_1,e_1,v_2\},\quad\quad \cX^p = \{v_0\},\quad\quad \cX^r = \{v_2\}, \quad\quad \cX^q = \{e_0,v_1,e_1\} .
\]
This decomposition is reflected in the algebra since each chain group $C_j(\cX)$ splits as  $$C_j(\cX)=C_j(\cX^p)\oplus C_j(\cX^r)\oplus C_j(\cX^q).$$  As in Definition~\ref{defn:grad:assoc} the boundary operator $\partial_j$ can be written as the $\sP$-filtered linear map (see Section~\ref{sec:grad:vs})
\[
\partial_j = 
\bordermatrix{    
             & C_j(\cX^p) & C_j(\cX^r) & C_j(\cX^q)  \cr
    C_{j-1}(\cX^p) & \Delta_j^{pp} & 0  & \Delta_j^{pq} \cr
    C_{j-1}(\cX^r) & 0 & \Delta_j^{rr} & \Delta_j^{rq}   \cr
    C_{j-1}(\cX^q) & 0 & 0 & \Delta_j^{qq}  }.
\]
In particular, $\Delta_1$ (the only nonzero differential) can be written as
\[
\partial_1 = 
\bordermatrix{    
                  & e_0 & e_1  \cr
              v_0 & 1 & 0  \cr
              v_1 & 1 & 1  \cr
              v_2 & 0 & 1  }
= \partial_1^{pq}+\partial_1^{rq}+\partial_1^{qq}
= \begin{pmatrix}
\Delta_1^{pq}\\
\Delta_1^{rq}\\
\Delta_1^{qq}
\end{pmatrix}.
\]
}
\end{ex}

The next two examples show how graded cell complexes arise in applied topology and topological data analysis.

\begin{ex}\label{ex:filt}
{\em
In applications the input is often a cell complex $\cX$ and a function $\bar \nu\colon \cX^+\to \R$ on top cells (see Definition~\ref{defn:topcell}).  For instance, image data is often a two dimensional cubical complex with greyscale values on pixels (2-cells). Let $(\sQ,\leq)$ be the totally ordered set where $\sQ:=\bar \nu(\cX^+)$ and $\leq$ is the total order inherited from $\R$. We may extend $\bar \nu$ to a grading $\nu\colon (\cX,\leq)\to \sQ$ via
\[
\cX\ni \xi \mapsto \min\{\bar \nu (\eta )\colon  \eta \in \st(\xi)\cap \cX^+\} \in \sQ.
\]
The map $\nu$ is a poset morphism since if $\xi \leq \eta$ then $\st(\eta)\subseteq \st(\xi)$.   As $(\cX,\nu)$ is a $\sQ$-graded cell complex we may consider the Birkhoff dual $\sO(\nu)\colon \sO(\sQ)\to {\mathsf \Sub}_{Cl}(\cX)$.  Since $\sQ$ is totally ordered, the collection $\{\sO(\nu)(a)\}_{a\in \sO(\sQ)}$ is a filtration of $\cX$.  This is the standard input for the topological data analysis pipeline.
}
\end{ex}

\begin{ex}
{\em 
Consider $(\R^n,\leq)$ where $\leq$ is given by $$(a_1,\ldots,a_n) \leq (b_1,\ldots,b_n) \iff a_i\leq b_i\quad \text{for all $i$}.$$ 
Let $(\sP,\leq)$ be a poset where $\sP\subseteq \R^n$ and the partial order $\leq$ is inherited from $\R^n$.  Let $(\cX,\nu)$ be a $\sP$-graded cell complex and consider $\sO(\nu)\colon \sO(\sP)\to \Sub_{Cl}(\cX,\leq)$.  In the theory of multi-parameter persistence~\cite{csz}, the collection $\{\sO(\nu)(a)\}_{a\in \sO(\sP)}$ of subcomplexes is called a {\em one-critical multi-filtration} of $\cX$, since any cell enters the lattice/multi-filtration at a unique minimal element with respect to the partial order on $\sO(\sP)$.  Namely, a cell $\xi$ enters the multi-filtration at $\downarrow\! \nu(\xi)$.  Multi-filtrations can be converted to one-critical multi-filtrations via the mapping telescope~\cite{csz}.
 }
\end{ex}

\subsection{Homotopy Category of Graded Complexes}

Recall the general construction of the homotopy category $\bK(\cA)$ for an additive category $\cA$ from Section~\ref{sec:prelims:HA}.  It follows from Section~\ref{sec:prelims:HA} that there is a homotopy category $\bKG$ of the category of $\sP$-graded chain complexes $\bCh(\bGVec(\sP))$.  To unpack this a bit, first recall the definition of $\sP$-filtered chain maps in Section~\ref{sec:grad:ch}. We say that two $\sP$-filtered chain maps $\phi,\psi\colon (C,\pi)\to (C',\pi)$ are {\em $\sP$-filtered chain homotopic} if there is a $\sP$-filtered chain contraction $\gamma\colon C\to C'$ such that $\phi_n-\psi_n = \gamma_{n-1}\circ \partial_n+\partial'_{n+1}\circ \gamma_n.$
We denote this by $\psi\sim_\sP \phi$.  The map $\gamma$ is called a {\em $\sP$-filtered chain homotopy} from $\phi$ to $\psi$.   A $\sP$-filtered chain map $\phi\colon (C,\pi)\to (C',\pi)$ is a {\em $\sP$-filtered chain equivalence} if there is a $\sP$-filtered chain map $\psi\colon (C',\pi)\to (C,\pi)$ such that $\psi\phi\sim_\sP \id_C$ and $\phi\psi\sim_\sP \id_{C'}$.  In this case we say that $(C,\pi)$ and $(C',\pi)$ are {\em $\sP$-filtered chain equivalent}.

Following Definition~\ref{def:prelim:HA:hocat}, the {\em homotopy category of $\sP$-graded chain complexes}, denoted by $\bK(\bGVec(\sP))$, is the category whose objects are $\sP$-graded chain complexes and whose morphisms are $\sP$-filtered chain homotopy equivalence classes of $\sP$-filtered chain maps.   There is a quotient functor $q\colon \bCh(\bGVec(\sP))\to \bK(\bGVec(\sP))$ which sends each $\sP$-graded chain complex to itself and each $\sP$-filtered chain map to its $\sP$-filtered chain homotopy equivalence class.




\begin{ex}\label{ex:homotopy}
{\em
The $\sP$-graded chain complexes $(C(\cX),\pi)$ and $(C(\cX'),\pi)$ of Example~\ref{ex:complex} are $\sP$-filtered chain equivalent via $\sP$-filtered chain maps
\begin{align*}
\phi\colon C(\cX) \to C(\cX')\quad \psi\colon C(\cX') \to C(\cX),
\end{align*}
and $\sP$-filtered chain homotopies 
\begin{align*}
\gamma\colon C(\cX)\to C(\cX)\quad \gamma'\colon C(\cX')\to C(\cX'),
\end{align*}
which are described below.  The nonzero differentials are

\begin{minipage}[t]{0.3\textwidth}
\begin{align*}
\partial_1 &= 
\bordermatrix{    
                  & e_0 & e_1  \cr
              v_0 & 1 & 0  \cr
              v_1 & 1 & 1  \cr
              v_2 & 0 & 1  }\\
\end{align*}
\end{minipage}
\begin{minipage}[t]{0.5\textwidth}
\begin{align*}
\partial_1' &= 
\bordermatrix{    
                  & e_0'   \cr
              v_0' & 1   \cr
              v_1' & 1   }.\\
\end{align*}
\end{minipage}

The nonzero parts of the chain maps $\phi$ and $\psi$ are as follows.

\begin{minipage}[t]{0.3\textwidth}
\begin{align*}
\psi_0 &= \bordermatrix{    
                  & v_0 & v_1 & v_2 \cr
              v_0' & 1 & 0 & 0  \cr
              v_1' & 0 & 1 & 1 }\\
\psi_1 &= \bordermatrix{    & e_0 & e_1  \cr
              e_0' & 1 & 0  } \\
\end{align*}
\end{minipage}
\begin{minipage}[t]{0.5\textwidth}
\begin{align*}
\phi_0 &= \bordermatrix{    
                  & v_0' & v_1' \cr
              v_0 & 1 & 0  \cr
              v_1 & 0 & 0  \cr 
              v_2 & 0 & 1  }\\
\phi_1 &= \bordermatrix{    
                  & e_0'   \cr
              e_0 & 1   \cr
              e_1 & 1 }\\ 
\end{align*}
\end{minipage}

In this case $\gamma'=0$. And the nonzero part of $\gamma$ is 

\begin{minipage}[t]{0.3\textwidth}
\begin{align*}
\gamma_0 &= \bordermatrix{    
                  & v_0 & v_1 & v_2 \cr
              e_0 & 0 & 0 & 0  \cr
              e_1 & 0 & 1 & 0 }.\\
\end{align*}
\end{minipage}


A lengthy but routine calculation shows that $\psi$ and $\phi$ are $\sP$-filtered chain maps and that $\phi\circ\psi = \gamma\circ \partial+\partial\circ \gamma$ and $\psi\circ\phi=\id$. 
}
\end{ex}

 

We can now introduce our definition of the connection matrix.  In particular, our definition of connection matrix rests on the homotopy-theoretic language.
 
 \begin{defn}\label{def:grad:cm} {\em
 Let $(C,\pi)$ be a $\sP$-graded chain complex.  A $\sP$-graded chain complex $(C',\pi)$ is a {\em Conley complex for} $(C,\pi)$ if 
 \begin{enumerate}
 \item $(C',\pi)$ is \cyclic{}, i.e., $\partial_j'^{pp}=0$ for all $p$ and $j$, and
 \item $(C',\pi)$ is isomorphic to $(C,\pi)$ in $\bKG$.
\end{enumerate}
If $(C',\pi)$ is a Conley complex for $(C,\pi)$ then we say the associated boundary operator $\partial'=\{\Delta'^{pq}\}_{p,q\in \sP}$ is a {\em connection matrix for} $(C,\pi)$.
} \end{defn}

\begin{rem}
With the definition in place, we make some remarks about  existence and uniqueness.
\begin{itemize}
    \item Given a $\sP$-graded chain complex $(C,\pi)$, a Conley complex $(C',\pi)$ for $(C,\pi)$ exists.  This follows from the proof of correctness of  Algorithm~\ref{alg:cm}.
    \item A classical issue in Conley theory is the non-uniqueness of the connection matrix~\cite{fran,atm}.  In our treatment of connection matrix theory using chain equivalence and homotopy categories we show that Conley complexes are unique up to isomorphism.  Thus a connection matrix is unique up to a similarity transformation in the sense that if one fixes a basis, then given two connection matrices $\Delta$ and $\Delta'$ there is a $\sP$-filtered chain map $\Phi$ such that $\Delta' = \Phi^{-1}\Delta\Phi$, cf.~\cite{atm}. See Remark~\ref{rem:grad:unique}.  
\end{itemize}
\end{rem}
 
\begin{ex}
{\em
Consider $(C(\cX),\pi)$ and $(C(\cX'),\pi)$ of Example~\ref{ex:complex}. A straightforward verification shows that $(C(\cX'),\pi)$ is \cyclic{} and an object of $\bChGz$ (recall Definition~\ref{defn:grad:cyclic}).  Moreover, from Example~\ref{ex:homotopy} we see that $(C(\cX'),\pi)$ and $(C(\cX),\pi)$ are isomorphic in $\bKG$.  Therefore $(C(\cX'),\pi)$ is a Conley complex for $(C(\cX),\pi)$ and $\partial'=\{\Delta^{pq}\}$ is a connection matrix for $(C(\cX),\pi)$.
}
\end{ex}

 Proposition~\ref{prop:cat:subquo} allows for the following definition.
 
 \begin{defn}
 {\em
 Let $\bKGz$ denote the full subcategory of $\bKG$ whose objects are the objects of $\bChGz$. Then 
 \[
 \bKGz=\bChGz/\! \sim_\sP
 \]
 and there is a quotient functor $q\colon\bChGz \to \bKGz$.
 }
 \end{defn}
 
  \begin{prop}\label{prop:grad:cmiso}
 \Cyclic{} $\sP$-graded chain complexes are isomorphic in $\bChG$ if and only if they are $\sP$-filtered chain equivalent. 
 
 \end{prop}
  \begin{proof}
  The `only if' direction is immediate: set the homotopies $\gamma=\gamma'=0$.    Conversely, if $(C,\pi)$ and $(C',\pi)$ are $\sP$-filtered chain equivalent then there are $\sP$-filtered chain equivalences $\phi\colon (C,\pi)\to (C',\pi)$ and $\psi\colon (C',\pi)\to (C,\pi)$, and a  $\sP$-filtered chain homotopy $\gamma\colon (C,\pi)\to (C,\pi)$, such that for each $n$
  \[
  \psi_n \phi_n - \id_n = \gamma_{n-1}\partial_n+\partial_{n+1}\gamma_n.
  \]
  It follows from Proposition~\ref{prop:grad:comp} that for each $p\in \sP$
  \begin{align*}
      \psi_n^{pp}\phi_n^{pp} - \id_n^{pp} =
      (\psi_n\phi_n - \id_n)^{pp} =
      (\gamma_{n-1}\partial_n+\partial_{n+1}\gamma_n)^{pp} =
      \gamma_{n-1}^{pp}\partial_n^{pp} + \partial_{n+1}^{pp}\gamma_n^{pp} = 0.
  \end{align*}

  Therefore each $\phi_n^{pp}$ is an isomorphism with inverse $\psi_n^{pp}$.   It follows from elementary matrix algebra that $\phi=\sum_{p\leq q} \phi^{pq}$ is an isomorphism.  
\end{proof}
 
\begin{cor}\label{prop:grad:conserv}
The quotient functor $q\colon \bChGz\to \bKGz$ is a conservative functor.
\end{cor}

\begin{rem}\label{rem:grad:unique}

Proposition~\ref{prop:grad:cmiso} addresses non-uniqueness of the connection matrix in our formulation.  In particular, the connection matrix is unique up to a choice of basis.  Non-uniqueness manifests as a change of basis.
See~\cite{braids} for some applications where non-uniqueness arises.  See~\cite{fran,atm,reineck} for more discusion of non-uniqueness in connection matrix theory. 
\end{rem}

\subsection{Examples}\label{sec:grad:ex}

\begin{ex}[Computing Homology]\label{ex:homology}
{\em 
Consider the situation where $\cX$ is a cell complex and one is interested in computing the homology $H(\cK)$ of a closed subcomplex $\cK\subset \cX$.  Connection matrix theory applies to this situation in the following fashion.  Let $\sQ = \{0,1\}$ be the poset with $0\leq 1$.  Define the order-preserving map $\nu\colon (\cX,\leq)\to \sQ$ via
\[
  \nu(x) =
  \begin{cases}
          0 & \text{$x\in \cK$ }\\
          1 & \text{$x\in \cX\setminus \cK$ }
  \end{cases}
\]
The pair $(\cX,\nu)$ is a $\sP$-graded cell complex and $(C(\cX),\pi^\nu)$ is the associated $\sP$-graded chain complex (see Definition~\ref{defn:grad:assoc}).  We have $\cK=\nu^{-1}(0)$ and for any $j\in \Z$ 
\[
C_j(\cX)=C_j(\cX^0)\oplus C_j(\cX^1) = C_j(\cK)\oplus C_j(\cX\setminus \cK).
\]
Let $(D,\pi)$ be a Conley complex for $(C(\cX),\pi^\nu)$.  Then for each $j\in \Z$
\[
D_j = D_j^0\oplus D_j^1 .
\]
Moreover, as $(D,\pi)$ is $\sP$-graded, the boundary operator $\partial_j\colon D_j\to D_{j-1}$ can be written as the matrix
\[
\partial_j = 
\bordermatrix{    
           & D_j^0 & D_j^1  \cr
 D_{j-1}^0 & \Delta_j^{00} & \Delta_j^{01}  \cr
D_{j-1}^1 & 0 & \Delta_j^{11}  } .
\]
The first condition in the definition of a Conley complex (Definition~\ref{def:grad:cm}) gives that $(D,\pi)$ is \cyclic{}.  Therefore $\Delta^{00}_j=0$ and $\Delta_j^{11}=0$.  The second condition in the definition implies that there is a $\sP$-filtered chain equivalence $\phi\colon (D,\pi)\to (C(\cX),\pi^\nu)$.  We can write $\phi_j\colon D_j\to C_j(\cX)$ as
\begin{align*}
\partial_j = 
\bordermatrix{    
           & D_j^0 & D_j^1  \cr
 C_{j-1}(\cX^0) & \Phi_j^{00} & \Phi_j^{01}  \cr
C_{j-1}(\cX^1) & 0 & \Phi_j^{11}  } 
.
\end{align*}
It follows that the map 
\[
\Phi_\bullet^{00}\colon D_\bullet^0\to C_\bullet(\cX^0)
\]
is a chain equivalence.  Thus for all $j\in \Z$ $$H_j(C_\bullet(\cX^0))\cong H_j(D^0_\bullet,\Delta^{00}_\bullet) = D^0_j,$$
where the last equality follows from  Proposition~\ref{prop:cyclic} and the fact that $(D,\pi)$ is \cyclic{}.  
}
\end{ex}

\begin{ex}[Long Exact Sequence of a Pair]\label{ex:les}
{\em 
Consider $\sQ=\{0,1\}$ from Example~\ref{ex:homology} and $(\cX,\nu)$, $(\cX',\nu')$ and $\sP$ from Example~\ref{ex:complex}.  Let $\rho\colon \sP\to \sQ$  be the epimorphism given below.
\[
  \rho(x) =
  \begin{cases}
          0 & \text{$x=p$ }\\
          0 & \text{$x=r$ } \\
          1 & \text{$x=q$ } \\
  \end{cases}
\]
Let $\mu\colon \cX\to \sQ$ be the composition $\mu=\rho\circ \nu$ so that $(\cX,\mu)$ is a $\sQ$-graded cell complex and  $\cX$ partitions as $\cX=\cX^0\sqcup \cX^1$, where $\cX^i=\mu^{-1}(i)$. $\cX^0$ is a closed subcomplex and $\cX^1$ is an open subcomplex. There is a short exact sequence
\[
0\to C(\cX^0)\to C(\cX)\to C(\cX^1)\to 0.
\]
In the associated long exact sequence on homology all homology groups are zero aside from the following:
\[
\ldots \to H_1(\cX^1)\xrightarrow{\delta} H_0 (\cX^0)\to H_0(\cX)\to \ldots
\]
A straightforward computation shows that this sequence is 
\[
\ldots \to \Z_2
\xrightarrow{
\begin{pmatrix}
1\\
1
\end{pmatrix}
} \Z_2\oplus\Z_2 \xrightarrow{\begin{pmatrix}
1 & 1
\end{pmatrix}} \Z_2\to \ldots
\]
Consider the $\sQ$-graded complex $(\cX',\mu')$ where $\mu'=\rho\circ\nu'$.   A quick verification shows that the chain map $\phi\colon C(\cX')\to C(\cX)$ of Example~\ref{ex:homotopy} is a $\sQ$-filtered chain equivalence. Therefore $(C(\cX'),\pi^{\mu'})$ is a Conley complex for $(C(\cX),\pi^\mu)$. The map $\phi$ induces a morphism of short exact sequences, given in the following diagram.
\[
\begin{tikzcd}[column sep=small]
0 \ar[r] & C(\cX^0) \ar[d,"\Phi^{00}"] \ar[r]& C(\cX) \ar[d,"\phi"] \ar[r] & C(\cX^1) \ar[d,,"\Phi^{11}"] \ar[r] & 0\\
0 \ar[r]& C(\cX'^0) \ar[r]& C(\cX') \ar[r]& C(\cX'^1)\ar[r] & 0
\end{tikzcd}
\]

The morphism of short exact sequences induces a morphism of the long exact sequences.  The fact that $\phi$ is a $\sQ$-filtered chain equivalence implies that the induced maps on homology are isomorphisms.
\[
\begin{tikzcd}[column sep=small]
\ldots \ar[r] & H_1(\cX^1)  \ar[d,"\cong"]\ar[r,"\delta"] & H_0(\cX^0) \ar[d,"\cong"] \ar[r] & H_0(\cX)  \ar[d,"\cong"]\ar[r] & \ldots\\
\ldots \ar[r] & H_1(\cX'^1)  \ar[r,"\delta'"] & H_0(\cX'^0) \ar[r] & H_0(\cX')\ar[r] & \ldots\\
& C_1(\cX') \ar[u,"\id"]\ar[r,"\Delta_1'^{01}"] & C_0(\cX'^0) \ar[r]\ar[u,"\id"]  & C_0(\cX')  \ar[u,"\id"]
\end{tikzcd}
\]


This discussion shows that in the setting of a $\sQ$-graded complex -- where $\sQ=\{0,1\}$ -- the connection matrix $\Delta'$ is the connecting homomorphism of the long exact sequencen of a pair.
}
\end{ex}

\section{Reductions}\label{sec:reductions}

In this section we introduce \emph{reductions} -- the theoretical tool which formalizes our method of computing Conley complexes.  In particular, reductions formalize the use of discrete Morse theory; see Example~\ref{ex:red:dmt} and Section~\ref{sec:reductions:gmt}. 
In Section~\ref{sec:alg} we present two algorithms based on discrete Morse theory which build on the theory discussed in this section.

First, we review the tools for the chain complexes and the category $\bCh(\bVec)$;  we then proceed to the graded version within the category $\bChG$.  
Although we proceed in this order, the material of Section~\ref{sec:red:chain} can be deduced from the results of Section~\ref{sec:red:grad} by considering chain complexes graded over a poset consisting of a single element, i.e., $\sQ=\{0\}$. Even still, it is worthwhile to setup the theory of reductions for chain complexes explicitly as we will cite these results for the proof of correctness for Algorithm~\ref{alg:hom} (\textsc{Homology}).


\subsection{Reductions of Chain Complexes}\label{sec:red:chain}

In computational homological algebra, one often finds a simpler representative with which to compute homology.  A model for this is the notion of {\em reduction}, which is a particular type of chain homotopy equivalence.  The notion also goes under the moniker {\em strong deformation retract} or sometimes chain contraction~\cite{sko2}.\footnote{We previously introduced the term {\em chain contraction} in Section~\ref{sec:prelims:HA} which agrees with~\cite{weibel}.  This idea should not be confused with reduction.}  It appears in~\cite{emac} and in homological perturbation theory~\cite{barnes:lambe} and forms the basis for effective homology theory~\cite{rubio:sergeraert} and algebraic Morse theory~\cite{sko,sko2}. Our exposition of reductions primarily follows the preprint~\cite{rubio:sergeraert}. Roughly speaking, a reduction is a method of data reduction for a chain complex without losing any information with respect to homology.

\begin{defn}
{\em A {\em reduction} is a pair of chain complexes and triple of maps, often visualized as
\[
\begin{tikzcd}
C 
\arrow[out=60,in=120,distance=5mm,swap,"\gamma"]
\arrow[r, shift left=1,"\psi"]
& M \arrow[l, shift left=1,"\phi"],
\end{tikzcd}
\]
where $\phi,\psi$ are chain maps and $\gamma$ is a chain contraction, satisfying the identities:
\begin{enumerate}
\item $\psi\phi = \id_M$\label{cont:cond1},
\item $\phi\psi = \id_C-(\gamma \partial +\partial \gamma)$\label{cont:cond2},
\item $\gamma^2 = \gamma\phi = \psi\gamma = 0$.\label{cont:cond3}
\end{enumerate}
}
\end{defn}

  From the definition it is clear that $\phi$ is a monomorphism and $\psi$ is an epimorphism.  In applications, one calls $M$ the {\em reduced complex}.  When reductions arise from algebraic-discrete Morse theory $M$ is sometimes called the {\em Morse complex}.  The point is that one wants $|M|\ll |C|$, then one may compute $H(M)$ (and thus $H(C)$) efficiently.  Notice that by using~\eqref{cont:cond3}, an application $\gamma$ on the left of~\eqref{cont:cond2} gives:
 \begin{align}\label{eqn:cont:gamma}
 0 = (\gamma \phi) \psi = \gamma ( \id_C-\gamma\partial -\partial \gamma) = \gamma - \gamma\partial \gamma.
 \end{align}
 This equation is axiomized as the condition for a degree 1 map (see Definition~\ref{defn:degree1}) called a  {\em splitting homotopy}.  
\begin{defn}
{\em
Let $C$ be a chain complex.  A {\em splitting homotopy} is a degree 1 map $\gamma\colon C\to C$ such that $\gamma ^2=0$ and $\gamma \partial \gamma = \gamma$.
}
\end{defn}

The upshot is that reductions can be obtained from splitting homotopies. The conditions $\partial ^2=\gamma^2=0$ and $\gamma \partial \gamma  = \gamma$ ensure that $\gamma \partial +\partial \gamma$ is idempotent.  Therefore $\rho=\id_C-(\gamma \partial+\partial\gamma)$ is a projection onto the subspace complementary  to $\img(\gamma\partial+\partial\gamma)$.  Since $\rho$ is a projection, there is a splitting of $C$ into subcomplexes:
\[
C=\ker\rho\oplus \img\rho.
\] The image $(M,\partial_M)=(\img\rho,\partial|_{\img\rho})$ is a subcomplex of $C$.  We have the following reduction:
\begin{equation}\label{reduction:homotopy}
\begin{tikzcd}
C 
\arrow[out=60,in=120,distance=5mm,swap,"\gamma"]
\arrow[r, shift left=1,"\rho"]
& M \arrow[l, shift left=1,"i"] .
\end{tikzcd}
\end{equation}

We can calculate the differential $\partial_M$ via 
$$\partial_M = \partial\rho = \partial(\id_C-(\gamma \partial+\partial\gamma )) = \partial-\partial\gamma \partial + \partial\partial\gamma = \partial-\partial\gamma \partial.$$

Finally, it is straightforward that the remaining identities $\gamma i=\rho\gamma = 0$ are easily verified. Furthermore, $\ker\rho$ is a subcomplex of $C$ and $\gamma|_{\ker\rho}$ is a chain contraction, since $\id_{\ker\rho} = \partial\gamma +\gamma \partial$.  This implies that $\ker\rho$ is acyclic, i.e., $H_\bullet(\ker\rho) = 0$.  It is known that reductions and splitting homotopies are (up to isomorphism) in bijective correspondence, as recorded in the next result.  We include a proof here for completeness.

\begin{prop}\label{prop:ch:bij}
Reductions and splitting homotopies are in bijective correspondence, up to isomorphism.
\end{prop}
\begin{proof}
From~\eqref{eqn:cont:gamma} we have that any chain homotopy in a reduction is a splitting homotopy.  Moreover, any splitting homotopy can be put into the reduction in~\eqref{reduction:homotopy}.   Let $\gamma$ be a splitting homotopy and consider two reductions:
\[
\begin{tikzcd}
M  \arrow[r, shift left=1,"i"]  &
\arrow[l, shift left=1,"\rho"]
C 
\arrow[out=60,in=120,distance=5mm,swap,"\gamma"]
\arrow[r, shift left=1,"\rho'"]
& M \arrow[l, shift left=1,"i'"] .
\end{tikzcd}
\] 

A routine computation using the conditions~\eqref{cont:cond1}--\eqref{cont:cond3} shows that the compositions $\rho\circ i'$ and $\rho'\circ i$ are inverses.  Therefore $M$ and $M'$ are chain isomorphic.
\end{proof}

\begin{ex}\label{ex:red:dmt}
{\em
Let $\cX$ be a cell complex and let $(A,w\colon Q\to K)$ be an acyclic partial matching, see Section~\ref{prelims:dmt}.  By Proposition~\ref{prop:matchinghomotopy} there exists a unique splitting homotopy $\gamma$.  From Theorem~\ref{thm:focm:red} defining the maps
$$\psi:=\pi_A\circ (\id_\cX-\partial \gamma) \quad\quad  \phi:= (\id_\cX-\gamma \partial)\circ \iota_A \quad\quad \partial^A:= \psi\circ \partial\circ \phi .$$
leads to a reduction:
\begin{equation}\label{reduction:dmt}
\begin{tikzcd}
C_\bullet(\cX)
\arrow[out=60,in=120,distance=5mm,swap,"\gamma"]
\arrow[r, shift left=1,"\psi"]
& (C_\bullet(A),\partial^A) \arrow[l, shift left=1,"\phi"] 
\end{tikzcd}
\end{equation}
Notice that this is a different reduction than the one defined in Diagram~\eqref{reduction:homotopy}.  However, we have $(C_\bullet(\cA),\partial^\cA) \cong (M,\partial_M)$ from Proposition~\ref{prop:ch:bij}.  In contrast to Diagram~\eqref{reduction:homotopy}, using the reduction of Diagram~\eqref{reduction:dmt} has the property that the Morse complex is comprised of critical cells of the matching.   
}
\end{ex}

\begin{defn}
{\em
  We say a reduction is {\em \strict{} } if the reduced complex $M$ is \trivial{}.  We say a splitting homotopy $\gamma$ is {\em perfect} if $\partial= \partial\gamma \partial $. 
}
\end{defn}

\begin{prop}\label{prop:ch:strictbij}
\Strict{} reductions and perfect splitting homotopies are in bijective correspondence.
\end{prop}
\begin{proof}
If the reduction is \strict{}, then $\partial_M=0$.  Thus $\partial i\rho = (i\partial_M)\rho = 0$.  By hypothesis $i\rho = \id_C-\partial\gamma -\gamma \partial $.  Application of $\partial $ to both sides yields $$0= \partial(i\rho) = \partial(\id_C-\partial\gamma -\gamma \partial) = \partial-\partial\gamma \partial .$$
Conversely, if $\gamma$ is perfect and $M=\img(\rho)$, then  the differential $\partial_M$ is calculated as $$\partial_M= \partial-\partial\gamma \partial =0 .$$ 
Therefore $M$ is \trivial{} and the reduction is \strict{}.
\end{proof}

A perfect splitting homotopy implies $\img(\rho)\cong H_\bullet (C)$.  This allows the homology to be read from the reduction without computation.  In addition, we have $\partial i = i\partial_M = 0$.  Therefore $\img(i)\subset \ker\partial$ and the map $i\colon M\to \ker \partial$ gives representatives for the homology in $C$. In the case of fields, perfect splittings always exist.\footnote{In fact, this is a Corollary of Algorithm~\ref{alg:hom}.}  This implies that a chain complex $C$ and its homology $H_\bullet(C)$ always fit into a reduction. Moreover any reduction where $C$ is a \trivial{} complex is trivial in the sense that the two complexes are isomorphic.

\begin{prop}\label{prop:cont:cyclic}
Let $C$ be a \trivial{} chain complex. Any reduction
\[
\begin{tikzcd}
C 
\arrow[out=60,in=120,distance=5mm,swap,"\gamma"]
\arrow[r, shift left=1,"\psi"]
& M \arrow[l, shift left=1,"\phi"] 
\end{tikzcd}
\]
is \strict{}.  Moreover, we have $M\cong C$.
\end{prop}
\begin{proof}
We first show that the reduction is minimal.  This follows since $\partial_M=\partial_M(\psi\phi) = \psi\partial\phi = 0$.  We now show that $M\cong C$.  We have $\psi \circ \phi  = \id_M$.  If $C$ is \trivial{} then $\phi\circ\psi = \id_C-(\gamma \partial + \partial\gamma) = \id_C$.
\end{proof}

In this sense, the homology $H_\bullet(C)$ is the algebraic core of a chain complex and the minimal representative for $C$ with respect to reductions.  This result will have an analogue in the graded case.  Finally, we show that reductions compose.  
\begin{prop}\label{prop:ch:seq}
Given the sequence of reductions:
\[
\begin{tikzcd}
C 
\arrow[out=60,in=120,distance=5mm,swap,"\gamma"]
\arrow[r, shift left=1,"\psi"]
& M \arrow[l, shift left=1,"\phi"] 
\arrow[out=60,in=120,distance=5mm,swap,"\gamma'"]
\arrow[r, shift left=1,"\psi'"]
& M' \arrow[l, shift left=1,"\phi'"] 
\end{tikzcd}
\]
there is a reduction 
\[
\begin{tikzcd}
C 
\arrow[out=60,in=120,distance=5mm,swap,"\gamma''"]
\arrow[r, shift left=1,"\psi''"]
& M' \arrow[l, shift left=1,"\phi''"] 
\end{tikzcd}
\]
with the maps given by the formulas
\[
\phi'' = \phi \circ \phi ' \quad\quad \psi'' = \psi'\circ  \psi \quad\quad \gamma'' = \gamma + \phi \circ \gamma '\circ \psi .
\]
\end{prop}
\begin{proof}
Elementary computations show that 
\[
\rho''\circ i'' = \id_{M''}\quad\text{and}\quad i'' \circ \rho'' = \id_C-(\partial\gamma'' + \gamma'' \partial).
\]
Conditions~\eqref{cont:cond1}--\eqref{cont:cond3} follow from the same conditions for $\gamma$ and $\gamma'$.
\begin{align*}
(\gamma'')^2 &= (\gamma+i\gamma'\rho)(\gamma+i\gamma'\rho) = \gamma^2 +  (\gamma i)\gamma'\rho + i\gamma'(\rho\gamma)  + i\gamma' (\rho i)\gamma'\rho = i (\gamma'\gamma' )\rho = 0\\
\gamma'' \circ i'' &= (\gamma + i\gamma'\rho)(i\circ i') = (\gamma i)i'  + \rho' (\rho i)\gamma '\rho =  (\rho'\gamma) '\rho = 0\\
\rho''\circ \gamma'' &= (\rho'\rho)(\gamma+i\gamma'\rho) = \rho' (\rho\gamma) + \rho'( \rho i) \gamma'\rho = (\rho'\gamma' )\rho = 0 .
\tag*{\qedhere}
\end{align*}
\end{proof}



An inductive argument gives the following result.

\begin{prop}\label{prop:cont:tower}
Given a tower of reductions
\begin{equation*}
    \begin{tikzcd}
C
\arrow[out=60,in=120,distance=5mm,swap,"\gamma_0"]
\arrow[r, shift left=1,"\psi_0"]
& M_0 \arrow[l, shift left=1,"\phi_0"] 
\arrow[out=60,in=120,distance=5mm,swap,"\gamma_1"]
\arrow[r, shift left=1,"\psi_1"] 
& \ldots \arrow[l, shift left=1,"\phi_1"]
\arrow[r, shift left=1,"\psi_{n-1}"] 
& M_{n-1} \arrow[l, shift left=1,"\phi_{n-1}"] 
\arrow[out=60,in=120,distance=5mm,swap,"\gamma_n"]
\arrow[r, shift left=1,"\psi_n"] 
& M_n \arrow[l, shift left=1,"\phi_n"] ;
\end{tikzcd}
\end{equation*}
\begin{enumerate}
\item there is a reduction 
\begin{equation}\label{eq:ch:tower}
\begin{tikzcd}
C
\arrow[out=60,in=120,distance=5mm,swap,"\Gamma"]
\arrow[r, shift left=1,"\Psi_n"]
& M_n\arrow[l, shift left=1,"\Phi_n"] 
\end{tikzcd}
\end{equation}
with maps given by the formulas
\[
 \Psi_m = \prod_{i=0}^m \psi_i   \quad\quad   \Phi_m = \prod_{i=0}^m \phi_i 
\quad\quad 
\Gamma = \gamma_0 + \sum_{i=0}^{n-1} \Phi_i\circ \gamma_{i+1}\circ \Psi_i .
\]
\item $\Gamma$ is a splitting homotopy and $\Gamma$ is perfect if any $\gamma_i$ is perfect.

\end{enumerate}
\end{prop}
\begin{proof}
Part (1) follows from Proposition~\ref{prop:ch:seq} and an inductive argument.  Given~\eqref{eq:ch:tower} the fact that $\gamma$ is a splitting homotopy follows from the proof of Proposition~\ref{prop:ch:bij}.  If $\gamma_i$ is perfect, then $M_i$ is \trivial{} by Proposition~\ref{prop:ch:strictbij}.  Thus $M_j$ is \trivial{} for $j\geq i$ by Proposition~\ref{prop:cont:cyclic}. In particular $M_n$ is \trivial{} and~\eqref{eq:ch:tower} is \strict{}.  Therefore $\gamma$ is a perfect splitting homotopy by Proposition~\ref{prop:ch:strictbij}.
\end{proof}




\subsection{Graded Reductions}\label{sec:red:grad}

The graded version of a reduction is obtained by porting the definition to the category of $\sP$-graded chain complexes.

\begin{defn}
{\em 
A {\em $\sP$-graded reduction} is a pair of $\sP$-graded chain complexes and triple of $\sP$-filtered maps
\begin{center}
\begin{tikzcd}
(C,\pi)
\arrow[out=60,in=120,distance=5mm,swap,"\gamma"]
\arrow[r, shift left=1,"\psi"]
& (M,\pi) \arrow[l, shift left=1,"\phi"] 
\end{tikzcd}
\end{center}
where $\phi,\psi$ are chain maps and $\gamma$ is a chain contraction, satisfying the identities:
\begin{enumerate}
\item $\psi\phi = \id_M$,
\item $\phi\psi = \id_C-(\gamma \partial+\partial\gamma)$,
\item $\gamma^2 = \gamma\phi = \psi\gamma = 0.$
\end{enumerate}
 An $\sP$-graded reduction is {\em \strict{}} if $(M,\pi)$ is \cyclic{}. 
 }
\end{defn}

 \begin{defn}
 {\em
 A \emph{$\sP$-filtered splitting homotopy} is a  degree 1 map $\gamma\colon (C,\pi)\to (C,\pi)$ such that $\gamma^2=0$ and $\gamma\partial\gamma = \gamma$. A $\sP$-filtered splitting homotopy is \emph{perfect} if $\partial^{pp} = \partial^{pp}\gamma^{pp}\partial^{pp}$ for each $p\in \sP$.
 }
 \end{defn}    

  
 Again, one may define $\rho=\id_C-(\gamma\partial+\partial\gamma)$ and $M=\img(\rho)$.  Then $M$ is a $\sP$-graded subcomplex of $(C,\pi)$, $p\circ i = \id_M$ and $i\circ p = \id_C-(\gamma\partial+\partial\gamma)$.  
  


\begin{prop}\label{prop:grad:bij}
$\sP$-filtered splitting homotopies and $\sP$-graded reductions are in bijective correspondence.  Furthermore, perfect $\sP$-filtered splitting homotopy and \strict{} $\sP$-graded reductions are in bijective correspondence.
\end{prop}
\begin{proof}
The proof of the first result follows the proof of Proposition~\ref{prop:ch:bij}, except the maps are now $\sP$-filtered.  For the second part, note that for any reduction
\[
i\partial_M\rho = \partial (i \rho) = \partial(\id_C-\gamma\partial-\partial\gamma) = \partial-\partial\gamma\partial.
\]
For a \strict{} reduction, we have
\[
0 =i^{pp}\partial_M^{pp}\rho^{pp} = \partial^{pp}-\partial^{pp}\gamma^{pp}\partial^{pp}.
\]

Thus $\gamma$ is perfect.  Conversely, let $\gamma$ be a perfect $\sP$-filtered splitting homotopy.  The formula for the differential on $M=\img(\rho)$ is $\partial_M = \partial-\partial\gamma\partial$.  Since the maps $\partial$ and $\gamma$ are $\sP$-filtered, we have 
\begin{align*}
\partial_M^{pp} = (\partial-\partial\gamma\partial)^{pp} = \partial^{pp}-\partial^{pp}\gamma^{pp}\partial^{pp} = 0 .
\tag*{\qedhere}
\end{align*}
\end{proof}
Observe that in a \strict{} reduction $\img(\phi^{pp})\subset \ker\partial^{pp}$ since $\partial^{pp} \phi^{pp} = \phi^{pp}\partial_M^{pp} = 0$.  Therefore the images $\phi^{pp}(M^p)$ are representatives of the homology $H_\bullet(C^p,\Delta^{pp})$.    We may also show that \cyclic{} $\sP$-graded chain complexes are minimal with respect to reductions.  This mirrors Proposition~\ref{prop:cont:cyclic}.  The point is that \cyclic{} $\sP$-graded complexes are the graded analogue of \trivial{} complexes.

\begin{prop}
Let $(C,\pi)$ be \cyclic{}.  Any reduction
\begin{center}
\begin{tikzcd}
(C,\pi)
\arrow[out=60,in=120,distance=5mm,swap,"\gamma"]
\arrow[r, shift left=1,"\psi"]
& (M,\pi) \arrow[l, shift left=1,"\phi"] 
\end{tikzcd}
\end{center}
is \strict{}. Moreover $(M,\pi)$ and $(C,\pi)$ are $\sP$-filtered chain isomorphic, i.e., isomorphic in $\bChG$.
\end{prop}
\begin{proof}
We have $\partial_M=\partial_M\psi\phi = \psi\partial\phi$.  Thus $\partial_M^{pp}=\psi^{pp}\partial^{pp}\phi^{pp}=0$.  Therefore $(M,\pi)$ is \cyclic{}. Since $i$ and $p$ are chain equivalences, invoking Proposition~\ref{prop:grad:cmiso} shows that $(M,\pi)$ and $(C,\pi)$ are $\sP$-filtered chain isomorphic. 
\end{proof}

For a tower of graded reductions, we have the following result, which is analogous to Proposition~\ref{prop:cont:tower}.

\begin{prop}\label{prop:cont:gtower}
Given a tower of $\sP$-graded reductions
\begin{center}
\begin{tikzcd}
(C,\pi)
\arrow[out=60,in=120,distance=5mm,swap,"\gamma_0"]
\arrow[r, shift left=1,"\psi_0"]
& (M_0,\pi) \arrow[l, shift left=1,"\phi_0"] 
\arrow[out=60,in=120,distance=5mm,swap,"\gamma_1"]
\arrow[r, shift left=1,"\psi_1"] 
& \ldots \arrow[l, shift left=1,"\phi_1"]
\arrow[r, shift left=1,"\psi_{n-1}"] 
& (M_{n-1},\pi) \arrow[l, shift left=1,"\phi_{n-1}"] 
\arrow[out=60,in=120,distance=5mm,swap,"\gamma_n"]
\arrow[r, shift left=1,"\psi_n"] 
& (M_n,\pi) \arrow[l, shift left=1,"\phi_n"] ;
\end{tikzcd}
\end{center}

\begin{enumerate}
\item there is a reduction 
\begin{center}
\begin{tikzcd}
(C,\pi)
\arrow[out=60,in=120,distance=5mm,swap,"\Gamma"]
\arrow[r, shift left=1,"\Psi_n"]
& (M,\pi) \arrow[l, shift left=1,"\Phi_n"] 
\end{tikzcd}
\end{center}
with maps given by the formulas
\[
 \Psi_m = \prod_{i=0}^m \psi_i   \quad\quad   \Phi_m = \prod_{i=0}^m \phi_i 
\quad\quad 
\Gamma = \gamma_0 + \sum_{i=0}^{n-1} \Phi_i\circ \gamma_{i+1}\circ \Psi_i.
\]

\item $\Gamma$ is a $\sP$-filtered splitting homotopy and $\Gamma$ is perfect if any $\gamma_i$ is perfect.

\end{enumerate}
\end{prop}

\section{Connection Matrix Algorithm}\label{sec:alg}

In this section we introduce the algorithm for computing Conley complexes and connection matrices.  The algorithm is based on (graded) Morse theory, which is described in Section~\ref{sec:reductions:gmt}.  It is formalized via the framework of reductions developed in Section~\ref{sec:reductions}.   

In Section~\ref{sec:alg:prelims} we recall the Morse theoretic algorithms of~\cite{focm}.  The exposition relies on the discrete Morse theory reviewed in Section~\ref{prelims:dmt}. In Section~\ref{sec:alg:hom} we demonstrate, via Algorithm~\ref{alg:hom} (\textsc{Homology}),  how to compute the homology of a chain complex using discrete Morse theory and reductions.  Section~\ref{sec:alg:cm} describes Algorithm~\ref{alg:cm} (\textsc{ConnectionMatrix}), the algorithm for computing a connection matrix based on graded discrete Morse theory and graded reductions.  The computation of a connection matrix is analogous to computing homology, except generalized to the category of $\sP$-graded chain complexes, i.e., the algorithm \textsc{ConnectionMatrix} is the analogy of the algorithm \textsc{Homology}, only adapted to the graded setting.  This compelling analogy provides a nice conceptual method for digesting how connection matrices can be computed.

\subsection{Morse Theoretic Algorithms}\label{sec:alg:prelims}

Our algorithm relies on~\cite[Algorithm 3.6]{focm} and~\cite[Algorithm 3.12]{focm}, which are reproduced below, respectively, as the algorithms  \textsc{Matching} and \textsc{Gamma}.    In particular, Lemma~\ref{prop:alg:cyclic}, which relies on  \textsc{Matching}, is used to to verify the correctness of Algorithms~\ref{alg:hom} and~\ref{alg:cm}.  First, recall the notion of a coreduction pair and free cell, from Definition~\ref{defn:cored}, and that of acyclic partial matching, from Definition~\ref{defn:acyclicmatching}.

\begin{algorithmic}
\Function{Matching}{$\cX$}
	\State $\cX' \gets \cX$
    \While {$\cX'$ is not empty} 
    	\While {$\cX'$ admits a coreduction pair $(\xi,\xi')$}
        	\State Excise $(\xi,\xi')$ from $\cX'$
            \State $K\gets \xi$, $Q\gets \xi'$
            \State $w(\xi'):=\xi$
    	\EndWhile
        \While {$\cX'$ does not admit a coreduction pair}
        	\State Excise a free cell $\xi$ from $\cX'$
            \State $A\gets \xi$
        \EndWhile
    \EndWhile
    \State \Return $(A,w\colon  Q\to K)$
\EndFunction
\Function{Gamma}{$\xi_{in}, w\colon Q\to K$}
\State $\xi \gets \xi_{in}$
\State $c\gets 0$
\While{$\xi\not\in C(A)\oplus C(K)$}
\State Choose a $\leq$-maximal $\xi'\in Q$ with $\kappa(\xi,\xi')\neq 0$
\State $\xi''\gets w(\xi')$
\State $c\gets c+\xi''$
\State $\xi\gets \xi+\partial \xi''$
\EndWhile
\State \Return $c$
\EndFunction
\end{algorithmic}

The proof of correctness of our algorithms depend upon the following lemma. 
\begin{lem}\label{prop:alg:cyclic}
Let $\cX$ be a cell complex.  If $(A,w)$ is an acyclic partial matching on $\cX$ obtained from $\textsc{Matching}(\cX)$ such that $A=\cX$ then $(C_\bullet(A),\partial^A)=(C_\bullet(\cX),\partial^\cX)$ is a \trivial{} complex.
\end{lem}
\begin{proof}
Let $\xi\in \cX=A$.  We wish to show that $\partial(\xi) = 0$.  Since $A=\cX$ there are no coreduction pairs in the execution of the algorithm.  This implies of the two secondary {\bf while} loops in \textsc{Matching}, only the second {\bf while} has executed.   This {\bf while} loop has iterated $n=|\cX|$ times and each iteration adds a cell $\xi$ to the collection of critical cells $A$.  We may therefore regard $A$ as a stack and label each $\xi\in A$ with the integer giving the particular iteration of the {\bf while} loop that added $\xi$ to $A$.  Denote this labeling $\mu\colon A\to \N$.   Now set $n=\mu(\xi)$ and let $U=\mu^{-1}[0,n)$.  From Algorithm 3.6 we must have that $\xi$ is a free cell in $\cX\setminus U$.  Therefore if $\kappa(\xi,\xi')\neq 0$ for some $\xi'\in \cX$ then $\xi'\in U$.  Suppose that $\kappa(\xi,\xi')\neq 0$ for some $\xi'\in U$.  Let $m=\mu(\xi')$ and $U'=\mu^{-1}[0,m)$.  We must have that $(\xi,\xi')$ is a coreduction pair in $\cX\setminus U'$.  This is a contradiction of the execution of the algorithm.
\end{proof}

\subsection{Homology Algorithm}\label{sec:alg:hom}

We first give an algorithm for computing the homology of a complex $\cX$ based on discrete Morse theory. This will provide an intuition and the basis for the Algorithm~\ref{alg:cm}, \textsc{ConnectionMatrix}.
We let $A_\infty$ denote the output of \textsc{Homology}.

\begin{alg}\label{alg:hom}
{\em
\hspace{5mm}
\begin{algorithmic}
\Function{Homology}{$\cX_{in},\partial_{in}$}
\State $A\gets\cX_{in}, \Delta \gets \partial_{in}$
  \Do
    \State $\cX\gets A,\partial \gets \Delta$
    \State $(A,w\colon Q\to K)\gets \textsc{Matching}(\cX)$
    \For{$\xi \in A$}
    \State compute and store $\Delta(\xi)$ using $\textsc{Gamma}(\xi,w)$
    \EndFor 
  \doWhile{$|A|<|\cX|$}
  \State \Return $A$
\EndFunction
\end{algorithmic}
}
\end{alg}
\begin{thm}\label{thm:alg:hom}
Given a cell complex $\cX$, Algorithm~\ref{alg:hom} (with input $\cX$ and $\partial$) halts and outputs the homology of $\cX$.
\end{thm}
\begin{proof}
The fact that $\cX_{in}$ is finite, together with the fact that \textsc{Matching} halts~\cite{focm}, implies that \textsc{Homology} halts.  Finally, if the algorithm terminates with $A_\infty=\textsc{Homology}(\cX_{in},\partial_{in})$ then $C(A_\infty)$ is a \trivial{} chain complex by Lemma~\ref{prop:alg:cyclic}. 

It remains to prove that $C(A_\infty)\cong H(\cX)$.  In any iteration of the {\bf do} loop, there are chain equivalences $\psi\colon C(\cX)\to C(A)$ and $\phi\colon C(A)\to C(\cX)$, which are as defined in Eqn.~\eqref{eqn:prelim:dmt} of Section~\ref{prelims:dmt}, using $\gamma(\cdot)=\textsc{Gamma}(\cdot,w)$.  The pair of complexes $C(\cX)$, $C(A)$ and the triple maps $\phi,\psi,\gamma$ fit into a reduction via Example~\ref{ex:red:dmt}. Therefore an execution of the entire the {\bf do-while} loop is associated to a tower of reductions:
\begin{equation*}
\begin{tikzcd}
C(\cX_{in})\arrow[out=60,in=120,distance=5mm,swap,"\gamma_0"]  \arrow[r, shift left=1] & 
\ldots \arrow[l, shift left=1] \arrow[r, shift left=1] &
\arrow[l, shift left=1] C_\bullet(\cX)
\arrow[out=60,in=120,distance=5mm,swap,"\gamma"]
\arrow[r, shift left=1,"\psi"]
& (C_\bullet(A),\partial^A) \arrow[l, shift left=1,"\phi"] \arrow[r, shift left=1]\arrow[out=60,in=120,distance=5mm] 
&\ldots\arrow[l, shift left=1]  \arrow[r, shift left=1]&
\arrow[l, shift left=1] C(A_\infty).
\end{tikzcd}
\end{equation*}

Thus the output $C(A_\infty)$ is isomorphic to the homology $H(\cX_{in})$.
\end{proof}

\begin{ex}\label{ex:alg:homology}
{\em
In this example we give some flavor of the concepts behind Algorithm~\ref{alg:hom} (\textsc{Homology}).  Consider the cubical complex $\cK$ given in Fig.~\ref{fig:alg:hom}(a) that consists of four $2$-cells, 14 edges and nine vertices.  We work over the field $\Z_2$.  Therefore we have
\[
C_2(\cK) = \Z_2^4 \quad\quad C_1(\cK) = \Z_2^{14}\quad\quad C_0(\cK) = \Z_2^9.
\]
The complex $\cK$ is open on the right in order to simplify the Morse theory, viz., there is no critical vertex.  For the sake of an example, we want to illustrate that in practice the algorithm uses multiple rounds of Morse theory, and there is an associated nontrivial tower of reductions. Unfortunately, the algorithm \textsc{Homology} as stated is too effective in this example as the \textsc{Matching} subroutine simplifies to a \trivial{} complex in only one round of Morse theory.  Instead, we may substitute \textsc{Matching} with the set of cubical acyclic partial matchings proposed in~\cite{hms2}.  In this case $\cK$ is a cubical complex in $\R^2$, and each coordinate direction -- the $x$ and $y$ directions -- gives an acyclic partial matching by attempting to match cells `right' along that direction.
\begin{figure}[ht!]
\begin{minipage}[t]{0.325\textwidth}
\begin{center}
\begin{tikzpicture}[dot/.style={draw,circle,fill,inner sep=1.5pt},line width=.7pt]
\draw[->,thin] (-.25,-.25) -- (.25,-.25);
\draw[->,thin] (-.25, -.25) -- (-.25,.25);
\node (u) at (.5,-.25) {$x$};
\node (v) at (-.25,.5) {$y$};
\fill[pattern=north east lines](0,0) rectangle (1,2);
\fill[pattern=north east lines](2,0) rectangle (2.9,2);
\foreach \x in {0,1,2}
    \foreach \y in {0,1,2}
        \node (\x\y) at (\x,\y)[dot] {};
\foreach \y in {0,1,2}
    \node (3\y) at (3,\y) {};
\foreach \x in {0,1,2}
    \foreach \y[count=\yi] in {0,1}
        \draw (\x\y)--(\x\yi);
\foreach \x [count=\xi]in {0,1,2}
    \foreach \y in {0,2}
        \draw (\x\y)--(\xi\y);
\draw (21)--(31);
\draw (01)--(11);
\end{tikzpicture}

(a)
\end{center}
\end{minipage}
\begin{minipage}[t]{0.3\textwidth}
\begin{center}
\begin{tikzpicture}[dot/.style={draw,circle,fill,inner sep=1.5pt},line width=.7pt]
\node (u) at (-.25,-.25) {};
\fill[pattern=north east lines](0,0) rectangle (1,2);
\fill[pattern=north east lines](2,0) rectangle (2.9,2);
\foreach \x in {0,1,2}
    \foreach \y in {0,1,2}
        \node (\x\y) at (\x,\y)[dot] {};
\foreach \y in {0,1,2}
    \node (3\y) at (3,\y) {};
\foreach \x in {0,1,2}
    \foreach \y[count=\yi] in {0,1}
        \draw (\x\y)--(\x\yi);
\foreach \x [count=\xi]in {0,1,2}
    \foreach \y in {0,2}
        \draw (\x\y)--(\xi\y);
\draw (21)--(31);
\draw (01)--(11);
\foreach \y in {0,.5,1,1.5,2}
    \foreach \x in {0,2}
    \draw[-latex,line width=1.5pt] (\x,\y) -- (\x+.6,\y);
\foreach \y [count=\xi] in {0,2}
    \draw[-latex,line width=1.5pt] (1,\y) -- (1.6,\y);
\node (e0) at (1.25,.5){$e_0$};
\node (e1) at (1.25,1.5){$e_1$};
\node (v0) at (1.35,1) {$v_0$};
\end{tikzpicture}

(b)
\end{center}
\end{minipage}
\begin{minipage}[t]{0.15\textwidth}
\begin{center}
\begin{tikzpicture}[dot/.style={draw,circle,fill,inner sep=1.5pt},line width=.7pt]
\node (u) at (-5-.25,-.25) {};

 \node (0) at (-5,0) {};
 \node (1) at (-5,1)[dot] {};
 \node (2) at (-5,2) {};
 \draw (0) -- (1);
 \draw (1) -- (2);
 \draw[-latex,line width=1.5pt] (1) -- (-5,1.6);
\node (e0) at (-5+.25,.5){$e_0$};
\node (e1) at (-5+.35,1.5){$e_1$};
\node (v0) at (-5+.35,1) {$v_0$};
\end{tikzpicture}

(c)
\end{center}
\end{minipage}
\begin{minipage}[t]{0.15\textwidth}
\begin{center}
\begin{tikzpicture}[dot/.style={draw,circle,fill,inner sep=1.5pt},line width=.7pt]
\node (u) at (-5-.25,-.25) {};
 \node (0) at (-5,0) {};
 \node (1) at (-5,1) {};
 \draw (0) -- (1);
 \node (e0) at (-5+.25,.5){$e_0$};
\end{tikzpicture}

(d)
\end{center}
\end{minipage}
\caption{(a) Cubical complex $\cK$. (b)  First Pairing $(A_0,w_0)$.  (c) Second Pairing $(A_1,w_1)$. (d)  \Trivial{} complex, i.e., the homology of $\cK$.  A single 1-cell $e_0$ remains with $\partial(e_0) = 0$.}\label{fig:alg:hom}
\end{figure}

The algorithm begins with executing the first iteration of the {\bf while} loop that computes an acyclic partial matching on $\cK$ by attempting to pair all cells right along the $x$ direction. This furnishes an acyclic partial matching $(A_0,w_0)$.   This is visualized in Fig.~\ref{fig:alg:hom}(b) where a pair $\xi\in Q$ and $\xi'\in K$  with $w(\xi)=\xi'$ is visualized with directed edge $\xi'\to \xi$. 
The directed edges in Fig.~\ref{fig:alg:hom}(b) may also be thought of as a graphical representation of the degree 1 map $V$ (see Section~\ref{prelims:dmt}). 
The cells $e_0,e_1$ and $v_0$ do not have right coboundaries and are therefore critical, i.e., $A_0=\{e_0,e_1,v_0\}$.  
The set of cells $A_0$ is the basis for a chain complex $(C(A_0),\Delta)$ where
\begin{align*}
C_1(A_0) = \Z_2\langle e_0\rangle \oplus \Z_2\langle e_1\rangle, \quad\quad C_0(A_0) = \Z_2\langle v_0\rangle,
\quad\quad
\Delta_1 = \bordermatrix{& e_0 & e_1  \cr
                         v_0 & 1 & 1  }.
\end{align*}
The second iteration of the {\bf while} loop attempts to pair remaining cells, i.e., the cells in $A_0$, upwards (along the $y$ direction).  This furnishes an acyclic partial matching $(A_1,w_1)$ on $A_0$, visualized in Fig.~\ref{fig:alg:hom}(c). The cells $v_0$ and $e_1$ are paired, i.e., $w_1(v_0)=e_1$, and $A_1 = \{e_0\}$. 
 Moreover, the set of cells $A_1$ is a basis for the chain complex $(C(A_1),\Delta)$ where
\[
C_1(A_1) = \Z_2\langle e_0\rangle \quad\quad \text{and}\quad\quad \Delta = 0.
\]
In the final iteration of the {\bf while} there are no coreduction pairs and $A_2 = A_1$ is returned. The algorithm terminates with $A_\infty=\textsc{Homology}(\cK,\partial) = A_1$.  The two rounds of Morse theory occurring during the algorithm give rise to the chain complexes $C(\cX),C(A_0)$ and $C(A_1)$, together with the maps $\{\phi_i,\psi_i,\gamma_i\}$, which fit into the tower of reductions:
\begin{equation*}
\begin{tikzcd}
C(\cK)\arrow[out=60,in=120,distance=5mm,swap,"\gamma_0"]  \arrow[r, shift left=1,"\psi_0"] & 
C(A_0) \arrow[out=60,in=120,distance=5mm,swap,"\gamma_1"]
\arrow[l, shift left=1,"\phi_0"] \arrow[r, shift left=1,"\psi_1"] &
C(A_1) \arrow[l, shift left=1,"\phi_1"] .
\end{tikzcd}
\end{equation*}
}
\end{ex}


\subsection{Graded Morse Theory}\label{sec:reductions:gmt}

In this section, we review a graded version of discrete Morse theory.  Consider a $\sP$-graded cell complex $(\cX,\nu)$.  Recall that the underlying set $\cX$ decomposes as $\cX=\bigsqcup_{p\in \sP} \cX^p$ where $\cX^p = \nu^{-1}(p)$.
\begin{defn}
{\em
Let $(\cX,\nu)$ be a $\sP$-graded cell complex and  let $(A,w\colon Q\to K)$ be an acyclic partial matching on $\cX$.  We say that $(A,w)$ is {\em $\sP$-graded}, or simply {\em graded}, if it satisfies the property that $w(\xi)=\xi'$ only if $\xi,\xi'\in \cX^p$ for some $p\in \sP$.  That is, matchings may only occur in the same fiber of the grading.  
}
\end{defn}

The idea of graded matchings can be found many places in the literature, for instance, see~\cite{mn} and~\cite[Patchwork Theorem]{koz}.  Recall from Definition~\ref{defn:grad:assoc} that a $\sP$-graded cell complex $(\cX,\nu)$ has an associated $\sP$-graded chain complex $(C(\cX),\pi^\nu)$.

\begin{prop}\label{prop:dmt:gmt}
Let $(\cX,\nu)$ be a $\sP$-graded cell complex and $(A,w:Q\to K)$ a graded acyclic partial matching. Let $A^p=A\cap \cX^p$. Then 
\begin{enumerate}
    \item $(C(A),\partial^A,\pi)$ is a $\sP$-graded chain complex where the projections $\pi=\{\pi_n^p\}_{n\in \Z,p\in\sP}$ are given by
    \begin{equation}\label{eqn:ace:proj}
    \pi_j^p\colon C_j(A)\to C_j(A^p).
    \end{equation}
    \item The maps $\phi,\psi$ of Eqn.~\eqref{eqn:prelim:dmt} and $\gamma$ of Eqn.~\eqref{eqn:gamma} fit into a  $\sP$-graded reduction
\begin{equation}\label{eqn:ace:red}
\begin{tikzcd}
(C(\cX),\pi^\nu)
\arrow[out=60,in=120,distance=5mm,swap,"\gamma"]
\arrow[r, shift left=1,"\psi"]
& (C(A),\pi)\arrow[l, shift left=1,"\phi"] .
\end{tikzcd}
\end{equation}

\end{enumerate}
\end{prop}
\begin{proof}
It follows from Proposition~\ref{thm:focm:red} that $(C(A),\partial^A)$ is a chain complex. We must show that if $\partial_A^{pq}\neq 0$ then $p\leq q$.  By Proposition~\ref{prop:matchinghomotopy} there is a unique splitting homotopy $\gamma\colon C(\cX)\to C(\cX)$ associated to the matching $(A,w)$. The fact that $(A,w)$ is graded implies that $V$, as defined in Eqn.~\eqref{eqn:defn:V} of Section~\ref{prelims:dmt}, is $\sP$-filtered.  From the definition of $\gamma$ given in~\eqref{eqn:gamma} a routine verification shows that $\gamma$ is $\sP$-filtered.  Therefore by Proposition~\ref{prop:grad:bij} there is an associated reduction
\[
\begin{tikzcd}
(C(\cX),\pi^\nu)
\arrow[out=60,in=120,distance=5mm,swap,"\gamma"]
\arrow[r, shift left=1,"\psi"]
& (C(A),\pi)\arrow[l, shift left=1,"\phi"] .
\end{tikzcd}
\]

Let $p\in \sP$.  Consider $(A^p,w^p)$ the matching restricted to the fiber $\cX^p = \nu^{-1}(p)$.  We have $$A^p = A\cap \cX^p\quad\quad \quad \quad  w^p\colon Q\cap \cX^p\to K\cap \cX^p.$$

It follows that $(A^p,w^p)$ is an acyclic partial matching on the fiber $\cX^p$.  Proposition~\ref{prop:matchinghomotopy} implies that there is a unique splitting homotopy $\gamma^p\colon C(\cX^p)\to C(\cX^p)$. In particular, $\gamma^{pp}=\gamma^p$.
\end{proof}

\begin{ex}
{\em 
Consider the graded complex $(\cX,\nu)$ of Example~\ref{ex:complex}.   Let
\begin{align*}
    A:= \{v_0,v_2,e_0\} \quad\quad Q:=\{v_1\} \quad\quad K:=\{e_1\}\quad\quad  w(v_1)=e_1
\end{align*}
This is depicted in Fig.~\ref{fig:gradedmatching}.
\begin{figure}[h!]
\centering
\begin{tikzpicture}[dot/.style={draw,circle,fill,inner sep=1.5pt},line width=.7pt]
\node (v0) at (0:0)[dot,label=above:{$v_0$}] {};
\node (v1) at (0cm:2cm)[dot,label=above:{$v_1$}] {};
\node (v2) at (0cm:4cm)[dot,label=above:{$v_2$}] {};
\node (e0) at (0cm:1cm)[label=above:{$e_0$}] {};
\node (e1) at (0cm:3cm)[label=above:{$e_1$}] {};
\draw (v0) to[bend left=10] (v1);
\draw (v1) to[bend left=10] (v2);
\draw[-latex,line width=2.25pt] (v1) -- (.075cm:3cm);
\end{tikzpicture}
 \caption{Graded matching $(A,w)$ on $\cX$.  The pairing $w(v_1)=e_1$ is visualized with an arrow $v_1\to e_1$.  The set $A=\{v_0,e_0,v_2\}$ are the critical cells. }\label{fig:gradedmatching}
\end{figure}

This is a acyclic partial matching $(A,w)$ on $\cX$.  It is straightforward that $(A,w)$ is graded as $v_1,e_1\in \nu^{-1}(q)$.  The maps of the associated $\sP$-graded reduction are precisely the ones described in Example~\ref{ex:homotopy}.
}
\end{ex}

\subsection{Connection Matrix Algorithm}\label{sec:alg:cm}

We can now state the algorithm for computing a connection matrix, which relies on \textsc{Matching} and \textsc{Gamma} of Section~\ref{sec:alg:prelims}.



\begin{alg}\label{alg:cm}
{\em 
\hspace{5mm}
\begin{algorithmic}
\Function{ConnectionMatrix}{$\cX_{in},\nu_{in},\partial_{in}$}
\State $A\gets\cX_{in}, \Delta \gets \partial_{in}, \mu \gets \nu_{in}$
  \Do
    \State $\cX\gets A,\partial \gets \Delta, \nu\gets \mu$
    \For{$p\in \sP$}
    \State $(A^p,\omega^p\colon Q^p\to K^p)\gets \textsc{Matching}(\cX^p)$
    \EndFor
    \State $(A,w)\gets (\bigsqcup_{p\in \sP}A^p,~\bigsqcup_{p\in \sP}w^p)$
    \For{$\xi \in A$}
    \State compute and store $\Delta(\xi)$ using $\textsc{Gamma}(\xi,w)$
    \EndFor 
    \State $\mu\gets \nu|_A$
  \doWhile{$|A|<|\cX|$}
  \State \Return $(A,\Delta,\mu)$
\EndFunction
\end{algorithmic}
}
\end{alg}

\begin{thm}\label{thm:alg:cm}
Let $(\cX,\nu)$ be a $\sP$-graded cell complex.   Algorithm~\ref{alg:cm} (with input $(\cX,\nu)$ and $\partial$) halts.  Moreover, the returned data $(A_\infty,\Delta,\pi)$  has the property that $(C(A_\infty),\Delta,\pi)$ is a Conley complex for $(C(\cX),\pi^\nu)$.
\end{thm}
\begin{proof}
Since $\cX_{in}$ is finite and \textsc{Matching} halts, it follows that \textsc{ConnectionMatrix} halts.     Let $(A_\infty,\Delta_\infty,\mu_\infty)=\textsc{ConnectionMatrix}(\cX_{in},\nu,\partial)$.  It follows from Proposition~\ref{prop:dmt:gmt} that $(C(A_\infty),\Delta,\pi)$, where $\pi$ is given by Eqn.~\eqref{eqn:ace:proj}, is $\sP$-graded.   It follows from Lemma~\ref{prop:alg:cyclic} that for each $p\in \sP$ the fiber $A_\infty^p$ is \trivial{}.  This implies that the $\sP$-graded chain complex $C(A_\infty)$ is \cyclic{}. 

It remains to show that $(C(A_\infty),\Delta,\pi)$ is a Conley complex.  In any iteration of the {\bf do} loop, it follows from Proposition~\ref{prop:dmt:gmt} that there are $\sP$-filtered chain equivalences $\psi\colon C(\cX)\to C(A)$ and $\phi\colon C(A)\to C(\cX)$, which are as defined in \eqref{eqn:prelim:dmt}, using $\gamma(\cdot)=\textsc{Gamma}(\cdot,w)$.  The pair of complexes $(C(\cX),\pi^\nu)$, $(C(A),\pi^\mu)$ and the triple maps $\phi,\psi,\gamma$ fit into the $\sP$-graded reduction of~\eqref{eqn:ace:red}.  Therefore an execution of the entire {\bf do-while} loop is associated to a tower of reductions:
\begin{equation}\label{dia:algtower}
\begin{tikzcd}[column sep=1.8em]
(C(\cX_{in}),\pi_{in})\arrow[out=60,in=120,distance=5mm,swap,"\gamma_0"]  \arrow[r, shift left=1] & 
\ldots \arrow[l, shift left=1] \arrow[r, shift left=1] &
\arrow[l, shift left=1] 
(C(\cX),\pi^\nu)
\arrow[out=60,in=120,distance=5mm,swap,"\gamma"]
\arrow[r, shift left=1,"\psi"]
& (C(A),\pi^\mu) \arrow[l, shift left=1,"\phi"] \arrow[r, shift left=1]\arrow[out=60,in=120,distance=5mm] 
&\ldots\arrow[l, shift left=1]  \arrow[r, shift left=1]&
\arrow[l, shift left=1] (C(A_\infty),\pi).
\end{tikzcd}
\end{equation}
Thus the output $(C(A_\infty),\Delta)$ is a Conley complex.
\end{proof}
\begin{ex}\label{ex:alg:cm}
{\em

We give an example of the algorithm \textsc{ConnectionMatrix}.   Let $\cX$ be the cubical complex in Fig.~\ref{fig:alg:cm}(a) and let $\cK$ be the cubical complex from Example~\ref{ex:alg:homology}. We again work over the field $\Z_2$. The cubical complex $\cX$ consists of $\cK$ together with the 2-cells $\{\xi_0,\xi_1\}$ and the 1-cell $e_2$. The 2-cells in $\cX\setminus \cK$ are shaded, while the 2-cells in $\cK$ are drawn with hatching. Let $\sQ=\{0,1\}$ be the poset with order $0\leq 1$.  There is a $\sQ$-graded cell complex where $(\cX,\nu)$ and $\nu\colon \cX\to \sQ$ is given via 
\[
  \nu(x) =
  \begin{cases}
          0 & \text{$x\in \cK$ } \\
          1 & \text{$x\in \cX\setminus \cK$.} 
  \end{cases}
\]
 Once again, using \textsc{Matching} would be too effective on this example to illustrate multiple rounds of Morse theory.  We proceed as before and use the graded cubical acyclic partial matchings proposed in~\cite{hms2}.  The $x$ and $y$ directions each give an acyclic partial matching by attempting to pair cells along this direction; in this case care is taken to ensure cells are only matched if they belong to same fiber, i.e., the matching is graded. 
\begin{figure}[h!]
\begin{minipage}[t]{0.3\textwidth}
\begin{center}
\begin{tikzpicture}[dot/.style={draw,circle,fill,inner sep=1.5pt},line width=.7pt]
\draw[->,thin] (-.25,-.25) -- (.25,-.25);
\draw[->,thin] (-.25, -.25) -- (-.25,.25);
\node (u) at (.5,-.25) {$x$};
\node (v) at (-.25,.5) {$y$};
\fill[pattern=north east lines](0,0) rectangle (1,2);
\fill[pattern=north east lines](2,0) rectangle (2.9,2);
\fill[color=gray!20](1,0) rectangle (2,2);

\foreach \x in {0,2}
    \foreach \y in {0,1,2}
        \node (\x\y) at (\x,\y)[dot] {};
\foreach \y in {0,1,2}
    \node (3\y) at (3,\y) {};
\foreach \x in {0,1,2}
    \foreach \y[count=\yi] in {0,1}
        \draw (\x\y)--(\x\yi);
\foreach \x [count=\z]in {0,1,2}
    \foreach \y in {0,2}
        \draw (\x\y)--(\z\y);
\draw (21)--(31);
\draw (01)--(11);
\draw (11)--(21);
\foreach \x [count=\z] in {0,2}
        \draw (\x1)--(\z1);
\foreach \y in {0,1,2}
    \node (1\y) at (1,\y)[dot] {};
\node (e0) at (1.5,.5){$\xi_0$};
\node (e1) at (1.5,1.75){$\xi_1$};
\node (v0) at (1.5,1.2){$e_2$};
\end{tikzpicture}

(a)
\end{center}
\end{minipage}
\begin{minipage}[t]{0.3\textwidth}
\begin{center}

\begin{tikzpicture}[dot/.style={draw,circle,fill,inner sep=1.5pt},line width=.7pt]
\fill[pattern=north east lines](0,0) rectangle (1,2);
\fill[pattern=north east lines](2,0) rectangle (2.9,2);
\fill[color=gray!20](1,0) rectangle (2,2);
\node (u) at (-.25,-.25) {};

\foreach \x in {0,2}
    \foreach \y in {0,1,2}
        \node (\x\y) at (\x,\y)[dot] {};
\foreach \y in {0,1,2}
    \node (3\y) at (3,\y) {};
\foreach \x in {0,1,2}
    \foreach \y[count=\yi] in {0,1}
        \draw (\x\y)--(\x\yi);
\foreach \x [count=\xi]in {0,1,2}
    \foreach \y in {0,2}
        \draw (\x\y)--(\xi\y);
\draw (21)--(31);
\draw (01)--(11);
\draw (11)--(21);
\foreach \y in {0,.5,1,1.5,2}
    \foreach \x in {0,2}
    \draw[-latex,line width=1.5pt] (\x,\y) -- (\x+.6,\y);
\foreach \y [count=\xi] in {0,2}
    \draw[-latex,line width=1.5pt] (1,\y) -- (1.6,\y);
\foreach \y in {0,1,2}
    \node (1\y) at (1,\y)[dot] {};
\end{tikzpicture}

(b)
\end{center}
\end{minipage}
\begin{minipage}[t]{0.15\textwidth}
\begin{center}
\begin{tikzpicture}[dot/.style={draw,circle,fill,inner sep=1.5pt},line width=.7pt]
\node (u) at (-.25,-.25) {};

\fill[color=gray!20](0,0) rectangle (1,2);
 \node (0) at (0,0) {};
 \node (1) at (0,1)[dot] {};
 \node (2) at (0,2) {};
 \draw (0,0) -- (1);
 \draw (1) -- (0,2);
 \draw (1)--(1,1);
 \draw[-latex,line width=1.5pt] (1) -- (0,1.6);
  \draw[-latex,line width=1.5pt] (.5,1) -- (.5,1.6);
\node (x0) at (.5,.5){$\xi_0$};
\node (x1) at (.5,1.75){$\xi_1$};
\node (e2) at (.75,1.2){$e_2$};
\node (e0) at (-.25,.5){$e_0$};
\node (e1) at (-.25,1.5){$e_1$};
\node (v0) at (-.25,1) {$v_0$};
\end{tikzpicture}

(c)
\end{center}
\end{minipage}
\begin{minipage}[t]{0.15\textwidth}
\begin{center}
\begin{tikzpicture}[dot/.style={draw,circle,fill,inner sep=1.5pt},line width=.7pt]
\node (u) at (-.25,-.25) {};

\fill[color=gray!20](0,0) rectangle (1,1);
 \draw (0,0) -- (0,1);
 \node (x0) at (.5,.5){$\xi_0$};
\node (e0) at (-.25,.5){$e_0$};
\end{tikzpicture}

(d)
\end{center}
\end{minipage}
\caption{(a) Graded Cubical Complex.  (b) First Graded Pairing $(A_0,w_0)$.  (c) Second Graded Pairing $(A_1,w_1)$. (d)  Conley complex.  A 2-cell $\xi_0$ and a 1-cell $e_0$ remain with $\partial(\xi_0)=e_0$. }\label{fig:alg:cm}
\end{figure}

The algorithm starts by computing a graded acyclic partial matching on $\cK$, attempting to pair all cells to the right (within their fiber).  The cells $e_0,e_1,v_0$ have right coboundaries $\xi_0,\xi_1,e_2$ respectively.  However, these do not lie in the same fiber as $e_0,e_1,v_0\in \cX^0$ and $\xi_0,\xi_1,e_2\in \cX^1$. Therefore these cells cannot be paired and $A_0=\{e_0,e_1,v_0,\xi_0,\xi_1,e_2\}$.  The second round of Morse theory attempts to pair the remaining cells up (within their fiber).  In this case, $w(v_0)=e_1$ and $w(e_2)=\xi_1$ and $A_1 = \{e_0,\xi_0\}$.  These two rounds of graded Morse theory give a tower of graded reductions:
\begin{equation*}
\begin{tikzcd}
(C(\cX),\pi^\nu)\arrow[out=60,in=120,distance=5mm,swap,"\gamma_0"]  \arrow[r, shift left=1,"\psi_0"] & 
(C(A_0),\pi^{\mu_0}) \arrow[out=60,in=120,distance=5mm,swap,"\gamma_1"]
\arrow[l, shift left=1,"\phi_0"] \arrow[r, shift left=1,"\psi_1"] &
(C(A_\infty),\pi^{\mu_\infty}) \arrow[l, shift left=1,"\phi_1"].
\end{tikzcd}
\end{equation*}
The returned data $(A_\infty,\Delta,\mu)$ form the \cyclic{} $\sP$-graded chain complex, where 
\[
C_2(A_\infty) = \Z_2\langle \xi_0\rangle \quad\quad C_1(A_\infty) = \Z_2\langle e_0\rangle\quad\quad \Delta_2^{01} = \begin{pmatrix} 1 \end{pmatrix}.
\]
We can visualize the tower in terms of a sequence of fiber graphs (see Example~\ref{ex:bigcomplex}) as in Fig.~\ref{fig:alg:fgraphs}.
\begin{figure}[h!]
\centering
\begin{minipage}{.3\textwidth}
\centering
\begin{tikzpicture}[node distance=.5cm]
\node (v0) at (0,0){};
\node[draw,ellipse]  (a) at (0,.5) {$0 : 9t^0+14t^1+4t^2$};
\node[draw,ellipse]  (b) [above=.5cm of a]{$1 : t^1+2t^2$};
\draw[->,>=stealth,thick] (b)--(a);
\end{tikzpicture}\\
(a)
\end{minipage}
\begin{minipage}{.45\textwidth}
\centering
\begin{tikzpicture}[node distance=.5cm]
\node (v0) at (0,0){};
\node[draw,ellipse]  (a) at (0,.5) {$0 : t^0+2t^1$};
\node[draw,ellipse]  (b) [above=.5cm of a]{$1 : t^1+2t^2$};
\draw[->,>=stealth,thick] (b)--(a);
\end{tikzpicture}\\
(b)
\end{minipage}
\begin{minipage}{.15\textwidth}
\centering
\begin{tikzpicture}[node distance=.5cm]
\node (v0) at (0,0){};
\node[draw,ellipse]  (a) at (0,.5) {$0 : t^1$};
\node[draw,ellipse]  (b) [above=.5cm of a]{$1 : t^2$};
\draw[->,>=stealth,thick] (b)--(a);
\end{tikzpicture}\\
(c)
\end{minipage}
\caption{(a) Fiber graph for the initial $\sQ$-graded cubical complex $(\cX,\nu)$.  (b) Fiber graph for intermediate $\sQ$-graded cell complex  $(A_0,\mu_0)$.  (c) Conley-Morse graph for $(\cX,\nu)$, i.e., fiber graph for the final $\sQ$-graded cell complex $(A_\infty,\mu_\infty)$; see Example\ \ref{ex:bigcomplex}.}
\label{fig:alg:fgraphs}
\end{figure}
}
\end{ex}

\begin{rem}
Algorithm~\ref{alg:cm} may be refined by returning either the entire tower of reductions or the reduction defined by the compositions as in Theorem~\ref{prop:cont:gtower}.  Returning these data allow one to lift generators back in the original chain complex.
\end{rem}
\begin{rem}
Our implementation of this algorithm is available at~\cite{cmcode}.  Also available is a Jupyter notebook for the application of the algorithm to a Morse theory on braids; see~\cite{hms2,braids}.   The application to braids falls within the scope of a larger project, namely developing the ability to compute connection matrices for transversality models; see~\cite{braids}.  More details on the specifics of the algorithm, along with timing data, are covered in~\cite{hms2}.
\end{rem}


\section{Filtered Complexes}\label{sec:lfc}


In this section we introduce chain complexes filtered by a lattice.  We also review the connection matrix as defined by Robbin-Salamon~\cite{robbin:salamon2}.   

The results of this section which are significant for connection matrix theory are Proposition~\ref{prop:filt:cmiso}, which shows that a strict lattice-filtered chain complexes is an invariant of the chain equivalence class, cf. Proposition~\ref{prop:grad:cmiso} and Theorem~\ref{thm:filt:cm}, which establishes that computing a connection matrix in the sense of~\cite{robbin:salamon2} can be done at the level of the $\sP$-graded chain complex.  The relationship between posets and lattices encapsulated by Birkhoff's theorem~(Section \ref{sec:birkhoff}) is also reflected in the homological algebra.  Namely, we establish an equivalence of categories -- Theorem~\ref{thm:filt:equiv} -- between the category of $\sL$-filtered chain complexes and $\sP$-graded chain complexes where $\sL=\sO(\sP)$. 

For the remainder of this section we fix $\sL$ in $\bFDLat$.  
\subsection{Filtered Vector Spaces}\label{sec:filt:vs}

\begin{defn}\label{defn:filt:vs}
{\em
An {\em $\sL$-filtered vector space} $V=(V,f)$ is a vector space $V$ equipped with a lattice morphism $f\colon \sL\to \Sub(V)$. We call $f$ an {\em $\sL$-filtering} of $V$.  Suppose $(V,f)$ and $(W,g)$ are $\sL$-filtered vector spaces.  A map $\phi\colon V\to W$ is $\sL$-filtered if \[
\phi(f(a))\subseteq g(a),\quad \text{for all $a\in \sL$.}
\]
The {\em category of $\sL$-filtered vector spaces}, denoted $\bFVec(\sL)$, is the category whose objects are $\sL$-filtered vector spaces and whose morphisms are $\sL$-filtered linear maps.
}
\end{defn}

Since $f$ is a finite lattice homomorphism, we have that under $f$  $$0_\sL\mapsto 0\quad \quad\text{and} \quad\quad 1_\sL \mapsto V .$$  

We write a family of $\sL$-filtered vector spaces as $(V_\bullet,f_\bullet)=\{(V_n,f_n)\}_{n\in \Z}$. For a fixed $a\in \sL$ there is a family of vector spaces $f(a) = f_\bullet(a)=\{f_n(a)\}_{n\in \Z}$.

Proposition~\ref{prop:map:filtgrad} enables the definition of a functor which constructs a lattice-filtered chain complex from a poset-graded chain complex.  Recall that $u$ is the forgetful functor $u\colon \bGVec(\sP)\to \bVec$ which forgets the grading.
\begin{defn}\label{defn:filt:funcL}
{\em 
Let $\sL=\sO(\sP)$. Define the functor $\fL\colon \bGVec(\sP)\to \bFVec(\sL)$ via
\[
\fL\big[(V,\pi)\big] := (\frU(V,\pi),f) = (V,f)
\]
where the $\sL$-filtering $f\colon \sL\to \Sub(V)$ sends $a\in \sO(\sP)$ to 
\[
 V^a = \frU^a\big[(V,\pi)\big] \in \Sub(V) \, .
 \]
Proposition~\ref{prop:map:filtgrad} states that a $\sP$-filtered map $\phi\colon (V,\pi)\to (W,\pi)$ is $\sL$-filtered.  Therefore we define $\fL$ to be the identity on morphisms:
\[
\fL(\phi) := \phi \in \Hom_{\bFVec}((V,f),(W,g))
\]
}
\end{defn}

\begin{thm}\label{thm:filt:vs:equiv}
Let $\sL:=\sO(\sP)$.  The functor $\fL\colon \bGVec(\sP)\to \bFVec(\sL)$ is additive, full, faithful and essentially surjective.
\end{thm}
\begin{proof}
The functor $\fL$ is additive since $\fL$ is an identity on hom-sets. That $\fL$ is a bijection on hom-sets (fully faithful) follows from Proposition~\ref{prop:map:filtgrad}.  We now show that $\fL$ is essentially surjective. Let $(V,f)$ be an $\sL$-filtered vector space; first we will construct a $\sP$-graded vector space $(W,\pi)$ and then we will show that it satsifies $\fL(W,\pi)=(V,f)$.  $\Sub(V)$ is a relatively complemented lattice (see Definition~\ref{def:prelims:lat:rc} and Example~\ref{def:prelims:lat:subv}).  Therefore we may choose for each join irreducible $p\in \sJ(\sL)$ a subspace $W^p\in \Sub(V)$ such that $W^p+f(\pred p) = f(p)$ and $W^p\cap f(\pred p)=0$.  Thus $f(p)=W^p\oplus f(\pred p)$.    As $\sL$ is an object of $\bFDLat$, Proposition~\ref{prop:jr:rep} gives that any $a\in \sL$ can be written as the irredundant join of join-irreducibles, i.e., we have $a=\vee_i q_i$ with $q_i\in \sJ(\sL)$.  Thus
\[
f(a) = f(\vee_i q_i) = \vee_i f(q_i) \, .
\] 
It follows from well-founded induction over the underlying poset of $\sL$, that for all $a\in \sL$ 
\[
f(a) = \bigoplus_{\substack{q\leq a,\\ q\in \sJ(\sL)}} W^q.
\]
Now set $W=\bigoplus_{q\in \sJ(\sL)} W^q$ and $\pi=\{\pi^q\}_{q\in \sJ(\sL)}$ where $\pi^q$ is defined to be the projection $\pi^q\colon V\to W^q$.  $(W,\pi)$ is a $\sJ(\sL)$-graded vector space.  From Birkhoff's theorem, $\sJ(\sL)$ and $\sP$ are isomorphic, which implies that $(W,\pi)$ may be regarded as a $\sP$-graded vector space.  Now we show that $\fL(W,\pi) = (V,f)$.  From the definition of $\fL$ it suffices to choose $a\in \sO(\sP)$ and show that $W^a = f(a)$.  This follows since $f(a) = \bigoplus_{q\leq a} W^q = W^a $.
\end{proof}

\subsection{Filtered Chain Complexes}\label{sec:filt:ch}

Similar to $\bGVec(\sP)$, the category $\bFVec(\sL)$ is additive but not abelian.  Following Section~\ref{sec:prelims:HA} once again, we may form the category $\bCh(\bFVec(\sL))$ of chain complexes in $\bFVec(\sL)$.  An object $C$ of $\bCh(\bFVec(\sL))$ is a chain complex in $\sL$-filtered vector spaces.  For short, we say that this is an {\em $\sL$-filtered chain complex}.  The data of $C$ can be unpacked as the triple  $C=(C_\bullet,\partial_\bullet,f_\bullet)$ where:
\begin{enumerate}
    \item $(C_\bullet,\partial_\bullet)$ is a chain complex,
    \item $(C_n,f_n)$ is an $\sL$-filtered vector space for each $n$, and
    \item $\partial_n\colon (C_n,f_n)\to (C_{n-1},f_{n-1})$ is an $\sL$-filtered linear map.
\end{enumerate}

We will denote $C$ as $(C,f)$ to distinguish the $\sL$-filtering.  A morphism $\phi\colon (C,f)\to (C',f')$ is a chain map $\phi\colon (C,\partial)\to (C',\partial')$ such that for each $n$, $\phi_n\colon (C_n,f_n)\to (C'_n,f'_n)$ is an $\sL$-filtered linear map.  We entitle the morphisms of $\bCh(\bFVec(\sL))$ the {\em $\sL$-filtered chain maps}.


If $(C,f)$ is an $\sL$-filtered chain complex then $\partial_n(f_n(a))\subseteq f_{n-1}(a)$.
Thus $\{f_n(a)\}_{n\in \Z}$ together with $\{\partial_n|_{f_n(a)}\}_{n\in \Z}$ is a subcomplex of $C$.  We define the map $f\colon \sL\to \Sub(C)$ via
\begin{equation}\label{eqn:sc:morphism}
f(a):=\big(\{f_n(a)\},\{\partial_n|_{f_n(a)}\}\big)\in \Sub(C) .
\end{equation}
A chain complex equipped with a lattice homomorphism $\sL\to \Sub(C)$ is the object that Robbin and Salamon work with. The next two results show that these two perspectives are equivalent.  The proofs are immediate, and are included for completeness.
\begin{prop}
If $(C,f)$ is an $\sL$-filtered chain complex then $f\colon \sL\to \Sub(C)$, given as in~\eqref{eqn:sc:morphism}, is a lattice morphism.
\end{prop}
\begin{proof}
Let $a,b\in \sL$.  Let $A=\{f_n(a)\}_{n\in \Z}$ and $B=\{f_n(b)\}_{n\in \Z}$.  Then
\begin{align*}
f(a)\vee f(b) &=   (A,\partial|_A)\vee (B,\partial|_B) = (A+B,\partial|_{A+B}) = f(a\vee b)\\
f(a)\wedge f(b) &=  (A,\partial|_A)\wedge (B,\partial|_B) = (A\cap B,\partial|_{A+B})
= f(a\vee b)
\tag*{\qedhere}
\end{align*}
\end{proof}

\begin{prop}
Let $C=(C_\bullet,\partial_\bullet)$ be a chain complex together with a lattice morphism $f\colon \sL\to \Sub(C)$.  If $\{f_n\colon \sL\to \Sub(C_n)\}_{n\in \Z}$ is the family of maps defined as
\[
f_n(a) := A_n\subseteq C_n \quad\quad\text{where }  f(a) = (A_\bullet,\partial_\bullet^A),
\]
then $(C_\bullet,\partial_\bullet,f_\bullet)$ is an $\sL$-filtered chain complex.
\end{prop}
\begin{proof}
We show first show that $(C_\bullet,f_\bullet)$ is a family of $\sL$-filtered vector spaces spaces. Let $a,b\in\sL$.  Let $A_\bullet=f(a), B_\bullet=f(b)$ and $D_\bullet = f(a\vee b)$.  As $f$ is a lattice morphism, we have that
\[
A_\bullet \vee B_\bullet = f(a)\vee f(b) = f(a\vee b) = D_\bullet.
\]
This implies $A_n+B_n=D_n$ for all $n$.  It follows that $f_n(a)\vee f_n(b) = A_n+B_n=D_n=f_n(a\vee b)$.  Similarly, it follows that $f_n(a)\wedge f_n(b)=f_n(a\wedge b)$. Observe that $\partial_n$ is an $\sL$-filtered linear map for each $n$ because $f(a)\in \Sub(C)$ implies that  $\partial_n f_n(a)\subseteq f_{n-1}(a)$.  
\end{proof}






\subsection{The Subcategory of \Cyclic{} Objects}\label{sec:filt:cyclic}

There is a notion of \cyclic{} object in the category $\bChFL$.  Recall from Definition~\ref{def:prelims:joinirr} the notion of a join-irreducible element $a\in \sJ(\sL)$ and its unique predecessor $\pred a\in \sL$.  

\begin{defn}
{\em 
An $\sL$-filtered chain complex $(C,f)$ is \emph{\cyclic} if 
\[
\partial_n (f_n(a)) \subseteq f_{n-1}(\pred a), \quad \text{for all } a\in \sJ(\sL) \text{ and } n\in \Z.
\]
In this case $f$ is a \emph{\cyclic{} filtering}.  The \cyclic{} objects form a subcategory $\bChFLz\subset \bChFL$, called the {\em subcategory of \cyclic{} objects}.
}
\end{defn}


In \cite[Section 8]{robbin:salamon2} a connection matrix is defined to be an $\sO(\sP)$-filtered chain complex $(C,f)$ such that for any $b\in \sO(\sP)$ and $n\in \Z$
\[
\partial_n(f_n(b))\subset f_{n-1}(b\setminus \setof{p})
\]
whenever $p$ is maximal in $b$. 
 For any $b\in \sO(\sP)$, $p$ is maximal in $b$ if and only if $b$ covers $b\setminus{\setof{p}}$.  Therefore, the following result shows that our notion of a \cyclic{} $\sL$-filtered complex is equivalent to their definition of connection matrix.

\begin{prop}
Let $(C,f)$ be an $\sL$-filtered chain complex.  Then $(C,f)$ is \cyclic{} if and only if it obeys the following property: given $n\in \Z$ and $a,b\in \sL$ such that $b$ covers $a$ then $\partial_n(f_n(b))\subseteq f_{n-1}(a)$.
\end{prop}
\begin{proof}
The `if' direction is immediate: $a$ covers $\pred a$ for $a\in \sJ(\sL)$.  Thus $\partial_n(f_n(a))\subseteq f_{n-1}(\pred a)$.  Now suppose that $(C,f)$ is \cyclic{} and that $b$ covers $a$.  As $\sL$ is an object of $\bFDLat$, Proposition~\ref{prop:jr:rep} states that any $b\in \sL$ can be written as the irredundant join of join-irreducibles, i.e., we have $b=\vee_i q_i$ with $q_i\in \sJ(\sL)$.  Since $f$ is a lattice morphism,
\[
f_n(b) = f_n(\vee_i q_i) = \vee_i f_n(q_i) .
\] 
Moreover, since $b$ covers $a$ there is precisely one $q_j$ such that $a\vee q_j = b$ with $q_j\not\leq a$ and $q_i \leq a$ for $i\neq j$.  That $b$ covers $a$ implies that $\pred{q_j} \leq a$, otherwise $a < a\vee \pred{q_j} < b$. For any $x\in f_n(b)$ we have $x=\sum_i x_i$ with $x_i\in f_n(q_i)$ and $\partial_n(x) = \sum_i \partial_n (x_i)$.  Since $f$ is a \cyclic{} filtering, $\partial_n(x_i)\in f_{n-1}(\pred{q_i})$ and $\partial_n(x)\in \vee_i f_{n-1}(\pred{q_i}) = f_{n-1}(\vee_i \pred{q_i}) \subseteq f_{n-1}(a)$.
\end{proof}


\subsection{Equivalence of Categories}\label{subsec:filt:equiv}
  
We now examine the relationship between graded and filtered chain complexes.  Our primary aim is to establish an equivalence between these two categories, as well as their \cyclic{} subcategories.   With $\sL=\sO(\sP)$, it follows from Proposition~\ref{prop:prelims:induce} that the functor $\fL\colon \bGVec(\sP)\to \bFVec(\sL)$ (see Definition~\ref{defn:filt:funcL}) induces a functor 
\begin{equation}\label{eqn:functor:LCh}
\fL_\bCh \colon \bChG\to \bChFL.
\end{equation}
Recalling the definition of $\fL_\bCh$ from Section~\ref{sec:prelims:HA}, we have
\[
\fL_\bCh\big[(C,\pi)\big] := (C,f)
\]
where the $\sL$-filtering $f\colon \sL\to \Sub(C)$ is given by
\[
 \sL\ni a \mapsto (C^a_\bullet,\Delta^a_\bullet) = \frU\circ \frU^a\big[(C,\pi)\big] \in \Sub(C) \, .
\]
Here, $\frU\circ \frU^a$ is the forgetful functor $\frU\circ\frU^a\colon \bChG\to \bCh(\bVec)$, described in Definition~\ref{def:grad:forget}.  When the context is clear we abbreviate $\fL_\bCh$ by $\fL$.

\begin{thm}\label{thm:filt:equiv}
Let $\sL=\sO(\sP)$.  The functor $\fL\colon \bChG\to \bChFL$ is additive, fully faithful and essentially surjective (hence an equivalence of categories).  Moreover, $\fL$ restricts to an equivalence of the subcategories 
\[
\fL\colon \bChGz\to \bChFLz.
\]
\end{thm}
\begin{proof}
The first part follows from  from Theorem~\ref{thm:filt:vs:equiv} and Proposition~\ref{prop:HA:equiv}.  For the second part, suppose $(C,\pi)$ is a $\sP$-graded chain complex and $(C,f)$ is an $\sL$-filtered chain complex such that $\sL=\sO(\sP)$ and $\fL((C,\pi))=(C,f)$.  We show that $(C,\pi)$ is \cyclic{} if and only if $(C,f)$ is \cyclic{}.  We first show that if $(C,\pi)$ is \cyclic{} then $(C,f)$ is \cyclic{}. Let $a\in \sJ(\sO(\sP))$.  We show that $\partial(f(a))\subseteq f(\pred a)$.  By Birkhoff's theorem, there exists $s\in \sP$ such that $a=\downarrow\! s$ and $\pred a = \bigcup_{p<s} \downarrow \! p$. If $x\in C^a_\bullet$ then $x = \pi^a(x)$.  Hence
\[
\partial x = \partial(\pi^a (x)) = \sum_{p < q} \partial^{pq} \sum_{r\leq s} \pi^r (x) = \sum_{p< q \leq s} \partial^{pq}( x ) .
\]
Since $\partial^{pq} (x)\in C^p$ and $\pred a = \bigcup_{p < s}\downarrow p$ it follows that $\partial x\in f(\pred a)$ as desired.

We now show that if $(C,f)$ is \cyclic{} then $(C,\pi)$ is \cyclic{}.  Let $p\in \sP$ and $a$ denote $\downarrow\! p$.  Let $x\in C$.  Then $\pi^p(x)\in C^p\subset C^a$.  Since $(C,f)$ is \cyclic{} $\partial(\pi^p(x))\in \partial(C^a)\subset C^{\pred a}$.  That $p\not\in \pred a$ implies $\partial^{pp} = \pi^p\partial (\pi^p (x))) = 0$. Therefore $\fL$ restricts to an equivalence of the strict subcategories.
\end{proof}

\subsection{Filtered Cell Complexes}

We consider again the data analysis perspective, and define the appropriate concept for cell complexes.  Recall from Section~\ref{sec:prelims:cell} that the notion of subcomplex for a cell complex is more general than for a chain complex.  Given a cell complex $\cX=(\cX,\leq,\kappa,\dim)$ we work with $\Sub_{Cl}(\cX)$ -- the lattice of closed subcomplexes (see Definition~\ref{defn:latclsub}).

\begin{defn}
{\em
An {\em $\sL$-filtered cell complex} is a cell complex $\cX=(\cX,\leq,\kappa,\dim)$ together with a lattice morphism $f\colon \sL\to \Sub_{Cl}(\cX)$.  The morphism $f$ is called an {\em $\sL$-filtering} of $\cX$.  We write $(\cX,f)$ to denote an $\sL$-filtered cell complex.
}
\end{defn}

\begin{defn}\label{defn:filt:assoc}
{\em
Let $(\cX,\nu)$ be a $\sP$-graded cell complex and $\sL=\sO(\sP)$.  The \emph{associated $\sL$-filtered chain complex} is the pair $(C(\cX),f^\nu)$ where $f^\nu$ is the composition 
$$\sL\xrightarrow{\sO(\nu)}\Sub_{Cl}(\cX)
\xrightarrow{\spans} \Sub(C(\cX)),$$ 
given explicitly by sending $a\in \sL$ to
\[\text{span}(\sO(\nu)(a))=
\setof{\sum_{i=0}^n \lambda_i \xi_i: n\in \N, \lambda_i\in \K, \xi_i\in \sO(\nu)(a)}
\in \Sub(C(\cX)).
\]
We write  $\cL\colon \Cell(\sP)\to \bChFL$ for the assignment $(\cX,\nu)\mapsto (C(\cX),f^\nu)$.

}
\end{defn}


The next result follows from an examination of the definitions of $\cC,\cL$ and $\fL$.

\begin{prop}\label{prop:filt:functor2}
Let $\sL=\sO(\sP)$.  The functor $\mathfrak{L}\colon \bChG\to \bChFL$ fits into the following commutative diagram with the assignments $\cC$ and $\cL$ (denoted by dashes arrows).
\begin{equation*}
\begin{tikzcd}
\text{\emph{Cell}}(\sP) \ar[dr,dashed,swap,"\cL"] \ar[r,"\cC",dashed] & \bChG \ar[d,"\fL"]\\
& \bChFL
\end{tikzcd}
\end{equation*}
\end{prop}

\subsection{Homotopy Category of Filtered Complexes}\label{sec:grad:hom}

Again we may follow Section~\ref{sec:prelims:HA} to introduce the homotopy category $\bKF$ of the category of $\sL$-filtered chain complexes $\bCh(\bFVec(\sL))$.  To spell this out a bit further, we say that two $\sL$-filtered chain maps $\phi,\psi\colon (C,f) \to (D,g)$ are \emph{$\sL$-filtered chain homotopic} if there is an $\sL$-filtered chain contraction $\gamma\colon C\to D$ such that $\phi-\psi=\gamma\circ\partial+\partial\circ\gamma$.  We denote this by $\psi\sim_\sL \phi$. 


Proceeding as in Section~\ref{sec:prelims:HA}, the {\em homotopy category of $\sL$-filtered chain complexes}, which we denote by $\bKF$, is the category whose objects are $\sL$-filtered chain complexes and whose morphisms are $\sL$-filtered chain homotopy equivalence classes of $\sL$-filtered chain maps.  It follows from Proposition~\ref{prop:prelims:induce} that the functor $\fL$ induces a functor  on the homotopy categories $\fL_\bK\colon \bKG\to \bKF$.  This  functor is defined on objects as
$\fL_\bK ((C,\pi)) = \fL_\bCh(C,\pi)$
and on morphisms as $\fL_\bK([\phi]_\sP) = [\phi]_\sL$. Moreover, this functor satisfies the identity
\[
\fL_\bK\circ q = q\circ \fL_\bCh.
\]

\begin{defn}
 {\em
 Let $\bKFz$ denote the full subcategory of $\bKF$ whose objects are the objects of $\bChFLz$. Then 
 \[
 \bKFz=\bChFLz/\! \sim_\sL,
 \]
 and there is a quotient functor $q\colon\bChFLz \to \bKFz$.
 }
 \end{defn}

\begin{prop}\label{prop:filt:hequiv}
Let $\sL=\sO(\sP)$.  The functors $\fL_\bK\colon \bKG\to \bKF$ and $\fL_\bK\colon \bKGz\to \bKFz$ are equivalences of categories.
\end{prop}
\begin{proof}
This follows from Theorem~\ref{thm:filt:equiv} and the construction of the homotopy categories outlined in Section~\ref{sec:prelims:HA}.
\end{proof}

In analogy to Proposition~\ref{prop:grad:cmiso}, \cyclic{} $\sL$-filtered chain complexes which are $\sL$-filtered chain equivalent are isomorphic in $\bChFL$.  The may be phrased in terms of the following result.

 \begin{prop}\label{prop:filt:cmiso}
The functor $q\colon \bChFLz\to \bKFz$ is conservative.
 \end{prop}
 \begin{proof}
 By Birkhoff's Theorem there is some $\sP$ such that $\sL\cong \sO(\sP)$. Without loss of generality, we let $\fL$ be the composition (of equivalences of categories)
 \[
 \bChG\xrightarrow{\fL} \bChFO\to \bChFL.
 \]
 We have the following diagram.
 \[
\begin{tikzcd}
 \bChGz\ar[r,"\fL"] \ar[d,"q"]  & \bChFLz \ar[d,"q"]\\
 \bKGz\ar[r,"\fL_\bK"] & \bKFz
\end{tikzcd}
 \]
 Theorem~\ref{thm:filt:equiv} and Proposition~\ref{prop:filt:hequiv}  state, respectively, that $\fL$ and $\fL_\bK$ are equivalences of categories.  Moreover, Proposition~\ref{prop:grad:conserv} states that $q\colon \bChGz\to \bKGz$ is conservative.   It follows that $q\colon \bChFLz\to \bKFz$ is conservative.
 \end{proof}

\begin{cor}\label{cor:filt:cmiso}
Let $(C,f)$ and $(D,g)$ be \cyclic{} $\sL$-filtered chain complexes.   $(C,f)$ and $(D,g)$ are $\sL$-filtered chain isomorphic if and only if they are isomorphic in $\bChFL$.
 \end{cor}
 
 \begin{rem}
   Corollary~\ref{cor:filt:cmiso} implies that, up to isomorphism, the \cyclic{} $\sL$-filtered chain complexes are an invariant of the $\sL$-filtered chain equivalence class, cf.~\cite[Conjecture 7.4]{robbin:salamon2}.
 \end{rem}

Fix $\sP$ and $\sL=\sO(\sP)$.  Consider the following commutative diagram.
\begin{center}
\begin{equation}\label{dia:full}
\begin{tikzcd}[back line/.style={densely dotted}, row sep=1.5em, column sep=1.5em]
& \bChFLz \ar[leftarrow]{dl}[swap]{\fL} \ar[back line]{rr}{} \ar[densely dotted]{dd}[near start]{q} 
  & & \bChFL \ar[back line]{dd}[near start]{} \ar[leftarrow]{dl}[back line,swap]{\fL} \\
\bChGz \ar[crossing over,back line]{rr}[near end]{} \ar[densely dotted]{dd}[swap, near start]{q} 
  & & \bChG \\
& \bKFz  \ar[densely dashed]{rr}[near start]{}\ar[leftarrow,swap]{dl}{\fL} 
  & & \bKF \ar[leftarrow]{dl}[swap]{\fL} \\
\bKGz \ar[densely dashed]{rr}{} & & \bKG \ar[back line, crossing over, leftarrow]{uu}[swap, near end]{}
\end{tikzcd}
 \end{equation}
 \end{center}
 Our results thus far have the following implications.
\begin{itemize}
    \item It follows from Theorem~\ref{thm:filt:equiv} and Proposition~\ref{prop:filt:hequiv} that the solid arrows are equivalences of categories.
    \item Propositions~\ref{prop:grad:cmiso} and~\ref{prop:filt:cmiso} show that the quotient functors 
    \[
    q\colon \bChGz\to \bKGz,\quad q\colon \bChFLz \to \bKFz
    \]
     are conservative.
    \item Finally, in Section~\ref{sec:catform} we show how the proof of correctness of the algorithm \textsc{ConnectionMatrix} establishes that the dashed arrows -- the inclusion functors $\bKGz\to \bKG$ and $\bKFz\to \bKF$ -- are essentially surjective (and thus equivalences of categories); see Theorem~\ref{thm:inc:equiv} and Corollary~\ref{cor:inc:equiv}.
\end{itemize}

Finally, we reach our definition of Conley complex and connection matrix for an $\sL$-filtered chain complex.

\begin{defn}
{\em
Let $(C,f)$ be an $\sL$-filtered chain complex and $\sL=\sO(\sP)$.  A $\sP$-graded chain complex $(C',\pi)$ is called a \emph{Conley complex} for $(C,f)$ if
\begin{enumerate}
    \item $(C',\pi)$ is strict, i.e., $\partial_j'^{pp}=0$, for all $p$ and $j$, and
    \item $\fL(C',\pi)$ is isomorphic to $(C,f)$ in $\bKF$.
\end{enumerate}

}
\end{defn}

With the theory that has been built up, the following result is straightforward.

\begin{thm}\label{thm:filt:cm}
Let $(\cX,\nu)$ be a $\sP$-graded cell complex. Let $\sL=\sO(\sP)$.  If $(C',\pi)$ is a Conley complex for $(C(\cX),\pi^\nu)$ then $(C',\pi)$ is a Conley complex for $(C(\cX),f^\nu)$.
\end{thm}

\begin{proof}
Since $(C',\pi)$ is a Conley complex for $(C(\cX),\pi^\nu)$, by definition it is an object of $\bChGz$.  Moreover, by definition $q(C',\pi)\cong q(C(\cX),\pi^\nu)$.
It follows from~\eqref{dia:full} that
\[
q\circ \fL (C',\pi)=  \fL_\bK\circ q(C',\pi) \cong \fL_\bK\circ q(C(\cX),\pi^\nu) = q\circ\fL ( C(\cX),\pi^\nu) 
\]
It follows  from Proposition~\ref{prop:filt:functor2} that $
q\circ\fL (C(\cX),\pi^\nu) = q (C(\cX),f^\nu)$.  Therefore $q\circ \fL (C',\pi)\cong q(C(\cX),f^\nu)$.
\end{proof}

Conceptually, Theorem~\ref{thm:filt:cm}  implies that one may do homotopy-theoretic computations within the category $\bChG$ in order to compute the relevant objects of interest for $\bKF$.  At this point in the paper, we refer the reader back to the left hand side of Diagram~\eqref{dia:concept}.


\section{Categorical Connection Matrix Theory}\label{sec:catform}

We now return to Diagram~\eqref{dia:full} of Section~\ref{sec:grad:hom} and discuss a categorical setup for connection matrix theory.

\begin{thm}\label{thm:inc:equiv}
The inclusion functor $i\colon \bKGz\to \bKG$ is an equivalence of categories.
\end{thm}
\begin{proof}
By definition, the inclusion functor $i$ is faithful; since the subcategory $\bKGz$ is full,  $i$ also full.  It follows from Proposition~\ref{prop:cats:equiv} that it only remains to show that $i$ is essentially surjective.  Let $(C,\pi)$ be a $\sP$-graded chain complex and let $\cX$ be a basis for $C_\bullet$.  For $\xi\in \cX$, define $\dim(\xi)=n$ if $\xi\in C_n$ and $\nu(\xi)=p$ if $\xi\in C_\bullet^p$.  Define $\kappa(\xi,\xi')$ as the appropriate coefficient of $\partial$, i.e.,
\[
\partial(\xi) = \sum_{\xi'\in \cX} \kappa(\xi,\xi')\xi',
\]
and define the partial order $\leq$ to be the reflexive transitive closure of the relation
\[
\xi' < \xi \text{ if and only if } \kappa(\xi,\xi')\neq 0.
\]
Then $\cX=(\cX,\leq,\kappa,\dim)$ together with $\nu\colon (\cX,\leq)\to \sP$ is a $\sP$-graded cell complex.  Consider the \cyclic{} $\sP$-graded chain complex $(C(A),\pi^\mu)$ where  $(A,\Delta,\mu)=\textsc{ConnectionMatrix}(\cX,\nu,\partial)$. It follows from Theorem~\ref{thm:alg:cm} that this is an object of $\bKGz\subset \bKG$ and that it is a Conley complex for $(C,\pi)$. Therefore $(C(A),\pi^\mu)$ is isomorphic to $(C,\pi)$ in $\bKG$.  Therefore the inclusion functor $i$ is essentially surjective, which completes the proof.
\end{proof}

\begin{cor}\label{cor:inc:equiv}
The inclusion functor $\bKFz\subset \bKF$ is an equivalence of categories.
\end{cor}
\begin{proof}
This follows from an examination of the bottom square of~\eqref{dia:full}: three of the functors are categorical equivalences.
\end{proof}

\begin{cor}\label{cor:cat:func}
Let $\sL=\sO(\sP)$.  There exists inverse functors, which we call Conley functors, 
\begin{enumerate}
    \item $\mathfrak{F}\colon \bKG\to \bKGz$, and
    \item $\mathfrak{G}\colon \bKF\to \bKGz$,
\end{enumerate}
which take a $\sP$-graded chain complex or $\sL$-filtered chain complex to its Conley complex.
\end{cor}

\begin{rem}
Corollary~\ref{cor:cat:func} provides a functorial framework for connection matrix theory. Algorithm~\ref{alg:cm} (\textsc{ConnectionMatrix}) computes the functor $\mathfrak{F}$ on objects.  Applying $\mathfrak{F}$ on morphisms gives the {\em transition matrix}, see~\cite{atm,franzosa2017transition,reineck1988connecting}.  The implications of this will be explored in a future paper.
\end{rem}

\section{Franzosa's Connection Matrix Theory}\label{sec:CMT}

In this section we will review connection matrix theory as developed by R. Franzosa in the sequence of papers~\cite{fran2,fran3,fran} from the late 1980's. The primary result of this section is Theorem~\ref{thm:braid:cm}, which states that our notion of connection matrix agrees with that of Franzosa after composition with the functor $\frH\circ \fB$.

\subsection{The Categories of Braids}
It was Conley's observation~\cite{conley:cbms} that focusing on the attractors of a dynamical system provides a generalization of Smale's Spectral Decomposition~\cite[Theorem 6.2]{smale}.  There is a lattice structure to the attractors of a dynamical system~\cite{lsa,lsa2,robbin:salamon2} and one is often naturally led  to studying a finite sublattice of attractors $\sA$ and an associated sublattice of attracting blocks $\sN$ with $\omega\colon \sN\to \sA$; see~\cite{kkv,lsa,lsa2}. This setup is expressed in the diagram below.
\[
\begin{tikzcd}[column sep = 1.5cm]
\sN  \ar[d,two heads,"\omega"]  \ar[r,"\subset"]
&  \sABlockR(\varphi)  \ar[d,two heads,"\omega"]\\
\sA \ar[r,"\subset"]
& \sAtt(\varphi)
\end{tikzcd}
\]
A sublattice of attracting blocks is what Franzosa terms an {\em index filtration}~\cite{fran,fran2,fran3}.  However, as these sublattices are not necessarily totally ordered, we follow~\cite{lsa} and call this an {\em index lattice}.

In his work, Franzosa introduces the notion of a {\em chain complex braid} as a data structure to hold the singular chain complexes that arise out of the topological data within the index lattice.  The chain complex braid is organized by the poset of join-irreducibles of the index lattice.  Implicit in Franzosa's work is a description of a category for chain complex braids over a fixed poset $\sP$.  We now describe this category, which we label $\bChB$.   First we recall the notion of adjacent convex sets.

\begin{defn}
{\em
An ordered collection $(I_1,\ldots, I_N)$ of convex sets of $\sP$ is {\em adjacent} if
\begin{enumerate}
\item $I_1,\ldots,I_n$ are mutually disjoint;
\item $\bigcup_{i=1}^n I_i$ is a convex set in $\sP$;
\item For all $p,q\in\sP$, $p\in I_i, q\in I_j, i < j$ implies $q \nless p$.
\end{enumerate}
}
\end{defn}

We are primarily interested in adjacent pairs of convex sets $(I,J)$ and for simplicity write the union $I\cup J$ as $IJ$.  
We denote the set of convex sets as $I(\sP)$ and the set of adjacent pairs and triples of convex sets as $I_2(\sP)$ and $I_3(\sP)$.  
\begin{defn}
{\em 
A pair $(I,J)$ of convex sets is {\em incomparable} if $p$ and $q$ are incomparable for any $p\in I$ and $q\in J$.  This immediately implies that $(I,J)$ and $(J,I)$ are adjacent.
}
\end{defn}

\begin{defn}
{\em
Following~\cite{fran},  a sequence of chain complexes and chain maps $$C_1\xrightarrow{i} C_2 \xrightarrow{p} C_3$$
is {\em weakly exact} if $i$ is injective, $p\circ i = 0$ and $p\colon C_2/\img (i)\to C_3$ induces an isomorphism on homology.
}
\end{defn}

\begin{prop}[\cite{fran}, Proposition 2.2]
Let 
\[
C_1\xrightarrow{i} C_2\xrightarrow{p} C_3
\]
be a weakly exact sequence of chain complexes  and $\partial_i$ the boundary operator of $C_i$.  There exists a natural degree -1 homomorphism $\partial\colon H(C_3)\to H(C_1)$ such that
\begin{enumerate}
    \item if $[x]\in H(C_3)$ then $\partial([x])=[i^{-1}\partial_2 p^{-1}(x)]$, 
    \item $\ldots\to H(C_1)\xrightarrow{i} H(C_2)\xrightarrow{p} H(C_3)\xrightarrow{\partial} H(C_1) \to \ldots$ is exact.
\end{enumerate}
\end{prop}

\begin{defn}
{\em
A {\em chain complex braid} $\scC$ over $\sP$ is a collection of chain complexes and chain maps in $\bCh(\bVec)$ such that
\begin{enumerate}
\item for each $I\in I(\sP)$ there is a chain complex $(C(I),\Delta(I))$,
\item for each $(I,J)\in I_2(\sP)$ there are chain maps $$i(I,IJ)\colon C(I)\to C(IJ)\quad\quad\text{and}\quad\quad p(IJ,J)\colon C(IJ)\to C(J),$$ which satisfy
\begin{enumerate}
\item $C(I)\xrightarrow{i(I,IJ)} C(IJ)\xrightarrow{p(IJ,J)} C(J)$ is weakly exact,
\item if $I$ and $J$ are incomparable then $p(JI,I)i(I,IJ)= \id|_{C(I)}$,
\item if $(I,J,K)\in I_3(P)$ then the following braid diagram commutes.
\[
\begin{tikzcd}[column sep=scriptsize,row sep=0.4 em]
& & C(J)\ar[dr] & &\\
& C(IJ)\ar[dr]\ar[ur] & & C(JK)\ar[dr] &\\
C(I)\ar[ur]\ar[rr] & & C(IJK)\ar[ur]\ar[rr] & & C(K)
\end{tikzcd}
\]
\end{enumerate}
\end{enumerate}
}
\end{defn}

\begin{defn}
{\em
The \emph{category of chain complex braids over $\sP$}, denoted $\bChB$, is the category whose objects are chain complex braids over $\sP$.  Given two chain complex braids $\mathscr{C}$ and $\scC'$ a morphism $\Psi\colon \scC\to \scC'$  is a collection of chain maps $\{\Psi(I)\colon C(I)\to C'(I)\}_{I\in I(\sP)}$  such that for $(I,J)\in I_2(\sP)$ the following diagram commutes.
\[
\begin{tikzcd}[row sep = large]
C(I)\ar[r]\ar[d,"\Psi(I)"] &C(IJ)\ar[d,"\Psi(IJ)"] \ar[r] & C(J)\ar[d,"\Psi(J)"]\\
C'(I) \ar[r] & C'(IJ) \ar[r] & C'(J)
\end{tikzcd}
\] 
}
\end{defn}


For a given sublattice of attractors $\sA$, two index lattices $\sN,\sN'$ associated with the same sublattice of attractors $\sA$, i.e., $\omega(\sN)=\sA$ and $\omega(\sN')=\sA$, may yield different chain complex braids.   However, the homology groups of the chain complexes contained in the chain complex braid are an invariant.  This is the motivation for {\em graded module braids}, which formalize the `homology' of a chain complex braid.  To match our terminology with Franzosa's~\cite{fran}, we define `graded module braids' in the generality of graded $R$-modules.  However, in this paper we specifically work in the case that $R$ is a field, and a graded $R$-module is a graded vector space.


\begin{defn}
{\em
Let $R$ be a ring.  A \emph{graded $R$-module} is a family $M_\bullet=\{M_n\}_{n\in \Z}$ of $R$-modules.  A \emph{graded $R$-module homomorphism} is a family $f\colon M_\bullet\to M_\bullet'$ are families of $R$-module homomorphisms $f=\{f_n\colon M_n\to M_n'\}_{n\in \Z}$. 
The category of \emph{graded $R$-modules}, denoted $\bRMod^\Z$, is the category whose objects are graded $R$-modules and whose morphisms are graded $R$-module homomorphisms. Let $M_\bullet$ and $M'_\bullet$ be graded $R$-modules.  A \emph{degree $d$ map} $\gamma$ from $M_\bullet$ to $M'_\bullet$ is a family of $R$-module homomorphisms $\{\gamma_n\colon M_n\to M'_{n+d}\}_{n\in \Z}$.
}
\end{defn}

\begin{defn}
{\em
A {\em graded $R$-module braid over $\sP$} $\scG$ is a collection of graded $R$-modules maps satisfying:
\begin{enumerate}
\item for each $I\in I(\sP)$ there is a graded $R$-module $G(I)$;
\item for each $(I,J)\in I_2(\sP)$ there are maps:
\begin{align*}
i(I,IJ)\colon G(I)\to G(IJ) \text{ of degree 0,}\\
p(IJ,J)\colon G(IJ)\to G(J) \text{ of degree 0,}\\
\partial(J,I)\colon G(J)\to G(I) \text{ of degree -1}
\end{align*}
which satisfy
\begin{enumerate}
\item $\ldots  \to G(I) \xrightarrow{i}  G(IJ)\xrightarrow{p} G(J) \xrightarrow{\partial} G(I) \to \ldots$ is exact,
\item if $I$ and $J$ are incomparable then $p(JI,I)i(I,IJ)= \id|_{G(I)}$
\item if $(I,J,K)\in I_3(P)$ then the braid diagram~\eqref{dia:braid} commutes.
\begin{equation}\label{dia:braid}
\begin{tikzcd}[column sep=normal,row sep=small]
\vdots \ar[d]&& \cdots \ar[drr] \ar[dll]&& \vdots \ar[d]\\
G(I)\ar[drr,"i"]\ar[dd,bend right,"i",swap] &&&&      G(K)\ar[dll,"\partial",swap]\ar[dd,bend left,"\partial"]\\ 
&&      G(IJ)\ar[drr,"p"]\ar[dll,"i",swap]       &&  \\
G(IJK)\ar[dd,bend right,swap,"p"]\ar[drr,"p"]  &&      &&G(J)\ar[dd,bend left,"\partial"]\ar[dll,"i",swap]\\
&&      G(JK)\ar[dll,"p",swap]\ar[drr,"\partial"]   &&      \\
G(K)\ar[drr,"\partial"] \ar[dd,bend right,"\partial",swap] &&&&      G(I)\ar[dd,bend left,"i"]\ar[dll,"i",swap]  \\
&&      G(IJ)\ar[dll,"p",swap]\ar[drr,"i"]    &&      \\
G(J)\ar[d]\ar[drr]    &&      &&      \ar[dll]G(IJK)\ar[d]\\
\vdots &&   \cdots  &&  \vdots
\end{tikzcd}
\end{equation}
\end{enumerate}

\end{enumerate}
}
\end{defn}

\begin{defn}
{\em
A morphism $\Theta\colon \scG\to \scG'$ of graded $R$-module braids is a collection of graded $R$-module homomorphisms $\{\Theta(I)\colon G(I)\to G'(I)\}_{I\in I(\sP)}$ such that for each $(I,J)\in I_2(\sP)$ the following diagram commutes:
\[
\begin{tikzcd}[row sep=large]
\ldots\ar[r] & G(I)\ar[r,"i"]\ar[d,"\Theta(I)"] & G(IJ)\ar[r,"p"] \ar[d,"\Theta(IJ)"]& G(J)\ar[r,"\partial"]\ar[d,"\Theta(J)"] & G(I)\ar[r] \ar[d,"\Theta(I)"]& \ldots\\
\ldots\ar[r] & G'(I)\ar[r,"i'"] & G'(IJ)\ar[r,"p'"] & G'(J)\ar[r,"\partial'"] & G'(I)\ar[r] & \ldots
\end{tikzcd}
\]
}
\end{defn}

\begin{rem}
Since a morphism of braids $\Theta\colon \scG\to \scG'$ involves a fixed map $\Theta(I)$ for each convex set $I$,  there is a commutative diagram involving the two braid diagrams of~\eqref{dia:braid} and $\Theta$ for any $(I,J,K)\in I_3(\sP)$.  In fact, as remarked in~\cite{bar,mcr} one does not need to use graded $R$-module braids, but only a collection of long exact sequences given this definition of morphism.
\end{rem}

\begin{defn}
{\em
Given a fixed ring $R$, the \emph{category of graded $R$-module braids over $\sP$}, denoted by ${\bf GMB(\sP,\bRMod^\Z)}$, is the category of graded $R$-module braids and their morphisms.
}
\end{defn}

When $R$ may be understood from the context we refer to a graded $R$-module braid as a graded module braid. This terminology matches Franzosa~\cite{fran}. For the purposes of this paper, $R$ is a field and we work with $\bGMB$.  Implicit in~\cite[Proposition 2.7]{fran} is the description of a functor from $\mathfrak{H}\colon \bChB\to \bGMB$ which is the analogy of the homology functor.   

\begin{defn}
{\em 
A pair of chain complex braid morphisms $\Psi,\Phi\colon \scC\to \scC'$ are {\em $\sP$-braided homotopic} if there is a collection $\{\Gamma(I)\colon C(I)\to C'(I)\}_{I\in I(\sP)}$ of chain contractions such that for each $I$
\[
\Phi(I)-\Psi(I)=\Delta(I)\Gamma(I)+\Gamma(I)\Delta'(I).
\] 
The collection $\Gamma=\{\Gamma(I)\}_{I\in I(\sP)}$ is called a {\em $\sP$-braided chain homotopy}.   We write $\Psi\sim_\sP \Phi$ if $\Psi$ and $\Phi$ are $\sP$-braided homotopic. 
}
\end{defn}

\begin{prop}
The binary relation $\sim_\sP$ is a congruence relation on $\bChB$.
\end{prop} 

\begin{defn}
{\em
Let $\scC,\scC'$ are chain complex braids.  A morphism of chain complex braids $\Phi\colon \scC\to \scC'$ is a {\em $\sP$-braided chain equivalence} if there is a $\sP$-braided chain map $\Psi\colon \scC'\to \scC$ such that $\Psi\Phi\sim_\sP \id_\scC$ and $\Phi\Psi\sim_\sP \id_\scC'$.   The \emph{homotopy category of chain complex braids over $\sP$}, which we denote $\bKB$, is the category whose objects are chain complex braids over $\sP$ and whose morphisms are $\sP$-braided chain homotopy equivalence classes of chain complex braid morphisms.  In other words, $\bKB$ is the quotient category $\bChB/\! \sim_\sP$ formed by defining the hom-sets via
 \[
 \Hom_{\bKB}(\scC,\scC') = \Hom_{\bChB}(\scC,\scC')\slash \! \sim_\sP,
 \]
 where $\sim_\sP$ is the braided homotopy equivalence relation. Denote by $q\colon \bChB\to \bKB$ the quotient functor which sends each chain complex braid over $\sP$ to itself and each chain complex braid morphism to its $\sP$-braided chain homotopy equivalence class.  It follows from the construction that two chain complex braids $\scC,\scC'$ are isomorphic in $\bKB$ if and only if $\scC,\scC'$ are $\sP$-braided chain equivalent.
}
\end{defn}

\begin{prop}\label{prop:braid:functor}
Let $\scC$ and $\scC'$ be chain complex braids over $\sP$.  If $\scC,\scC'$ are braided chain equivalent then $\mathfrak{H}(\scC)\cong \mathfrak{H}(\scC')$.  In particular, there is a functor
\[
\mathfrak{H}_\bK\colon \bKB \to \bGMB,
\]
 that sends braided chain equivalences to graded module braid isomorphisms.
\end{prop}


\subsection{Franzosa's Connection Matrix}

In this section we review Franzosa's definition of a connection matrix.
Let $\scC$ be a chain complex braid in $\bChB$.  
Historically, the connection matrix was introduced as a $\sP$-filtered (upper triangular) boundary operator $\Delta$ on the direct sum of homological Conley indices associated to the elements of $\sP$
\[
\Delta\colon \bigoplus_{p\in \sP} H_\bullet(C(p))\to  \bigoplus_{p\in \sP} H_\bullet(C(p))
\] 
that recovers the associated graded module braid $\mathfrak{H}(\scC)$.  See Definition~\ref{defn:cmt:cm} for the precise notion.   
$\Delta$ may be thought of as a matrix of linear maps $\{\Delta^{pq}\}$ and the identification with the matrix structure is the genesis of the phrase {\em connection matrix}.  

Recall that the functor $\fL$ of~\eqref{eqn:functor:LCh} is used to build an $\sO(\sP)$-filtered chain complex from $\sP$-graded complex.  The next results show that graded chain complexes can be used to build chain complex braids. First, recall the forgetful functor $u$, as well as the family of forgetful functors $\{\frU^I\}$ parameterized by the convex sets $I\in I(\sP)$ defined in Definition~\ref{def:grad:forget} (see also Section~\ref{sec:grad}).  For a $\sP$-graded chain complex $(C,\pi)$ and a convex set $I\subset \sP$, $\frU\circ \frU^I(C,\pi)=(C^I,\Delta^I)$ is a chain complex.  Given a $\sP$-filtered chain map $\phi\colon (C,\pi)\to (C',\pi)$, $\frU\circ \frU^I(\phi)$ is the chain map $\Phi^I=e^I\circ \phi\circ i^I\colon C^I\to C'^I$.

\begin{prop}[\cite{fran}, Proposition 3.4]\label{prop:fran:3.4}
Let $(C,\pi)$ be an $\sP$-graded chain complex.  The collection $\fB(C,\pi)$ consisting of the chain complexes $\{\frU\circ \frU^I(C,\pi)\}_{I\in (\sP)}$ and the natural chain maps $i(I,IJ)$ and $p(IJ,J)$ for each $(I,J)\in I_2(\sP)$ form a chain complex braid over $\sP$.  
\end{prop}

\begin{prop}[\cite{atm}, Proposition 3.2]\label{prop:UTMap}
Let $(C,\pi)$ and $(C',\pi)$ be $\sP$-graded chain complexes.  If $\phi\colon (C,\pi) \to (C',\pi)$ is a $\sP$-filtered chain map then the collection $\{\frU\circ \frU^I(\phi)=\Phi^I\colon C^I\to C'^I\}_{I\in I(\sP)}$ is a chain complex braid morphism from $\fB(C,\pi)$ to $\fB(C',\pi)$.
\end{prop}

Propositions~\ref{prop:fran:3.4} and~\ref{prop:UTMap} describe a functor 
\begin{equation}
    \fB\colon \bChG\to \bChB.
\end{equation} 
That is, the functor $\fB$ is defined on objects as $
\fB(C,\pi) = \{C^I,\Delta^I\}_{I\in I(\sP)}$ together with the natural inclusion and projection maps.  Moreover, $\fB$ is defined on morphisms as
\[
\fB(\phi) = \{\Phi^I\colon C^I\to C'^I\}_{I\in I(\sP)}.
\]

\begin{prop}\label{prop:cmt:functor}
The functor $\fB\colon \bChG\to \bChB$ is additive.  Moreover $\fB$ induces a functor on homotopy categories 
\[
\fB_{\bK}\colon \bKG\to \bKB.
\]
\end{prop}


We can now state Franzosa's definition of connection matrix.  In brief, this is a $\sP$-graded chain complex capable of reconstructing the appropriate graded module braid.

\begin{defn}[\cite{fran}, Definition 3.6]\label{defn:cmt:cm}
{\em
Let $\scG$ be a graded module braid over $\sP$ and $(C,\pi)$ be a $\sP$-graded chain complex.  The boundary operator $\partial$ of $(C,\pi)$ is called a {\em C-connection matrix for} $\scG$ if
\begin{enumerate}
    \item There is an isomorphism of graded module braids \begin{align}\label{eqn:defn:cm}
\mathfrak{H}\circ \fB(C,\pi)\cong \scG.
\end{align}
    \item  If, in addition, $(C,\pi)$ is \cyclic{}, i.e., $\partial^{pp}=0$ for all $p\in \sP$, then $\partial$ is a {\em connection matrix (in the sense of Franzosa) for} $\scG$.
\end{enumerate}
}
\end{defn}


In light of Definition~\ref{defn:cmt:cm}, the connection matrix is an efficient codification of data which is capable of recovering the braid $\scG$.  In Conley theory, a graded module braid $\scG$ over $\sJ(\sL)$ is derived from a index lattice.   In this way the connection matrix is a graded object (over $\sJ(\sL)$) capable of recovering (up to homology) the data of the index lattice.  Moreover, both the chain complex braid $\scC = \fB(C,\pi)$ and the graded module braid $\scG' = \mathfrak{H}\circ \fB(C,\pi)$ associated to the connection matrix are simple objects in their appropriate categories.  Observe that 
\begin{enumerate}
    \item For $\scC$ we have $C(I) = \bigoplus_{p\in I} C^p$ for all $I\in I(\sP)$.  
    \item For the graded module braid $\scG'$ if $[\alpha] \in G(I)$ then $\partial(J,I)([\alpha]) = [\Delta^{J,I}(\alpha)]$ from~\cite[Proposition 3.5]{fran}. 
\end{enumerate}
  
The next result is one of Franzosa's theorems on existence of connection matrices, written in our terminology.  As Franzosa works with $R$-modules, instead of vector spaces as we do, he assumes the chain complexes consist of free $R$-modules.
 
\begin{thm}[\cite{fran}, Theorem 4.8]
Let $\scC$ be a chain complex braid over $\sP$.  Let $\{B^p\}_{p\in \sP}$ be a collection of free chain complexes such that $H(B^p) \cong H(C(p))$ and set $B=\bigoplus_{p\in \sP} B^p$.  There exists a $\sP$-filtered boundary operator $\Delta$ so that  $(B,\pi)$, where $\pi= \{\pi^p\colon B\to B^p\}_{p\in \sP}$, is a $\sP$-graded chain complex.  Moreover, there exists a morphism of chain complex braids  $\Psi\colon \scB  \to \scC$ where $\scB= \fB(B,\pi)$ such that $\mathfrak{H}(\Psi)$ is a graded module braid isomorphism.
\end{thm}

Here is a simple application of Franzosa's theorem.  Let $\scC$ be a chain complex braid.  Choose $B = \{C(p)\}_{p\in P}$.  The theorem says that there exists a $\sP$-graded chain complex $(B,\pi)$, where $B=\bigoplus_{p\in\sP} C(p)$, and a morphism of chain complex braids  $\Psi\colon \fB(B,\pi)\to \scC$ that induces an isomorphism on graded module braids.  Therefore for any chain complex braid over $\sP$ there is a simple representative -- one coming from a $\sP$-graded chain complex $(B,\pi)$ -- that is quasi-isomorphic to $\scC$ (in the sense that there is a morphism $\Psi$ of chain complex braids that induces an isomorphism on graded module braids).  In the case when one works with fields the homology $H(C(p))$ of each chain complex $C(p)$ is a $\Z$-graded vector space (see Definition~\ref{def:homology}).  Therefore we may choose $B=\{H_\bullet(C(p))\}_{p\in \sP}$.  Invoking the theorem gives a $\sP$-graded chain complex $(B,\pi)$ such that 
\[
\Delta\colon \bigoplus_{p\in \sP} H_\bullet(C(p))\to \bigoplus_{p\in \sP} H_\bullet(C(p)).
\] 
 In our terminology this implies that $(B,\pi)$ is a Conley complex and $\Delta$ is a connection matrix, both in the sense of our definition of connection matrix (Definition~\ref{def:grad:cm}) and of Definition~\ref{defn:cmt:cm} of Franzosa.

 \begin{rem}

The classical definition of the connection matrix (Definition~\ref{defn:cmt:cm}) does not involve a chain equivalence.  In particular, the connection matrix is not associated to a representative of a chain equivalence class.  In fact, in Franzosa's definition the isomorphism of~\eqref{eqn:defn:cm} is not required to be induced from a chain complex braid morphism.  
\end{rem}

\begin{defn}\label{defn:braids:assign}
{\em
Let $(\cX,\nu)$ be a $\sP$-graded cell complex.   The preimage of each convex set $\cX^I:=\nu^{-1}(I)$ is a convex set in $(\cX,\leq,\kappa,\dim)$.  Therefore each $\cX^I = (\cX^I,\leq^I,\kappa^I,\dim^I)$ is a cell complex where $(\leq^I,\kappa^I,\dim^I)$ are the restrictions to $\cX^I$.  This implies that each  $(C_\bullet(\cX^I),\partial|_{\cX^I})$ is a chain complex.  A routine computation shows that the collection $$\{(C_\bullet(\cX^I),\partial|_{\cX^I})\}_{I\in I(\sP)}$$ satisfies the axioms of a chain complex braid over $\sP$, and that this is precisely the image of $(\cX,\nu)$ under the composition $$\Cell(\sP)\xrightarrow{\cC} \bChG \xrightarrow{\fB} \bChB.$$
The composition defines an assignment $\cB\colon \Cell(\sP)\to \bChB$. 
}
\end{defn}

\begin{thm}\label{thm:braid:cm}
Let $(\cX,\nu)$ be a $\sP$-graded cell complex, $(C(\cX),\pi^\nu)$ be the associated $\sP$-graded chain complex and $\scG = \frH(\cB(\cX,\nu))$ be the associated graded module braid.  If $(C',\pi)$ is a Conley complex for $(C(\cX),\pi^\nu)$ then $\partial'$ is a connection matrix (in the sense of Franzosa,  Definition~\ref{defn:cmt:cm}) for $\scG$.
\end{thm}
\begin{proof}
By definition $(C,\pi)$ and $(C',\pi)$ are $\sP$-filtered chain equivalent.  It follows from Proposition~\ref{prop:cmt:functor} that the associated chain complex braids $\fB(C',\pi)$ and $\fB(C,\pi)$ are $\sP$-braided chain equivalent.   Then 
\[
\frH\circ \fB(C',\pi)\cong \frH\circ \fB(C,\pi)=\scG,
\]
where the first isomorphism follows from Proposition~\ref{prop:braid:functor} and the equality follows from the definition of $\cB$.
\end{proof}

Theorem~\ref{thm:braid:cm} implies that one may do homotopy-theoretic computations within the category $\bChG$ in order to compute connection matrices in the classical sense of Definition~\ref{defn:cmt:cm}.   At this point in the paper we refer the reader back to the full Diagram~\eqref{dia:concept}, which encapsulates much of the machinery introduced so far in Part II of the paper.  Most importantly, taken together Theorems~\ref{thm:filt:cm} and~\ref{thm:braid:cm} imply that if one finds a \cyclic{} $\sP$-graded chain complex $(C',\pi)$  that is $\sP$-filtered chain equivalent to the given $(C(\cX),\pi^\nu)$, then 
\begin{enumerate}
\item One can construct a \cyclic{} $\sL$-filtered chain complex, $\fL(C',\pi)$, which is chain equivalent to the associated lattice-filtered complex $(C(\cX),f^\nu)$. 
\item $\partial'$ is a connection matrix in the classical sense of Definition~\ref{defn:cmt:cm} for the associated graded-module braid $\mathfrak{H}(\cB(\cX,\nu)$.
\end{enumerate}

Therefore taken together Theorems~\ref{thm:braid:cm} and~\ref{thm:filt:cm}  imply that to compute connection matrices in both the sense of Franzosa~\cite{fran} and Robbin-Salamon~\cite{robbin:salamon2}, it suffices to find a Conley complex in $\bChGz$.

%
%


%


\section{Relationship to Persistent Homology}\label{sec:PH}

Persistent homology is a quantitative method within applied algebraic topology and the most popular tool of topological data analysis.  We give a brief outline, and refer the reader to~\cite{edelsbrunner:harer,oudot} and their references within for further details.  In this section we show that given an $\sL$-filtered chain complex, one can recover its persistent homology using a Conley complex and connection matrix; see Example~\ref{ex:PH}.  Persistent homology may be viewed as a family of functors, parameterized by pairs of elements $a,b\in \sL$ with $a\leq b$:
\[
\{PH_\bullet^{a,b}\colon \bChFL\to \bCh_0(\bVec)\}_{a\leq b}.
\] 
\begin{rem}
To be consistent with the literature of persistent homology, our notation is $PH_\bullet^{a,b}$ where $a\leq b$.  This is in contrast to our `matrix' notation that runs through the rest of the paper.
\end{rem}

Let $a,b\in \sL$ with $a\leq b$. $PH_\bullet^{a,b}(-)$ is defined on objects as follows. Let $(C,f)$ be an $\sL$-filtered chain complex (see Section~\ref{sec:filt:ch}).  There is an inclusion of subcomplexes 
\begin{equation}\label{eqn:pers:sub}
    \iota^{a,b}\colon f(a)\hookrightarrow f(b) .
\end{equation}
Recall from Section~\ref{sec:prelims:HA} that we view homology as a functor $H_\bullet\colon \bCh\to \bCh_0$.  Applying $H_\bullet$ to Eqn.~\eqref{eqn:pers:sub} yields a map $H_\bullet(\iota^{a,b})\colon H_\bullet(f(a))\to H_\bullet(f(b))$.  Then
\[
PH_\bullet^{a,b}\big[(C,f)\big] :=\img H_\bullet(\iota^{a,b})\in \bCh_0  .
\]

From this setup we can recover the standard persistence: for $j\in \Z$ the \emph{$j$-th persistent homology group of $a\leq b$} is the vector space
\[
PH_j^{a,b}\big[(C,f)\big] :=\img H_j(\iota^{a,b}).
\]
The \emph{$j$-th persistent Betti numbers} are the integers $$\beta_j^{a,b} = \dim\img(H_j(\iota^{a,b})).$$

 $PH_\bullet^{a,b}(-)$ is defined on morphisms as follows. Let $\phi\colon (C,f)\to (C',f')$ be an $\sL$-filtered chain map.  Since $\phi$ is $\sL$-filtered, $\phi$ restricts to chain maps $\phi^a\colon f(a)\to f'(a)$ and $\phi^b\colon f(b)\to f'(b)$, which fit into the following commutative diagram.
 \[
 \begin{tikzcd}[row sep = large]
 H_\bullet f(a)\ar[d,"H_\bullet(\phi^a)",swap] \ar[r,"H_\bullet(\iota^{a,b})"] & H_\bullet f(b) \ar[d,"H_\bullet(\phi^b)"] \\
 H_\bullet f'(a) \ar[r,"H_\bullet(\iota'^{a,b})"] & H_\bullet f'(b)
 \end{tikzcd}
 \]
As the diagram commutes, $H_\bullet(\phi^b)$ restricts to a map $H_\bullet(\phi^b) \colon \img H_\bullet(\iota^{a,b})\to \img H_\bullet(\iota'^{a,b})$, and
\[
PH_\bullet^{a,b}(\phi) := H_\bullet(\phi^b)\colon \img H_\bullet(\iota^{a,b})\to \img H_\bullet(\iota'^{a,b}).
\]

\begin{prop}\label{prop:PH:factor}
$PH_\bullet^{a,b}$ sends $\sL$-filtered chain equivalences to isomorphisms in $\bCh_0$.
\end{prop}
\begin{proof}

Let $a,b\in \sL$ with $a\leq b$. Let $\phi\colon (C,f)\to (C',f')$ be an $\sL$-filtered chain equivalence.  Since $\phi$ is $\sL$-filtered, $\phi^a$ and $\phi^b$ are chain equivalences.  Proposition~\ref{prop:prelim:chiso} implies $H(\phi^a)$ and $H(\phi^b)$ are isomorphisms.  Thus $PH_\bullet^{a,b}(\phi)$ is an isomorphism.
\end{proof}

\begin{thm}\label{thm:PH:iso}
Let $(C,f)$ be a $\sL$-filtered chain complex.  Let $(C',\pi)$ be a Conley complex for $(C,f)$.  Then for all $j\in \Z$ and $a\leq b$ in $\sL$ 
\[
PH_\bullet^{a,b}\circ\fL(C',\pi)\cong PH_\bullet^{a,b}\big[(C,f)\big].
\]
\end{thm}
\begin{proof}
It follows from Proposition~\ref{prop:PH:factor} and Proposition~\ref{prelims:cats:quotient} that $PH_\bullet^{a,b}$ factors as $PH_\bK^{a,b}\circ q$, giving the following commutative diagram.
\[
\begin{tikzcd}
\bChFL\ar[dr,"q",swap]  \ar[rr,"PH_\bullet^{a,b}"] &&\bCh_0(\bVec)\\
& \bKF \ar[ur,"PH^{a,b}_\bK",swap]
\end{tikzcd}
\]
Since $(C',\pi)$ is a Conley complex for $(C,f)$, by definition we have that
\[
q\big[(C,f)\big] \cong q(\fL(C',\pi)) .
\]
It follows that
\[
PH_\bullet^{a,b}(C,f) = PH^{a,b}_\bK\circ q (C,f) \cong 
PH^{a,b}_\bK\circ q\circ \fL(C',\pi) = PH_\bullet^{a,b}\circ \fL(C',\pi).
\tag*{\qedhere}
\]
\end{proof}

\begin{rem}
As a corollary, all computational tools that tabulate the persistent homology groups, such as the persistence diagrams and barcodes (see~\cite{edelsbrunner:harer,oudot}), can be computed from the Conley complex.
\end{rem}


Let $\cX$ be a finite cell complex and $\sL$ a finite distributive lattice.  Suppose that $\{\cX^a\subset \cX\mid a\in \sL\}$ be an isomorphic lattice of subcomplexes.  Defining $f$ by taking $a\in \sL$ to $\cX^a\subset \cX$ yields a lattice morphism $f\colon \sL\to \Sub_{Cl}(\cX)$. Therefore  $(\cX,f)$ is an $\sL$-filtered cell complex.  Recall $\cL$ as defined in Definition~\ref{defn:filt:assoc}.

\begin{defn}
{\em
Let $\sL$ be a finite distributive lattice and $(\cX,f)$ be an $\sL$-filtered cell complex.  The \emph{persistent homology of $(\cX,f)$} is defined to be 
\[
PH_\bullet^{a,b}(\cX,f):= PH_\bullet^{a,b}\circ \cL(\cX,f).
\]
}
\end{defn}


\begin{proof}[Proof of Theorem~\ref{thm:PH}]

Define $f\colon \sL\to \Sub_{Cl}(\cX)$ as the lattice morphism $f(a)=\cX^a$ for $a\in \sL$.  Then $(\cX,f)$ is an $\sL$-filtered subcomplex.  Set $$M_\bullet=\bigoplus_{a\in \sJ(\sL)} H_\bullet(\cX^a,\cX^{\pred a}),$$
and define $g\colon \sL\to \Sub(M)$ as the lattice morphism 
\[
g(a) = M^a= \bigoplus_{\setof{b\in \sJ(\sL)\mid b\leq a}} H_\bullet(\cX^b,\cX^{\pred b}).
\]
Then $(M,g)$ is an $\sL$-filtered subcomplex.  We wish to show that $(M,g)$ and $(\cX,f)$ have the isomorphic persistent homology. 
By hypothesis $(M,\Delta,\pi)$ (where $\pi$ are the natural projections) is a Conley complex for $\cL(\cX,f)$ and $\Delta$ is a connection matrix for $\cL(\cX,f)$.   Moreover, from the definition of $g$ we have that $(M,g)=\fL(M,\Delta,\pi)$. Let $a,b\in \sL$ with $a\leq b$. We have that
\[
 PH_\bullet^{a,b}(M,g) = 
PH_\bullet^{a,b}(\fL(M,\Delta,\pi)) \cong
PH_\bullet^{a,b}(\cL(\cX,f)) = PH_\bullet^{a,b}(\cX,f),
\]
where the isomorphism follows from Theorem~\ref{thm:PH:iso}.
\end{proof}

\section*{Acknowledgements}

The work of SH and KM was partially supported by grants NSF-DMS-1248071, 1521771 and DARPA contracts HR0011-16-2-0033 and FA8750-17-C-0054, and NIH grant R01 GM126555-01. The work of KS was partially supported by the NSF Graduate Research Fellowship Program under grant DGE-1842213 and by EPSRC grant EP/R018472/1.  KS would like to thank Chuck Weibel for some very useful discussions regarding homological algebra and connection matrix theory.

\bibliographystyle{abbrv}
\bibliography{ref}

\end{document}